%% file: main.tex
\pdfoutput=1
\newif\ifdraft \draftfalse 
\draftfalse
\documentclass[11pt]{article}
\usepackage[numbers]{natbib}
\usepackage{algorithm}
\usepackage[noend]{algpseudocode}

\usepackage{macros}
\usepackage{xspace}
\usepackage{xcolor}
\definecolor{darkblue}{rgb}{0,0,.5}
\usepackage[colorlinks=true,allcolors=darkblue]{hyperref}
\usepackage{bbm}
\usepackage{tikz}
\usetikzlibrary{matrix}
\usepackage{pgfplots}
\usepackage{overpic}
\usepackage{makecell}
\usepgfplotslibrary{fillbetween}
\usetikzlibrary{patterns}
\usepackage{authblk}
\usepackage{enumerate}
\usepackage{amssymb}
\usepackage{amsbsy}
\usepackage{amsmath}
\usepackage{amsthm}
\usepackage{amsfonts}

\usepackage{latexsym}
\usepackage{graphicx}
\usepackage{color}
\usepackage{xifthen}

\usepackage{latexsym}

\usepackage{subcaption}

\newcommand{\R}{\mathbb{R}}

\newtheorem{theorem}{Theorem}
\newtheorem{lemma}{Lemma}[section]

\newtheorem{definition}{Definition}[section]

\newtheorem{remark}{Remark}

\usepackage{fullpage}
\begin{document}

\title{Privacy Aware Experimentation over Sensitive Groups: A General Chi Square Approach}
\author{Rina Friedberg and Ryan Rogers\thanks{The authors would like to thank Parvez Ahammad, YinYin Yu, and Rahul Tandra for their support, and Adrian Cardoso and Kenneth Tay for thoughtful feedback and reviews. We also thank participants in the Fields Institute Workshop on Differential Privacy and Statistical Data Analysis for their feedback and discussions.}}
\affil{Data Science Applied Research, LinkedIn}

\maketitle 

\begin{abstract}
We study a new privacy model where users belong to certain sensitive groups and we would like to conduct statistical inference on whether there is significant differences in outcomes between the various groups.  In particular we do not consider the outcome of users to be sensitive, rather only the membership to certain groups.  This is in contrast to previous work that has considered locally private statistical tests, where outcomes and groups are jointly privatized, as well as private A/B testing where the groups are considered public (control and treatment groups) while the outcomes are privatized.  We cover several different settings of hypothesis tests after group membership has been privatized amongst the samples, including binary and real valued outcomes.  We adopt the generalized $\chi^2$ testing framework used in other works on hypothesis testing in different privacy models, which allows us to cover $Z$-tests, $\chi^2$ tests for independence, t-tests, and ANOVA tests with a single unified approach. When considering two groups, we derive confidence intervals for the true difference in means and show traditional approaches for computing confidence intervals miss the true difference when privacy is introduced.  For more than two groups, we consider several mechanisms for privatizing the group membership, showing that we can improve statistical power over the traditional tests that ignore the noise due to privacy. We also consider the application to private A/B testing to determine whether there is a significant change in the difference in means across sensitive groups between the control and treatment.  
\end{abstract}

\clearpage

\tableofcontents
\clearpage

\input{intro}

\input{prelims}

\input{propTest}

\input{independenceTest}

\input{tTest}

\input{ANOVA}

\input{ABTest}

\input{conclusion}

\clearpage

\bibliography{bib2}
\bibliographystyle{abbrvnat}

\end{document}

%% file: intro.tex
\section{Introduction}

\subsection{Background}

When measuring the impact that new products or enhancements have on users, we rely on A/B testing to help us determine if the new feature significantly improves the user experience.  To determine whether certain user groups are negatively impacted, although overall metrics might improve, we would like to measure the outcomes across these groups. 

However, to even run these statistical tests or form estimates, we need to know attributes of each data sample, such as what group they belong to.  When dealing with personal data, privacy techniques should be considered, especially when we are dealing with sensitive groups, e.g. race/ethnicity or gender.  There are several different privacy models to consider, with each model requiring different modifications to traditional tests.  We cover the various privacy models one might consider and why they fall short in providing privacy of the group that each sample belongs to or drastically impact utility for the task at hand.

One approach would be to keep a dataset of sensitive attributes of users in a secure environment that cannot be directly accessed and only allow certain analyses to be conducted, perhaps via a rigorous auditing procedure.  However, if multiple experiments are conducted on the secure dataset and only the aggregated results are revealed, it may eventually become possible to reconstruct the sensitive attributes.  The fundamental law of information recovery \cite{DinurNi03} states roughly that given enough (fairly) accurate results on a dataset, a non-trivial fraction of the data can be reconstructed.  Privacy mitigations, such as differential privacy from \citet{DworkMcNiSm06}, can be used to ensure each outcome for each experiment has a sufficient amount of noise to protect the privacy of each individual. 
However, note that many times the sensitive features of users are not changing, such as gender, race, or ethnicity data; so with hundreds of A/B tests ran every day, it is clear that the overall privacy loss becomes massive relatively quickly.  Hence the fundamental law of information recovery may be a threat even with noise added.  We could limit the number of experiments that can be conducted or add so much noise that the results are worthless, but these are not practical solutions. We want to ensure both usefulness of A/B testing and privacy of these features over a very long or even infinite time horizon.  

Another approach is to create a synthetic dataset subject to differential privacy, which can be created once and then any experiment would use this synthetic dataset.  Note that we would no longer need to worry about the privacy loss accumulating with each experiment, since each computation would be post-processing on the synthetic data, and privacy loss cannot be increased due to post-processing \cite{DworkMcNiSm06}.  There are many recent works on generating differentially private synthetic data.  In fact, there was a NIST competition to create the most accurate DP synthetic data.  As noted in this \href{https://www.nist.gov/blogs/cybersecurity-insights/differentially-private-synthetic-data}{NIST blog post}, 
the way to generate DP synthetic data is to build a probabilistic model to generate samples from a population in which the true data was sampled from.  Unfortunately, for A/B testing, we want to be able to perform a join on the data by some user ID so that we know which members were in the control or treatment groups and what their outcomes were.  Hence, we need the identifier for each user in the synthetic data to match the identifier used to split between these groups so that later experiments can be evaluated, which is not possible if we are generating synthetic data by drawing from an estimated population distribution.

We then turn to local differential privacy as a solution - this ensures that (a) we have a privatized dataset that can be used an unlimited number of times without compromising privacy of which group each user belongs to, and (b) allows us to join with another dataset by a unique identifier for each user.  In this case, each user will have its sensitive group go through a local DP mechanism, and the result is then stored and used in all subsequent analyses, leaving the user ID unchanged.  The dataset that has been privatized contains a user ID and its corresponding privatized feature, which will remain unchanged. As new experiments are conducted, we would like to determine whether there is a significant difference across the sensitive features of the users.  In this case, the outcomes of users are not privatized, but we will want to join them with the dataset containing the privatized features, in this case groups.  The question becomes, should the analyses conducted across the sensitive groups be modified to account for the privatization mechanism, or can we utilize the same workflows as if no privatization method was used?  

For our setting, we are not concerned with the privacy of the outcomes themselves, as these outcomes are traditionally used internally to determine the impact of different experiments.  Recent work from \cite{JuarezKo22} has also considered privatizing the group membership with local DP, but also privatize the outcome, as the outcome might be ``correlated with group membership."  This is precisely what we want to test.  Considering the privacy of only the features can also be thought of as a complement to some recent work on label-differential privacy where the labels themselves are the private information while the features are not treated as sensitive.  We point out that the privacy setting for label differential privacy is in the global model, whereas in this work we consider feature level privacy in the local setting.  We also mention that in the global model of privacy, there has been work on private statistical tests that consider privatizing outcomes, but \emph{not} the membership of the groups themselves \cite{UhlerSlFi13, YuFiSlUh14}.  In particular, the groups may simply be whether someone is in the control or treatment group, as in A/B testing, which should not contain sensitive information about the user.   The privacy model we consider here can be viewed as a variant to a traditional local DP set up, as local DP would consider privatizing the features, e.g. the group membership, and the outcome together.  

In this privacy model, which we call the \emph{local group DP} (LGDP) model, we will consider several scenarios with different privacy mechanisms.  First, we will consider binary outcomes, which can easily be extended to categorical outcomes.  The simplest setting here would be to test the difference between proportions in two sensitive groups, which $Z$-tests are traditionally used  as well as for calculating confidence intervals.   We then consider $\chi^2$- independence testing for determining whether there is a difference between proportions over multiple sensitive groups simultaneously.  As there are existing (fully) local DP tests for independence \cite{GaboardiRo18}, we compare the statistical power in our less restrictive setting.  

We then consider outcomes that are real valued.  In such cases, the local DP approach would bound outcomes within some interval and then add noise that is proportional to this bound, or discretize the outcomes to make it categorical.  Either approach introduces additional parameters that may bias the data and would be difficult to correct for without additional assumptions on the data distribution.  We consider both confidence intervals and hypothesis tests for the difference in means, which are traditionally based on the t-test, as well as differences in means across multiple sensitive groups simultaneously via one-way Analysis of Variation (ANOVA) tests.

Lastly, we will consider the application to A/B tests where users are split across two groups randomly and which group each user belongs to is known to the analyst and does not need privatization.  In this case, we test whether there is a significant change from the control to treatment group in the difference of outcomes between two sensitive groups. In particular, this can be used to determine whether a feature creates a smaller difference between means of sensitive groups than without the feature.

Although we cover multiple statistical tests in different settings, our approach follows the general $\chi^2$ framework, which was also used to design new classes of private tests in the global model of privacy from \citet{KiferRo17} and in the local DP setting from \citet{GaboardiRo18} for goodness of fit and independence testing with categorical outcomes.  We show that the traditional tests, which do not account for the privacy mechanisms, can also be unified with the same framework, with no loss of statistical power.  This framework is general enough to then apply to various privacy mechanisms over groups with binary and continuous outcomes.  We now summarize our contributions below:
\begin{itemize}
\item We present a new privacy model over sensitive groups called local group DP that is less restrictive than the local DP setting, which would also privatize the outcome for each sample.
\item We present a unified framework based on general $\chi^2$ tests that will allow us to easily derive new asymptotically valid statistical tests for various privacy mechanisms and hypothesis tests, including $Z$-tests, independence tests, t-tests, and ANOVA.
\item We demonstrate that confidence intervals with traditional methods that do not factor in the privacy mechanism lead to false conclusions, while our approach is able to compute empirically valid confidence intervals that converges to the original test as $\diffp \to \infty$.
\item We show that when testing whether there is a significant difference in means across more than 2 groups, our tests can achieve higher statistical power for the same level of privacy than traditional approaches.
\item We apply our general framework to A/B tests for testing whether the difference across sensitive groups has changed between the control and treatment groups, which is not considered to be sensitive.   
\end{itemize} 

\subsection{Related Work}

Airbnb recently described an approach that could be used to determine disparate impacts across sensitive groups in \href{https://news.airbnb.com/wp-content/uploads/sites/4/2020/06/Project-Lighthouse-Airbnb-2020-06-12.pdf}{Project Lighthouse} \cite{lighthouse2022} by collecting sensitive information, e.g. race and ethnicity, about users in a privacy centric way --- using $k$-anonymity and $\ell$-diversity to hide the identity of users.  The approach works by considering a dataset of users with an ID and a feature vector, in their case the number of reservations made and the number of reservations rejected.  The ID is then replaced with an \emph{anonymous ID} and the features are made less granular to ensure that there are at least $k$ users with the same features.  The anonymous IDs as well as profile pictures and names are given to an independent council, which labels each anonymous ID with a race/ethnicity.  This data is then joined with the $k$-anonymous features.  If a particular feature has all the same race/ethnicity, then $\ell$-diversity is used to change some of the groups.  The resulting dataset can then be used to determine whether rejection rate is independent of race/ethnicity.  It is not immediately clear how statistical tests might be modified to determine whether there is a significant difference between groups, as sometimes $k$-anonymous features might result in very wide ranges of values.  Furthermore, what if it is later determined that the feature vector was not rich enough, perhaps day of the week should also be included with the rejections, then the whole procedure needs to be redone.  Additionally, if some change had been implemented, we would like to know whether the rejection rate difference in groups had changed.  Unfortunately, $k$-anonymous groups do not satisfy any composition property, so doing this procedure again with different features and hence different groups might allow someone to be identified, with the race/ethnicity information being revealed.   

We now cover several additional related works, although we point out that these works do not consider the privacy model we do here.  There have been several works that have considered statistical hypothesis tests under local differential privacy, including simple tests like distribution testing or goodness of fit and composite tests like independence testing \cite{GaboardiRo18, Sheffet18, AcharyaCaFrTy19} as well as A/B testing \cite{DingNoLiAl18, Waudby-SmithWuRa22}.  We point out that in the local DP A/B testing work, all works consider a sample's membership to either the treatment or control group as not being private, but the outcome is, while in our setting the membership of the sensitive group within the control or treatment group is privatized.  Further, in the local setting, there has been work in mean estimation and confidence intervals \cite{GaboardiRoSh19, JosephKuMaWu19, Waudby-SmithWuRa22} as well as minimax optimal schemes \cite{DuchiJoWa18, BhowmickDuFrKaRo19}.  We also mention several works on hypothesis testing in the global model of DP, including simple tests like distribution testing or goodness of fit \cite{CaiDaKa17, CanonneKaMcSmUl19, CanonneKaMcUlZa20, AwanSl20} and more composite tests, like independence testing \cite{VuSl09, UhlerSlFi13, YuFiSlUh14, WangLeKi15, GaboardiLiRoVa16, KiferRo17, KakizakiFuSa17} as well as ANOVA  \cite{CampbellBrRiGr18, SwanbergGlGrRiGrBr19} and linear regression \cite{Sheffet17}.  Also in the global model there has been work on mean estimation and confidence intervals, including  \cite{KarwaVa17, WangKiLe19, BiswasDoKaUl20, CovingtonHeHoKa21}.  

In Table~\ref{table:refs} we show how previous work in differentially private inference can be organized, based on whether outcomes, groups, or both are privatized in either the local or global privacy settings.  For goodness of fit tests or mean estimation, samples are not considered to be in different groups, and for A/B testing whether a sample is in the treatment or control group is not sensitive.  Prior work on independence testing and linear regression have considered privatizing both the group/features as well as the outcomes.  To our knowledge, no prior work have considered privatization of only the group of each sample.  

\begin{table}[h!]
\centering
\begin{tabular}{|c | c | c|} 
 \hline
  & Local DP & Global DP \\ [0.5ex] 
 \hline
 Privatize Outcome & \shortstack{A/B Testing: \cite{DingNoLiAl18, Waudby-SmithWuRa22}\\ Mean Estimation: \cite{GaboardiRoSh19, JosephKuMaWu19}} &  \shortstack{GOF: \cite{CaiDaKa17, CanonneKaMcSmUl19, CanonneKaMcUlZa20, AwanSl20}\\ A/B Testing: \cite{UhlerSlFi13, YuFiSlUh14, MovahediCaHoKnLiLiSaSeTa21}\\ Mean Estimation: \cite{KarwaVa17, BiswasDoKaUl20, CovingtonHeHoKa21}}  \\ [0.5ex] 
 \hline
 \shortstack{Privatize Outcome \\ and Group/Features} &  \shortstack{Independence Testing: \cite{GaboardiRo18, Sheffet18, AcharyaCaFrTy19} }  &  \shortstack{Independence Testing: \cite{VuSl09, WangLeKi15, GaboardiLiRoVa16, KiferRo17, KakizakiFuSa17}\\ ANOVA: \cite{CampbellBrRiGr18, SwanbergGlGrRiGrBr19}\\ Linear Regression: \cite{Sheffet17, AlabiVa22, FerrandoWaSh22}}   \\ [0.5ex] 
 \hline
\end{tabular}
\caption{Prior Work on Differentially Private Inference.\label{table:refs}}
\label{table:1}
\end{table}

%% file: prelims.tex
\section{Privacy Preliminaries}\label{sect:privacyPrelim}

We start with the definition of local differential privacy \cite{Warner65, EvfimievskiGeSr03, KasiviswanathanLeNiRaSm11}, which will treat the group and the outcome of each group to be sensitive information and contrast that with the privacy model we consider here.

\begin{definition}
An algorithm $M: \cX \to \cY$ is $\epsilon$-locally differentially private if for all possible inputs $x,x' \in \cX$ and for all outcomes $S \subseteq \cY$ we have
\[
\Pr[M(x) \in S] \leq e^{\epsilon} \Pr[M(x')\in S].
\]
\end{definition}
With this definition, we then need to define what a possible input might entail.  A user's data $x$ will consist of two components $x=(j,o)$ where $j\in [g]$ comes from one of $g$ possible groups and $o \in \cO$ is some outcome.  We will consider both $O = \R$ and $O \in \{ 0,1\}$ being binary, which can then be extended to categorical.  The local DP definition then requires the output distribution of $M$ to not change by much if someone changes their group and their outcome.  We will instead focus on the case where a user's outcome is already known to the data analyst, but the group membership is not.
\begin{definition}
An algorithm $M: [g] \times \cO \to \cY$ is $\epsilon$-local group DP (LGDP) if $M'(\cdot, o): [g] \to \cY$ is $\epsilon$-local DP for all outcomes $o \in \cO$, i.e. $\epsilon$-local DP in its first argument.  
\end{definition}

This less restrictive definition of privacy ensures deniability for the group a user belongs to.  In particular, for attributes, like race and ethnicity, the information is incredibly sensitive and does not change over time, while outcomes, such as salary or conversion rate on current ad campaigns, can change over time.  Outcomes are also not tied to the groups that users belong to, or that is precisely what we want to test --- whether an outcome is independent of group membership or not.  Another way to interpret LGDP is that it is local DP on the group membership of each user, disregarding the outcomes that are joined with the user.  

We now present the fundamental local DP mechanisms for privatizing categorial inputs.  We assume that the sensitive data is a group $j \in [g]$.  We mainly focus on three mechanisms, where one will actually be a special case of another.  There have been several great works on improving run time and communication costs of these mechanisms \cite{AcharyaSuZh19, FeldmanNeNgHuTa22}, but we largely ignore these issues in this work and are primarily interested in the tradeoffs between statistical power and privacy, and their comparisons to their non-private counterparts.  Each of the following mechanisms are $\epsilon$-locally differentially private.  We will sometimes abuse notation and sometimes write the input to the mechanisms as a one-hot vector in $\{0,1 \}^g$ rather than an element in $[g]$.  

\begin{definition}[$g$-Randomized Response\footnote{\citet{Warner65} introduced randomized response for $g =2$ groups without the privacy parameter $\diffp$, but this generalized mechanism is typically attributed to this work.} from \citet{Warner65}]
The $g$-randomized response mechanism $M: [g] \to [g]$ returns its input with probability $\frac{e^\epsilon}{e^\epsilon + g-1}$ and otherwise uniformly selects a different outcome with equal probability, i.e. 
\[ 
\Pr[M(j) = j] = \frac{e^\epsilon}{e^\epsilon + g-1}, \qquad \text{ and } \Pr[M(j) = \ell] = \frac{1}{e^\epsilon + g-1} \ \forall \ell \neq j.
\]
\end{definition}  
Note that we will typically use the outcome space of randomized response as $\{0,1\}^g$, which we can easily map an outcome $j \in [g]$ to a vector in $\{0,1 \}^g$ by making position $j$ equal to 1 and all other coordinates are zero.

\begin{definition}[Bit Flipping\footnote{Also referred to as RAPPOR.} from \citet{ErlingssonPiKo14}]
The bit flipping mechanism $M: [g] \to \{0, 1\}^g$ with input $j$ creates a vector of length $g$ with all zeros except a one in position $j$, iterates through each coordinate, and flips the bit with probability $\frac{1}{e^{\epsilon/2} + 1}$ or otherwise keeps it the same. 
\end{definition}  

It is known that Bit Flipping typically provides more accurate results than $g$-Randomized Response for the same level of privacy parameter $\diffp$ in the high privacy regime, i.e. $\diffp < 1$, while $g$-Randomized Response provides more accurate results when $\diffp > \log(g)$.  We now present a local DP mechanism that can achieve better accuracy over broader privacy regimes.  Note that we will write vector quantities as $a = (a[1], \cdots, a[g[)$.

\begin{definition}[Subset Mechanism from \citet{YeBa17}]\label{defn:subset}
The subset mechanism with parameter $k < g$ is $M: [g] \to \cY_k$ where $\cY_k = \{ y \in \{0,1 \}^g : \sum_{i=1}^d y_i = k  \}$,  and we describe the iterative process for $M(j) = (M(j)[\ell] : \ell \in [g])$ where the probability the $j$-th coordinate $M(j)[j] =1$ is $\tfrac{ke^\diffp}{k e^\diffp + g - k}$ and otherwise $M(j)[j] = 0$.  If $M(j)[j] = 1$, then we uniformly select $k-1$ distinct coordinates from the set $[g]\setminus \{j\}$ to be 1 or if $M(j)[j] = 0$, then we uniformly select $k$ distinct coordinates from the set $[g]\setminus \{j\}$ to be 1, with the remaining coordinates set to 0.  
\end{definition}

The subset mechanism can be thought of as a way to unify the $g$-randomized response and the bit flipping mechanism.  In fact when $k = 1$, we recover the $g$-randomized response mechanism.   We will pick $k = \lceil \tfrac{g}{e^\epsilon +1} \rceil$, as Proposition III.2 in \citet{YeBa17} claims this to be the optimal choice for $k$.  We will then set $k = \lfloor \tfrac{g}{e^\epsilon + 1} \rfloor$ or $\lceil \tfrac{g}{e^\epsilon + 1}\rceil$, so we will set $\epsilon$ in order for $\tfrac{g}{e^\epsilon + 1} $ to be an integer.  Note that for $k = 1$, the subset mechanism becomes the $g$-randomized response mechanism.  

\citet{YeBa17} showed that the subset mechanism is optimal in terms of $\ell_2$-error for \emph{medium} privacy parameters, while $g$-randomized response is optimal for the high privacy loss regime ($\epsilon > \log(g)$) and the bit flipping mechanism is optimal in the low privacy loss regime ($\epsilon < 1$).  When $g = 2$, we only consider the randomized response mechanism while for $g > 2$ we provide tests for these three privacy mechanisms.

\section{General Chi Squared Theory}\label{sect:generalChi}

We now present the general minimum $\chi^2$ theory, which will be crucial to unifying our various statistical tests with a single framework.  We directly follow the presentation from \cite{Ferg96} (Chapter 23). We are not the first to apply the general $\chi^2$ theory for private hypothesis tests.  Previous work from \cite{KiferRo17} and \cite{GaboardiRo18} used the general $\chi^2$ tests to design valid hypothesis tests for categorical data in goodness of fit and independence testing.  In this work, we also show how the same general framework can be used to compute confidence intervals and develop tests for continuous outcomes.  We then present the same general framework to set up notation and so that our work is self contained.  We show that the classical tests can instead be analyzed as a general $\chi^2$ test, which will then allow us to design new valid private hypothesis tests for various privacy mechanisms in the LGDP model.   

Consider random vectors $\{ Y_i \in \R^d: i \in [n]\}$ where each $Y_i$ is selected i.i.d. from some distribution with parameters $\truetheta \in \Theta$ where  $\Theta$ is a non-empty open subset of $\R^\nu$, with $\nu < d$. The function $A$ maps a $\nu$-dimensional parameter $\theta\in \Theta$ into a $d$-dimensional vector.  We will denote coordinates of a vector similar to Python, so that for vector $\theta \in \R^d$, we denote $\theta[j]$ as the $j$th coordinate of $\theta$.

The null hypothesis $H_0$ can be written as there being a $\truetheta\in\Theta$ such that with $\bar{Y} \defeq \tfrac{1}{n} \sum_{i=1}^n Y_i$ we have the following,
\begin{equation}
\sqrt{n}\left( \bar{Y} - A(\truetheta) \right) \stackrel{D}{\to} N(0, C(\truetheta))
\label{eq:asym_norm}
\end{equation}
where $C(\theta) \in \R^{d\times d}$ is a covariance matrix and we write $\stackrel{D}{\to}$ to denote convergence in distribution as $n$ increases.  We measure the distance between $\bar{Y}$ and $A(\theta)$ with a test statistic given by the following quadratic form:
\begin{equation}
\QuadChi{n}(\theta) = n \left(\bar{Y}- A(\theta) \right)^\intercal \hspace{-0.3em}M(\theta) \left(\bar{Y} - A(\theta)\right)
\label{eq:quadratic}
\end{equation}
where $M(\theta) \in \R^{d\times d}$ is a symmetric positive-semidefinite matrix.  We make standard regularity assumptions about $A(\theta)$, $M(\theta)$, and $\Theta$, detailed as Assumption 4.1 in \cite{KiferRo17}. 

When $\truetheta$ is not known, as is the case in this work, we need to estimate a good parameter $\mintheta$ to plug into \eqref{eq:quadratic}. As was shown in \cite{KiferRo17}, we need only use a rough estimate of $\truetheta$ based on the data, call it $\phi(\bar{Y})$, then we can plug it into the middle matrix to get:
\begin{equation}
\estQuadChi{n}(\theta) =  n \left(\bar{Y} - A(\theta) \right)^\intercal M(\phi(\bar{Y}) ) \left( \bar{Y} - A(\theta)\right).
\label{eq:R}
\end{equation}
and then set our estimator $\mintheta=\arg\min_{\theta\in\Theta} \estQuadChi{n}(\theta)$. The test statistic becomes $\estQuadChi{n}(\mintheta)$ and the following theorem describe its asymptotic properties under the null hypothesis.

\begin{theorem}[\citet{KiferRo17}]
Let $d' \leq d$ be the rank of $C(\truetheta)$.
If regularity assumptions and \eqref{eq:asym_norm} hold, and, for all $\theta \in \Theta$,
\[
C(\theta)M(\theta)C(\theta) = C(\theta) \qquad \text{ and } \qquad 
 C(\theta)M(\theta)  \dot A(\theta) = \dot A(\theta)
\]
where $\dot A(\theta)$ denotes the partial derivatives of $A(\theta)$, then for $\estQuadChi{n}(\theta)$ given in \eqref{eq:R} where $\phi(\bar{Y})$ converges in probability to $\truetheta$ we have: 
\[
\min_{\theta \in \Theta \subseteq \R^\nu} \estQuadChi{n}(\theta)  \stackrel{D}{\to} \chi^2_{d'-\nu}.
\] 
\label{thm:ferg}
\end{theorem}

Note that the middle matrix $M(\theta)$ is typically selected to be the inverse of the covariance matrix $C(\theta)$, when it is invertible.  Otherwise, we use a generalized inverse for $C(\theta)$, as is done for the classical Pearson $\chi^2$ test for independence.  \citet{KiferRo17} also showed that when the privacy mechanism makes a singular covariance matrix when privacy is not enabled actually invertible, it is then possible to \emph{project out} the corresponding eigenvector whose eigenvalue would be zero without noise.  This effectively reduces the degrees of freedom of the resulting $\chi^2$ statistic by one and makes the tests more powerful. 

In Algorithm \ref{alg:general_chi_sq}, we outline our general approach to designing new hypothesis tests with $1-\alpha$ significance, which, given our assumptions hold, have asymptotically a $\chi^2$ distribution with the corresponding degrees of freedom.  

\begin{algorithm}
\caption{General $\chi^2$ approach}\label{alg:general_chi_sq}
\begin{itemize}
\item Based on the hypothesis test, compute a random vector $Y_i \in \R^d$ that are sampled i.i.d. from the unknown population distribution with parameter space $\Theta \subset \R^\nu$ where $\nu < d$.
\item Compute the covariance matrix $C(\theta)$ of $Y_i$ under the null hypothesis for general $\theta \in \Theta$.
\item Compute estimates $\hat{\theta}_n$ that converge in probability to $\truetheta$ and write $\hat{C} \defeq C(\hat{\theta}_n)$.
\item Calculate the inverse or general inverse of the covariance matrix, i.e. $C(\hat{\theta}_n)^\dagger$, where $C(\theta^*)$ has rank at most $d' \leq d$.
\item Compute the $\chi^2$ test statistic
\begin{equation}
\hat{D} = \min_{\theta \in \Theta} \left\{ n \left( \bar{Y} - \E[Y_i ; \theta] \right)^\intercal C(\hat{\theta}_n)^\dagger  \left( \bar{Y} - \E[Y_i ; \theta] \right)\right\}
\label{eq:chiSqStat}
\end{equation}
\item If $\hat{D} > \chi^2_{d' - \nu, 1- \alpha}$, then reject the null hypothesis.
\end{itemize}
\end{algorithm}

Hence, for the various tests, we will need to figure out what random vector $Y_i$ to use and the parameter space $\Theta$ to optimize over.

%% file: propTest.tex
\section{Difference in Proportions with Binary Outcomes}
We start with the case where outcomes are binary and we want to test whether there is a statistically significant difference between the success probabilities across sensitive groups.  There are traditionally two ways to test whether the success probabilities are different between two groups, either with a $Z$-test or with a $\chi^2$-test.  In fact, many built in proportion test packages in Python and R use $\chi^2$, rather than $Z$-tests, as the default to determine $p$-values.\footnote{In fact, the \texttt{prop.test} method in R actually computes the $\chi^2$ statistic, see \href{https://www.rdocumentation.org/packages/stats/versions/3.6.2/topics/prop.test}{RDocumentation}.}  One useful property of the $Z$-test, compared to the $\chi^2$ test is that it can readily be used to compute confidence intervals for the difference in proportions.  Furthermore, the difference between the $Z$ and $\chi^2$ tests comes from the experimental design.  For $Z$-test, we fix two sample sizes: $n_1$ for group $1$ and $n_2$ for group $2$.  While for $\chi^2$-tests, we sample $n$ data points where there is some probability $\pi$ of being in group $1$, and otherwise probability $1-\pi$ of being in group $2$.  It is the latter experimental design that we will be primarily interested in, since we will assume that we do not know which group each person belongs to initially.  However, we do analyze $Z$-tests to see if they can still be used to compute valid confidence intervals for the difference in proportions between two groups.

\subsection{$Z$-test}

We will assume that we have a binary variable and two groups.  We will have data $\{(G_i,X_i)\}_{i=1}^n$, where we first sample which group each $X_i$ belongs to, which we model with $G_i \sim \text{Bern}(\pi) + 1$ for an unknown $\pi \in [0,1]$, then we have $X_i | G_i \sim \text{Bern}(p_{G_i})$, where $p_g \in [0,1]$ is the probability of success for group $g \in \{1,2\}$.  One common test we may want to conduct is to test the null hypothesis $\texttt{H}_0: p_1 = p_2 + \Delta$.  We will write $N[1] = \sum_{i=1}^n \1{G_i=1}$, $N[2] = n - N[1]$, and $\bar{X}[g] = \tfrac{1}{N[g]}\sum_{i=1}^{N[g]} X_i \cdot\1{G_i = g}$ for $g \in \{ 1,2\}$. To carry out the difference in proportions test, we would form the following test statistic.\footnote{If $\Delta = 0$, i.e. no difference and equal variances in both groups, the test statistic can use the \emph{pooled} variance, which would result in the following term in the denominator of the test statistic: $\sqrt{\frac{N[1] \cdot \bar{X}[1] + N[2] \cdot \bar{X}[2]}{ N[1] + N[2]} \left(1/N[1] + 1/N[2] \right)}.$}
\[
T(\bar{X}[1], \bar{X}[2], N[1], N[2]; \Delta) = \frac{\bar{X}[1] - \bar{X}[2] - \Delta}{\sqrt{\bar{X}[1] ( 1 - \bar{X}[2])/N[1]+ \bar{X}[2] ( 1 - \bar{X}[2])/N[2] }}
\]
We then compare the test statistic with its asymptotic distribution under the null hypothesis, which is a standard normal.  That is for significance level $1-\alpha$, ($\alpha = 0.05$ or $0.01$ typically, but we will use $\alpha = 0.05$ throughout), we reject $\texttt{H}_0$ if 
\[
T(\bar{X}[1], \bar{X}[2], N[1], N[2]; \Delta)  \notin [ \Phi^{-1}(\alpha/2), \Phi^{-1}(1-\alpha/2)].
\]  

We now turn to privatizing the group membership with a mechanism $M$ so that our samples are $\{(M(G_i), X_i) \}_{i=1}^n$ and then analyze the resulting $Z$-test.  We now change the set up where each individual belonging to group $g \in \{1,2\}$ can be switched to the other group $3-g$ with some probability $\leq \tfrac{1}{2}$.  In particular, for $\diffp$-LGDP, we can use randomized response $M: \{1,2 \} \to \{ 1,2\}$ where $\Pr[M(g) = g] = \tfrac{e^\diffp}{1 + e^\diffp}$.  Hence the privatized data for individual $i$ in group $g$ is a mixture of two Bernoulli's, which we write as
\begin{equation}
X^\diffp_{i}| (G_i = g) =  \tfrac{e^\diffp}{1 + e^\diffp} \text{Bern}(p_g) +  \tfrac{1}{1 + e^\diffp} \text{Bern}(p_{3-g}) = \text{Bern}( \tfrac{e^\diffp}{1 + e^\diffp} p_g +  \tfrac{1}{1 + e^\diffp}p_{3-g}).
\label{eq:mixBern}
\end{equation}

Note that the sample sizes of each group changes due to some outcomes from group $g$ switching to group $3-g$.  We will write the new (randomized) sample sizes for each group $g$ as $N^\diffp[g]$.  Note that $n = N^\diffp[1] + N^\diffp[2]$.
Hence, the number of successful outcomes that we see in group $g$ is then
\[
\bar{X}^\diffp[g] = \sum_{i=1}^{N^{\diffp}[g]} X_{i}^\diffp \cdot \1{M(G_i) = g}.
\]

Note that in the special case when $\Delta = 0$, we have $\texttt{H}_0: p_0 = p_1 = p$, which in this case we would still have $X_i^\diffp[g] \sim \text{Bern}(p)$ where $p_1 = p_2 = p$.  Carrying out the standard test statistic would give us $T(\bar{X}^\diffp[1], \bar{X}^\diffp[2], N^\diffp[1], N^\diffp[2]; \Delta=0)$.  The main difference now is that the number of samples $N^\diffp[g]$ in each group $g$ is randomized.  We then check to see if conducting the original $Z$-test as if there were no privacy, still provides valid results.  When $\Delta \neq 0$, we compute the expected difference between the two proportions,
\[
\E\left[\tfrac{1}{N^\diffp[1]}\sum_{j =1}^{N^\diffp[1]} X^\diffp_{i}\cdot \1{M(G_i) = 1} - \tfrac{1}{N^\diffp[2]}\sum_{j =1}^{N^\diffp[2]}  X^\diffp_{i} \cdot \1{M(G_i) = 2} \right] 
\]
However, the expectation becomes much more complicated, due to the random number of trials $N^\diffp[g]$.  To approximate this expectation, we try treating $N^\diffp[g]$ as fixed for $g \in \{1,2 \}$, to get the following expression, which can then be used to \emph{correct} confidence intervals.  
\begin{equation}
 \Delta^\diffp =  \left(\tfrac{n \pi e^\diffp}{(1+e^\diffp)N^\diffp[1]} - \tfrac{n \pi }{(1+e^\diffp)N^\diffp[2]} \right) \Delta
\label{eq:DeltaEpsilon}
\end{equation}
Since $\pi$ is not known, we can estimate it in the following way to get an unbiased estimator for it:
\[
\hat\pi = \left( \frac{e^\diffp +1}{ e^\diffp -1 } \right) \left( \frac{N^\diffp[1] + N^\diffp[2]}{n} - \frac{1}{e^\diffp + 1}\right).
\]

We then consider the following $Z$-statistic, $T(\bar{X}^\diffp[1], \bar{X}^\diffp[2], N^\diffp[1], N^\diffp[2]; \Delta=\Delta^\diffp)$, to account for privacy.  We compute confidence intervals in Figure~\ref{fig:ZTestWithCorrection} with level of significance $1-\alpha = 95\%$. Observe that this correction gives us valid confidence intervals (top panel), and appropriate type I error rates (bottom panel), compared to the uncorrected statistic. 

\begin{center}
\begin{figure}
  \includegraphics[width=0.46\linewidth]{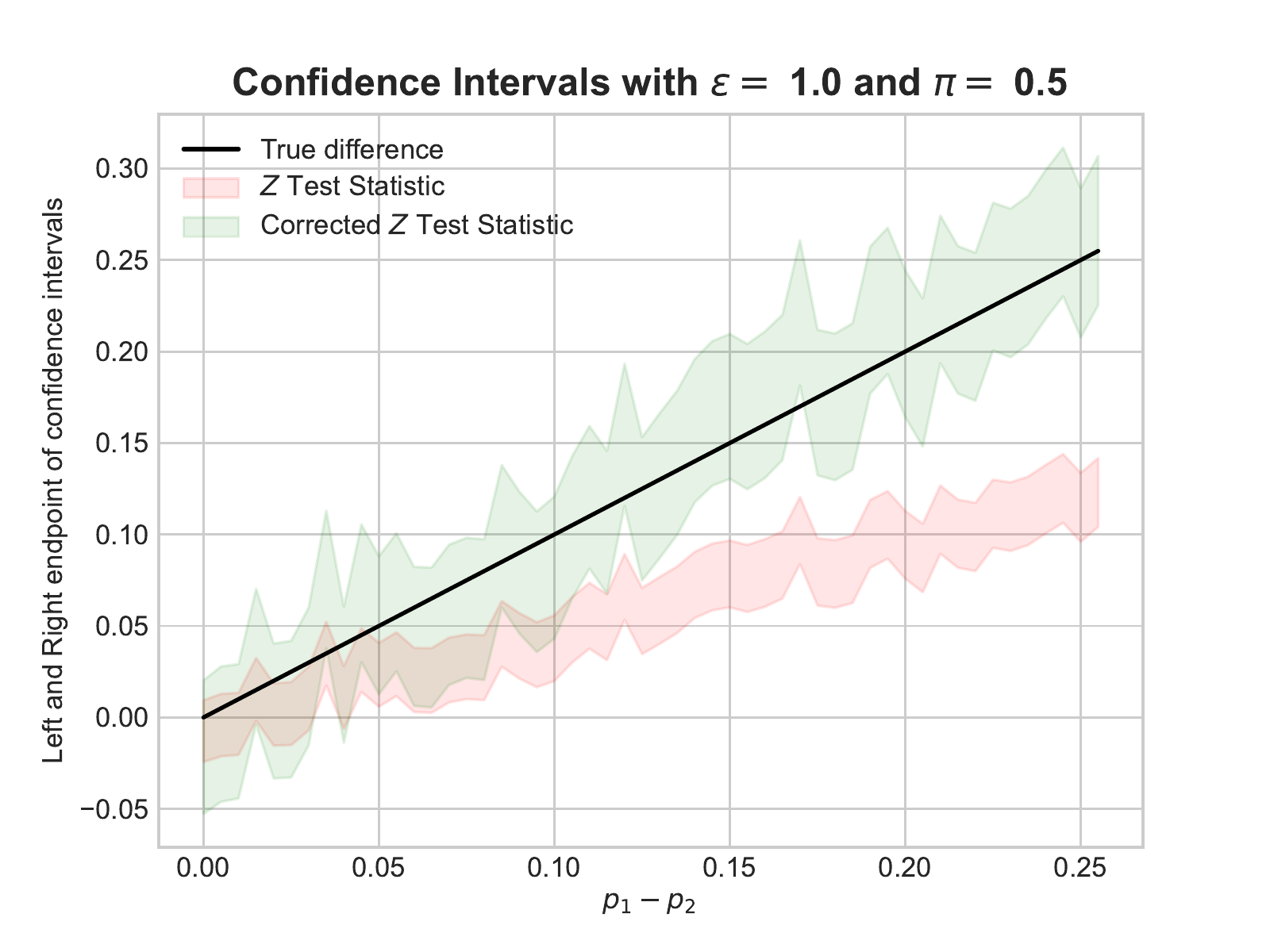}
  \includegraphics[width=0.46\linewidth]{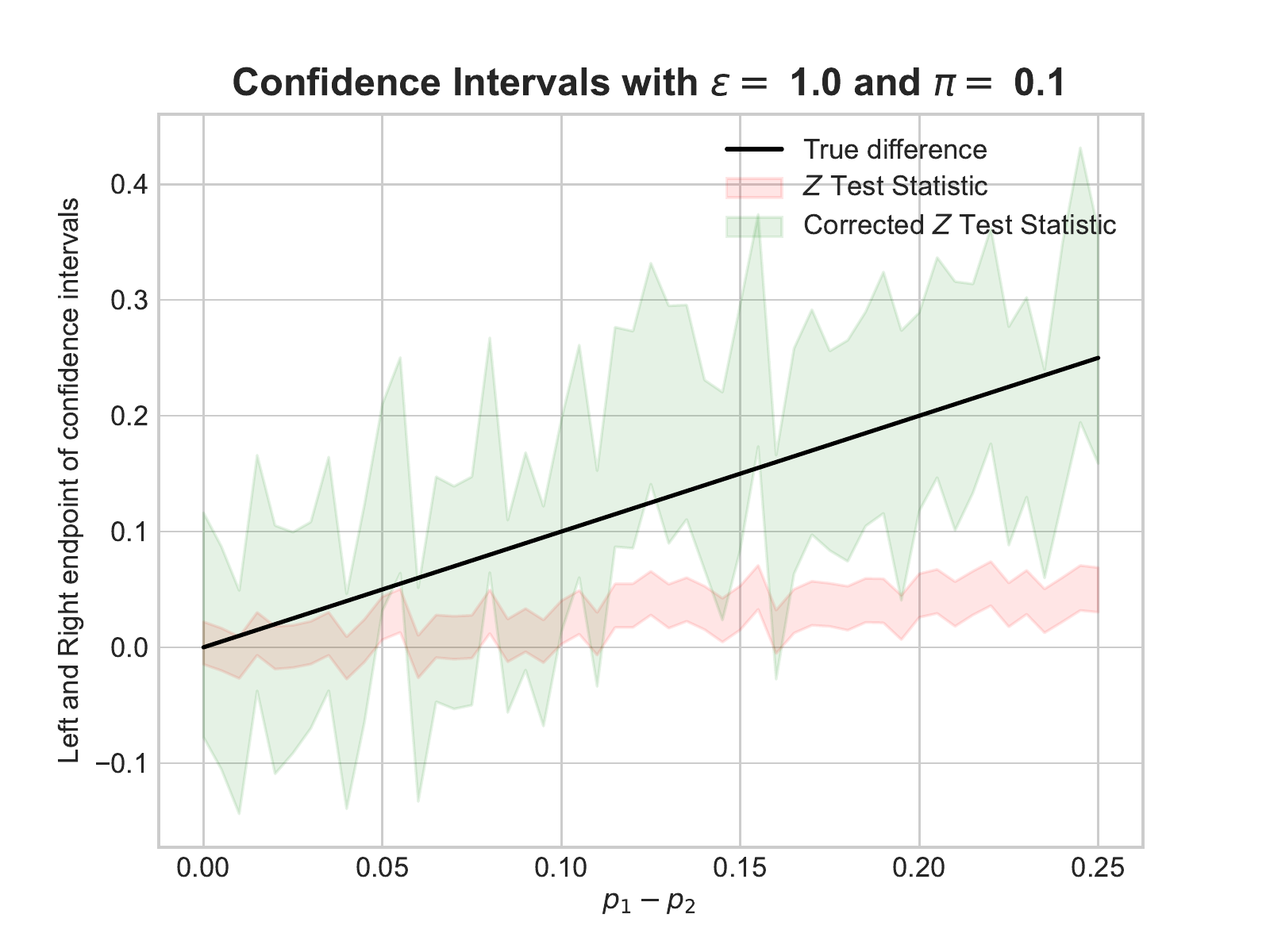}
    \includegraphics[width=0.46\linewidth]{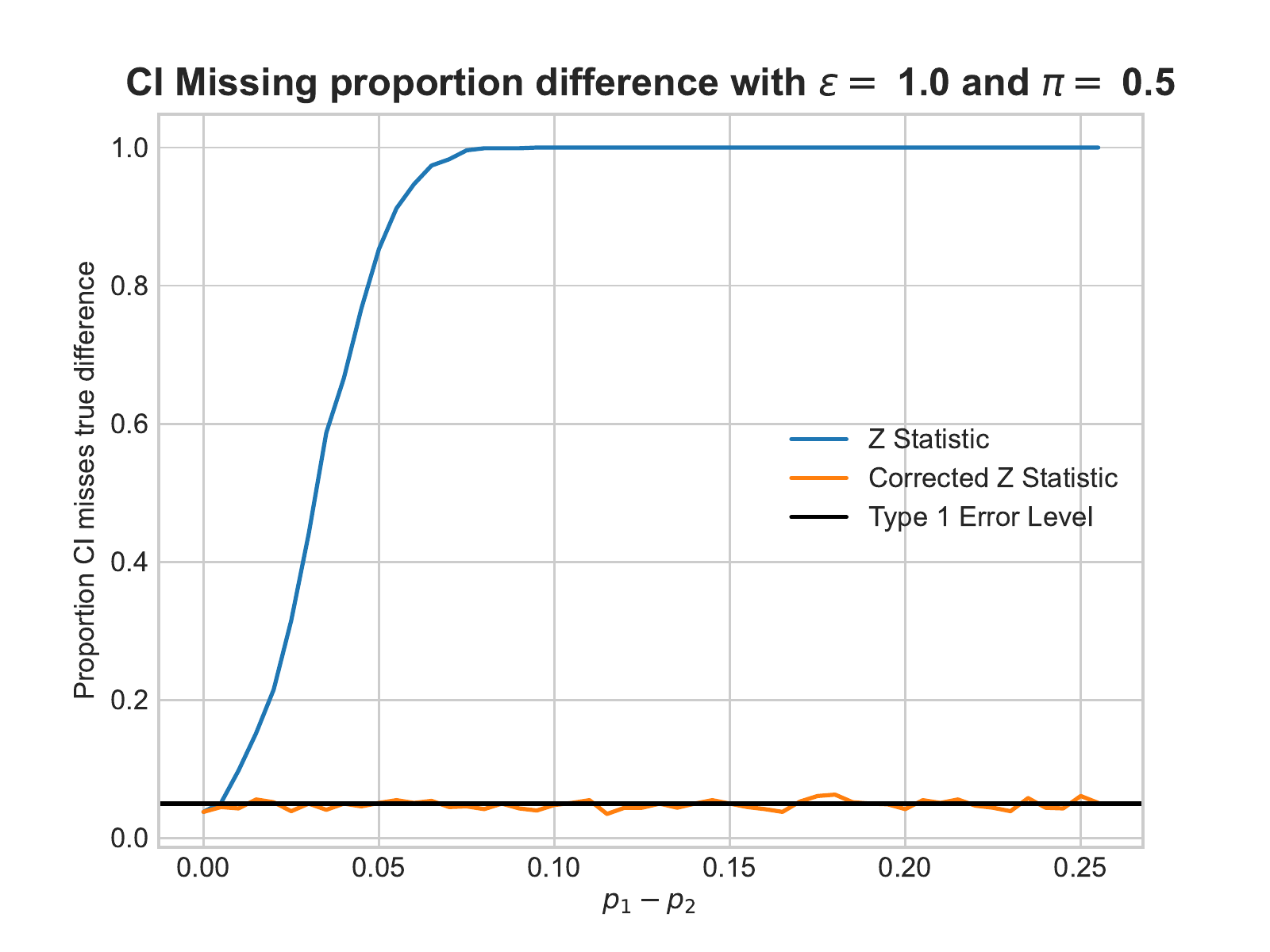}
      \includegraphics[width=0.46\linewidth]{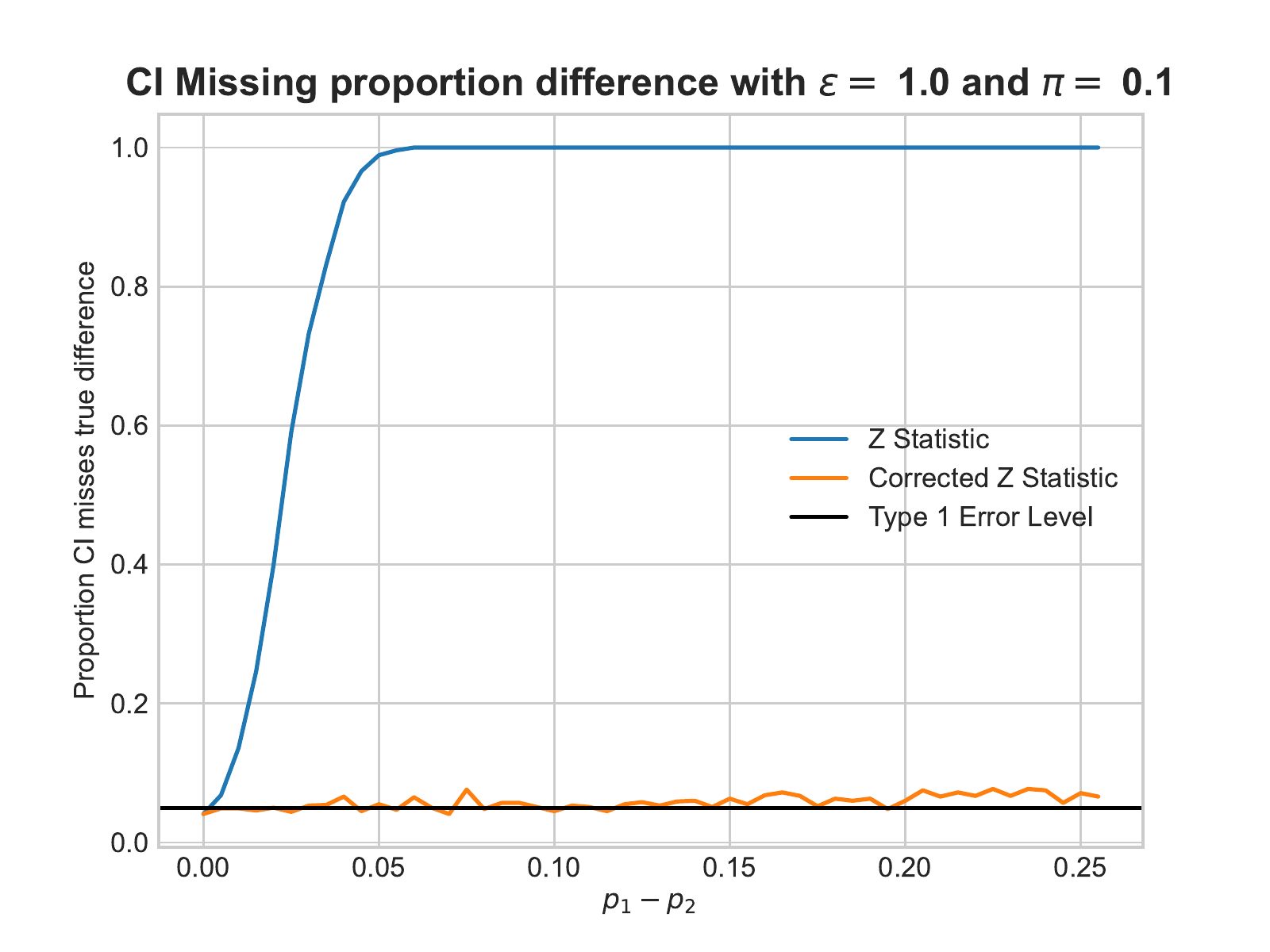}
  \caption{(Top Row) Confidence Intervals with $n = 10000$, $p_2 = 0.25$, and various $p_1 = p_2 + \Delta$ for $\diffp = 1.0$.  We show that correcting $\Delta$ to be $\Delta^\diffp$ in \eqref{eq:DeltaEpsilon} helps achieve valid confidence intervals for $\pi \in \{ 0.1, 0.5\}$. (Bottom Row) We use the same parameter settings as in the top row, but we compute the proportion of times the computed confidence interval actually overlaps with the true difference over 1000 independent trials. }
  \label{fig:ZTestWithCorrection}
\end{figure}
\end{center}

\subsection{$\chi^2$ test}
Given that many statistical packages use $\chi^2$ tests, rather than $Z$-tests, for testing the difference between two proportions, we then consider privatizing $\chi^2$-tests and further using these tests to compute valid confidence intervals.  There has been several works on privatizing $\chi^2$ tests, even in the more restrictive (fully) local DP setting \citep{GaboardiRo18, Sheffet18, AcharyaCaFrTy19}.  We then compare the relevant tests in the local DP setting with the tests in this less restrictive LGDP setting.  We will adopt the general minimum $\chi^2$ theory outlined earlier from \citet{KiferRo17}, which was also used to derive tests in the local DP setting by \citet{GaboardiRo18}.  

\subsubsection{Non-private test}
In Table~\ref{table:contingencyTable} we present the contingency table of outcomes, where Group $1$ has probability of success $p_1$ and Group $2$ has probability of success $p_2$.  We then want to test $\texttt{H}_0: p_1 = p_2 + \Delta$.     
\begin{table}[htbp]
\centering\setcellgapes{4pt}\makegapedcells
\begin{tabular}{ |c|c|c| } 
 \hline
 Outcome Probs & \shortstack{Group $1$ \\ w.p. $\pi$} & \shortstack{Group $2$ \\ w.p. $1-\pi$} \\ 
 \hline
Success & $Y[1,1]$ & $Y[1,2]$ \\ 
 \hline
Failure & $ Y[2,1]$ & $Y[2,2]$ \\ 
 \hline
\end{tabular}
\caption{Contingency Table giving outcome probabilities \label{table:contingencyTable}}
\end{table}

Hence, when we sample $n$ combined outcomes over groups $1$ and $2$, we can then consider a single multinomial random variable $Y = (Y[1,1], Y[1,2], Y[2,1], Y[2,2])^\intercal \sim \text{Multinom}(n, \vec{\theta}(\pi, p_1, p_2))$ where we flatten the outcome probabilities to an array
\[
\vec{\theta}(\pi, p_1, p_2) = \left( \underbrace{\pi p_1}_{\text{ Success in Group $1$}}, \underbrace{(1-\pi) p_2}_{\text{ Success in Group $2$}},  \underbrace{\pi (1-p_1)}_{\text{ Failure in Group $1$}}, \underbrace{(1-\pi)(1-p_2)}_{\text{ Failure in Group $2$}}\right)^\intercal.
\]

We use the $\chi^2$ statistic by providing an estimate for the parameters $p_1, p_2,$ and $\pi$, with $p_1 - p_2 = \Delta$ given.  Note that $Y$ can be written as the sum of i.i.d. random variables $\{ Y_i\}$, i.e. $Y = \sum_{i=1}Y_i$. The covariance matrix $C(\theta)$ for $Y_i$ is the following, where we write $\mathrm{Diag}(\theta)$ to be the diagonal matrix with entries on the diagonal as $\theta$.  
\begin{align}
&C(\pi, p_1. p_2) = \mathrm{Diag}\left( \vec{\theta}(\pi, p_1, p_2)\right) -  \vec{\theta}(\pi, p_1, p_2) \vec{\theta}(\pi, p_1, p_2)^\intercal
\label{eq:multinomCovar}
\end{align}
Note that the covariance matrix is singular and has the all one's vector in its null space.  It turns out that $\mathrm{Diag}\left(\vec{\theta}(\pi, p_1, p_2)\right)^{-1}$ is the generalized inverse for $C(\pi, p_1, p_2)$ \cite{Ferg96}, which we will use in our $\chi^2$ statistic.  We then use the estimates $\hat{p_1}, \hat{p_2}$ and $\hat{\pi}$ for $p_1, p_2,$ and $\pi$ respectively where
\begin{equation}
\hat{p}_2 = \frac{Y[1,1] + Y[1,2]}{n} - \Delta \hat{\pi}, \quad \hat{p}_1 = \hat{p}_2 + \Delta, \qquad \hat{\pi} = \frac{Y[1,1] + Y[2,1]}{n}
\label{eq:estimatePandPi}
\end{equation}

The $\chi^2$-statistic $\hat{D}$ then becomes the following,
\begin{equation}
\hat{D}(Y;\Delta) = n \cdot \left(Y/n - \vec{\theta}(\hat{\pi}, \hat{p}_1, \hat{p}_2)\right)^\intercal \mathrm{Diag}\left(\vec{\theta}(\hat{\pi}, \hat{p}_1, \hat{p}_2)\right)^{-1} \left( Y/n - \vec{\theta}(\hat{\pi}, \hat{p}_1, \hat{p}_2)\right)
\label{eq:chiSq}
\end{equation}
We then compare $\hat{D}(Y;\Delta)$ with a $\chi^2$ with $1$ degree of freedom, that is, if $\hat{D}(Y, \Delta) > \chi^2_{1,1-\alpha}$, then we reject $H_0: p_1 = p_2 + \Delta$ with significance level $1-\alpha$.  Note that this classical hypothesis test fits with the general $\chi^2$ test outlined in Section~\ref{sect:generalChi} as $Y$ actually has rank $3$, due to the all ones vector being in its null space,  and $\mathrm{Diag}\left(\vec{\theta}(\hat{\pi}, \hat{p}_1, \hat{p}_2)\right)^{-1}$ is the generalized inverse of the covariance matrix evaluated at the estimates given.  

One way to achieve valid confidence intervals for the difference $p_1 - p_2$ is to test for multiple values of $\Delta$ to see which intervals should be rejected under $\texttt{H}_0$.  That is, we search over the space $\Delta \in [-1,1]$, with some tolerance level $\tau$ (say $\tau = 0.001$), and check whether $\hat{D}(Y;\Delta) \leq \chi^2_{1,1-\alpha}$.  As we move from $\Delta = -1$, we will cross a point $\Delta = \Delta^L$ where $\hat{D}(Y; \Delta^L) > \chi^2_{1,1-\alpha}$, yet $\hat{D}(Y; \Delta^L+ \tau) \leq \chi^2_{1,1-\alpha}$.  This value $\Delta^L$ will be our left-end point of our confidence interval.  We then continue searching until we reach a point $\Delta = \Delta^R$ where  $\hat{D}(Y; \Delta^R- \tau) > \chi^2_{1,1-\alpha}$ yet $\hat{D}(Y; \Delta^R) > \chi^2_{1,1-\alpha}$.  This value $\Delta^R$ will be our right-end point of our confidence interval.  This simple grid search can also be replaced with a bisection root finding approach to the left and right of the $\Delta$ that minimizes the $\chi^2$ statistic.  We will use this method to compute confidence intervals with privatized groups in the following section and give very similar confidence intervals to the $Z$-test.

\subsubsection{Private Test}
We will then consider privatizing the group membership in the $\chi^2$ test for independence.   It will be helpful to write $Y = \sum_{i=1}^n Y_i$ in terms of each individual's data contribution.  Following the notation from above, we introduce random variables $W_i \sim \text{Bern}(\pi)$ and $X_i[g] \sim \text{Bern}(p_g)$ where
\[
Y_i = \left( W_i \cdot X_{i}[1], (1-W_i) \cdot X_{i}[2], W_i (1-X_{i}[1]), (1-W_i)(1-X_{i}[2]) \right)^\intercal.
\]

In order to privatize the group membership of each individual, we will use randomized response and privacy loss parameter $\diffp>0$.  There are other privatization mechanisms to consider, as we outline in Section~\ref{sect:privacyPrelim}, and we will consider these mechanisms when we increase the number of groups in the following section.  Recall that if a data point is in group $g \in \{1,2\}$, then it will remain in group $g$ with probability $\tfrac{e^{\diffp}}{1 + e^\diffp}$ and otherwise it will change to group $3-g$.  Let $Z^\diffp_{i}[g] \sim \text{Bern}\left( \tfrac{e^{\diffp}}{1 + e^\diffp}\right)$ where $i \in [n]$ and $g \in \{ 1,2 \}$, so we can write the privatized random vector as the following where we use $\mathbf{0}$ to be the $2\times 2$ matrix of zeros,
\begin{align*}
Y_i^\diffp & = 
\begin{pmatrix} Z^\diffp_{i}[1] \cdot W_i \cdot  X_{i}[1] + (1-Z^\diffp_{i}[2]) \cdot (1-W_i) \cdot X_{i}[2] \\  
			(1-Z^\diffp_{i}[1]) \cdot W_i \cdot  X_{i}[1] + Z^\diffp_{i,1}\cdot  (1-W_i) \cdot X_{i}[2] \\
			Z^\diffp_{i}[1] \cdot  W_i \cdot  (1-X_{i}[1]) + (1-Z^\diffp_{i,1}) \cdot (1-W_i)\cdot  (1-X_{i}[2]) \\
			(1-Z^\diffp_{i}[1]) \cdot W_i \cdot  (1-X_{i}[1]) + Z^\diffp_{i}[2]\cdot  (1-W_i)\cdot (1-X_{i}[2]) 
\end{pmatrix} \\
& =  \begin{bmatrix}
			Z_i^\diffp & \mathbf{0} \\
			\mathbf{0} & Z_i^\diffp
			\end{bmatrix} 
			Y_i, 		\quad \text{ where } Z_i^\diffp =  \begin{bmatrix} Z_i^\diffp[1] & 1 - Z_i^\diffp[2] \\ 1-Z_i^\diffp[1] & Z_i^\diffp[2] \end{bmatrix}.		
\end{align*}
simplifying to a block matrix. Observe that the first column of $Z_i^\diffp$ is exactly randomized response applied to group 1, while the second column is randomized response applied to group 2. 

In this case, we can again consider privatized data $Y^\diffp = (Y^\diffp[1,1], Y^\diffp[1,2], Y^\diffp[2,1], Y^\diffp[2,2])^\intercal = \sum_{i=1}^n Y_i^\diffp$ to be generated from $\text{Multinom}(n,\vec{\theta}^\diffp(\pi, p_1, p_2))$ where
\begin{align*}
\vec{\theta}^\diffp(\pi, p_1, p_2) = 
\begin{pmatrix}  
\tfrac{e^\diffp}{e^\diffp + 1} \left( \pi p_1\right) +  \tfrac{1}{e^\diffp + 1} \left( (1-\pi)p_2\right) \\
\tfrac{e^\diffp}{e^\diffp + 1} \left( (1-\pi) p_2 \right) +  \tfrac{1}{e^\diffp + 1} \left( \pi p_1\right)   \\
 \tfrac{e^\diffp}{e^\diffp + 1} \left( \pi (1-p_1)\right) +  \tfrac{1}{e^\diffp + 1} \left( (1-\pi)(1-p_2)\right) \\
\tfrac{e^\diffp}{e^\diffp + 1} \left( (1-\pi) (1-p_2) \right) +  \tfrac{1}{e^\diffp + 1} \left( \pi (1-p_1)\right)
\end{pmatrix}
\end{align*}
We then form similar estimates $\hat{\pi}$ and $\hat{p}_1, \hat{p}_2$ for $\pi$ and $p_1, p_2$ from \eqref{eq:estimatePandPi} under the null hypothesis $p_1 - p_2 = \Delta$ where instead we use outcomes $Y^\diffp$,
\begin{equation}
\hat{p}_2 = \frac{Y^\diffp[1,1] + Y^\diffp[2,1]}{n} -\hat{\pi} \Delta, \quad \hat{p}_1 = \hat{p}_2 + \Delta, \quad \hat{\pi} = \left(\frac{e^\diffp + 1}{e^\diffp - 1} \right) \left( \frac{Y[1,1]^\diffp + Y[1,2]^\diffp}{n} - \tfrac{1}{e^\diffp + 1} \right)
\label{eq:estimatePrivPandPi}
\end{equation}
We can then form the private $\chi^2$-statistic $\hat{D}^\diffp$ then becomes the following
\begin{equation}
\hat{D}^\diffp(\Delta) = n \cdot \min_{ \substack{p_1 - p_2 = \Delta \\\pi \in (0,1) } } \left\{ \left( Y^\diffp / n- \vec{\theta}^\diffp(\pi, p_1, p_2)\right)^\intercal \mathrm{Diag}\left(\vec{\theta}^\diffp( \hat{\pi}, \hat{p}_1, \hat{p}_2) \right)^{-1} \left(Y^\diffp/n - \vec{\theta}^\diffp(\pi, p_1, p_2)\right) \right\}
\label{eq:privChiSq}
\end{equation}

We can then plot power curves in Figure~\ref{fig:DiffProp1} for our the private test, testing with $\Delta = 0$ and increasing $p_1 - p_2 >0$ in the sample, and hence the test should reject $\texttt{H}_0$.  We then compare with the local DP $\chi^2$-test from \citet{GaboardiRo18}.  We show the number of trials that rejected $\texttt{H}_0: p_1 = p_2$ over $1000$ independent trials.  Note that we can drastically improve the statistical power of our tests for the same level of $\diffp$ in the less restrictive LGDP model compared with local DP tests.  

\begin{figure}
  \includegraphics[width=0.32\linewidth]{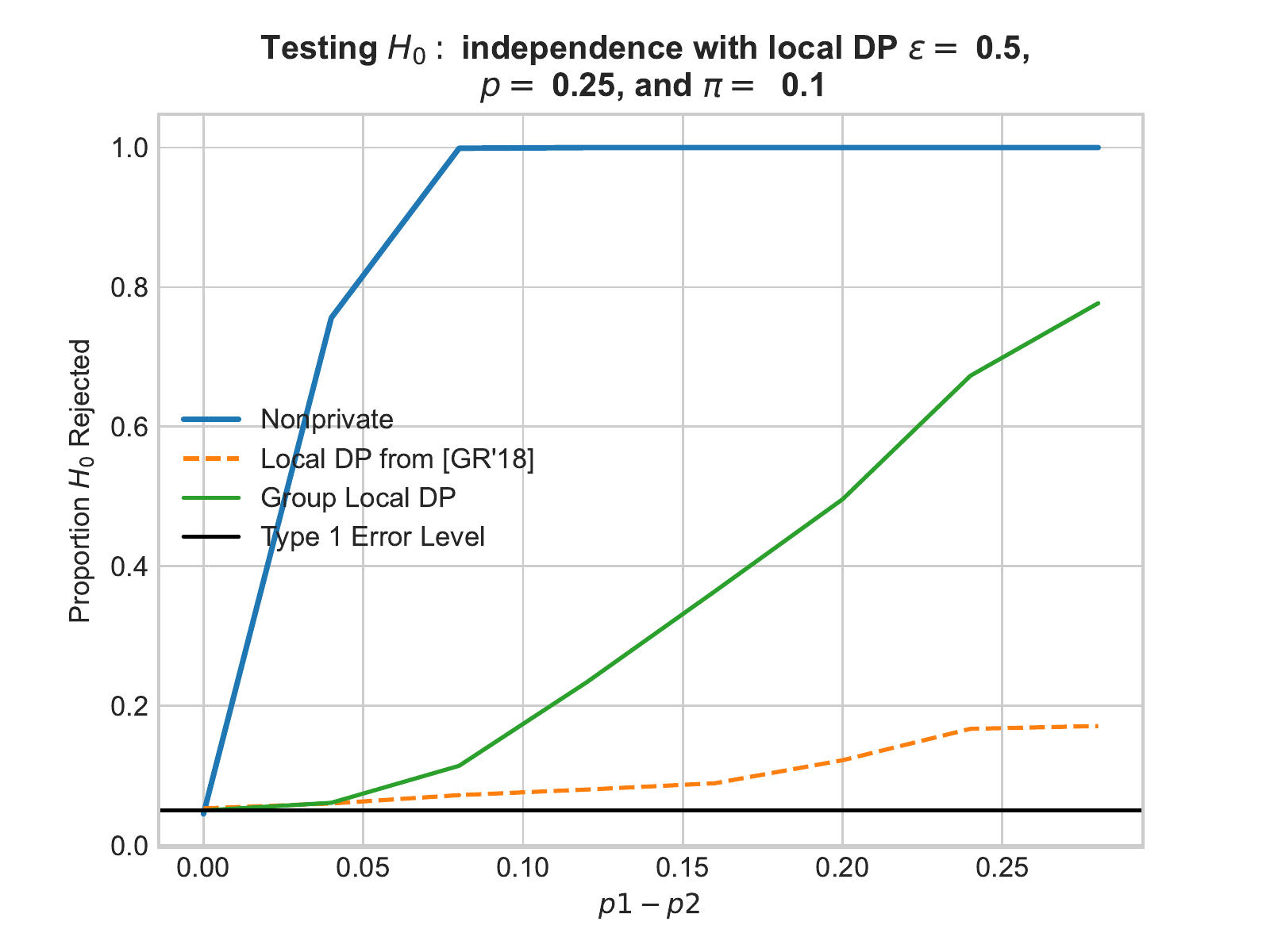}
  \includegraphics[width=0.32\linewidth]{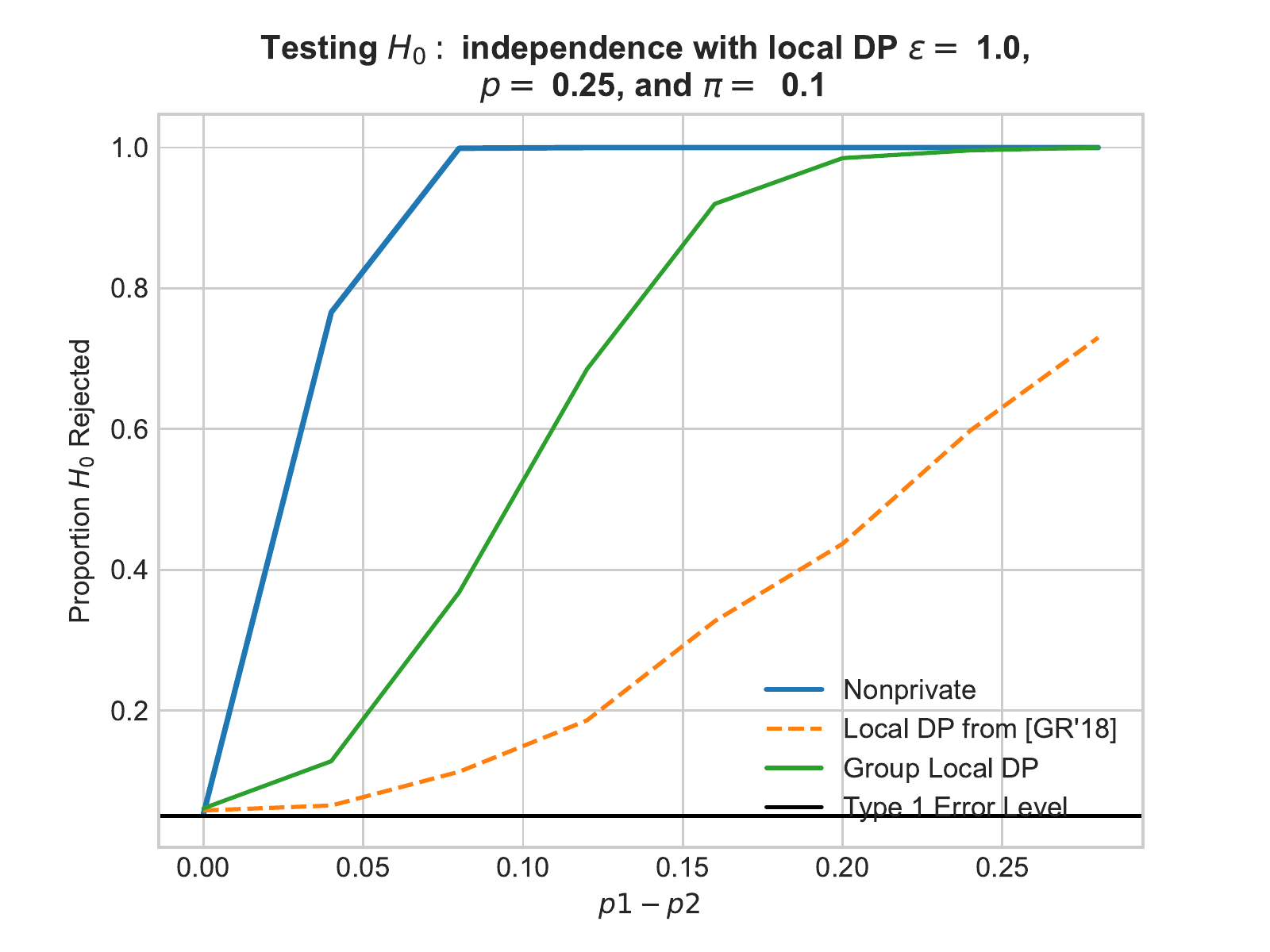}
  \includegraphics[width=0.32\linewidth]{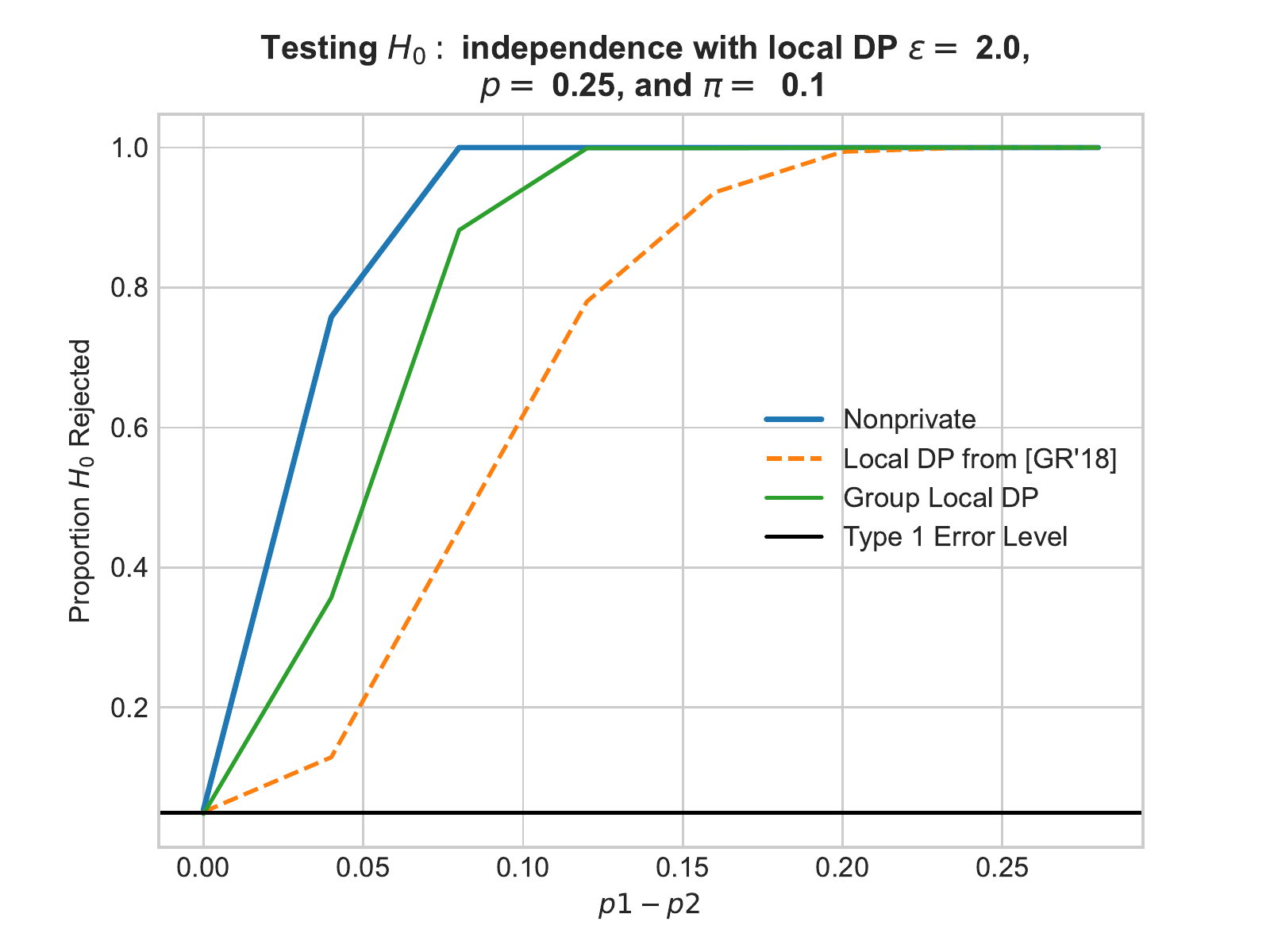}
  \caption{Comparing Group Local DP with (fully) LDP $\chi^2$ tests from \cite{GaboardiRo18} with the probability of being in Group $1$ being $\pi = 0.1$ and $\diffp \in \{0.5, 1,2 \}$ and $n = 10000$}
  \label{fig:DiffProp1}
\end{figure}

When testing for equal proportions across the 2 groups, we get $\vec{\theta}^\diffp(\pi, p, p) = \vec{\theta}(\pi^\diffp, p, p)$ where $\pi^\diffp = \tfrac{e^\diffp}{1+ e^\diffp} \pi + \tfrac{1}{1+e^\diffp} (1-\pi)$, which will result in essentially the same test as without privacy considerations, as the privatized vector will still be a multinomial just with a different group probability $\pi^\diffp$ as opposed to $\pi$.  Hence, we should reject the test under group probability $\pi$ if and only if we reject the test under group probability $\pi^\diffp$.  

\subsection{Confidence Intervals}\label{sect:CI_prop}

We have shown that the classical $Z$-test may not need modifying if we test $\texttt{H}_0: p_1 = p_2$, i.e., it still provides valid results. We can also use the $Z$-test with a correction on the difference $\Delta \to \Delta^\diffp$ from \eqref{eq:DeltaEpsilon} to compute valid confidence intervals, as in Figure~\ref{fig:DiffProp2}.  We also show that we can use the approach outlined above that uses the $\chi^2$-test to compute the end points of a confidence interval.  We then show the results in two cases, when $\pi = 1/2$ and when $\pi = 0.1$.  Note that there is not much difference between the confidence intervals using the $\chi^2$-test, and the $Z$-test with the correction factor. 

\begin{figure}
  \includegraphics[width=0.46\linewidth]{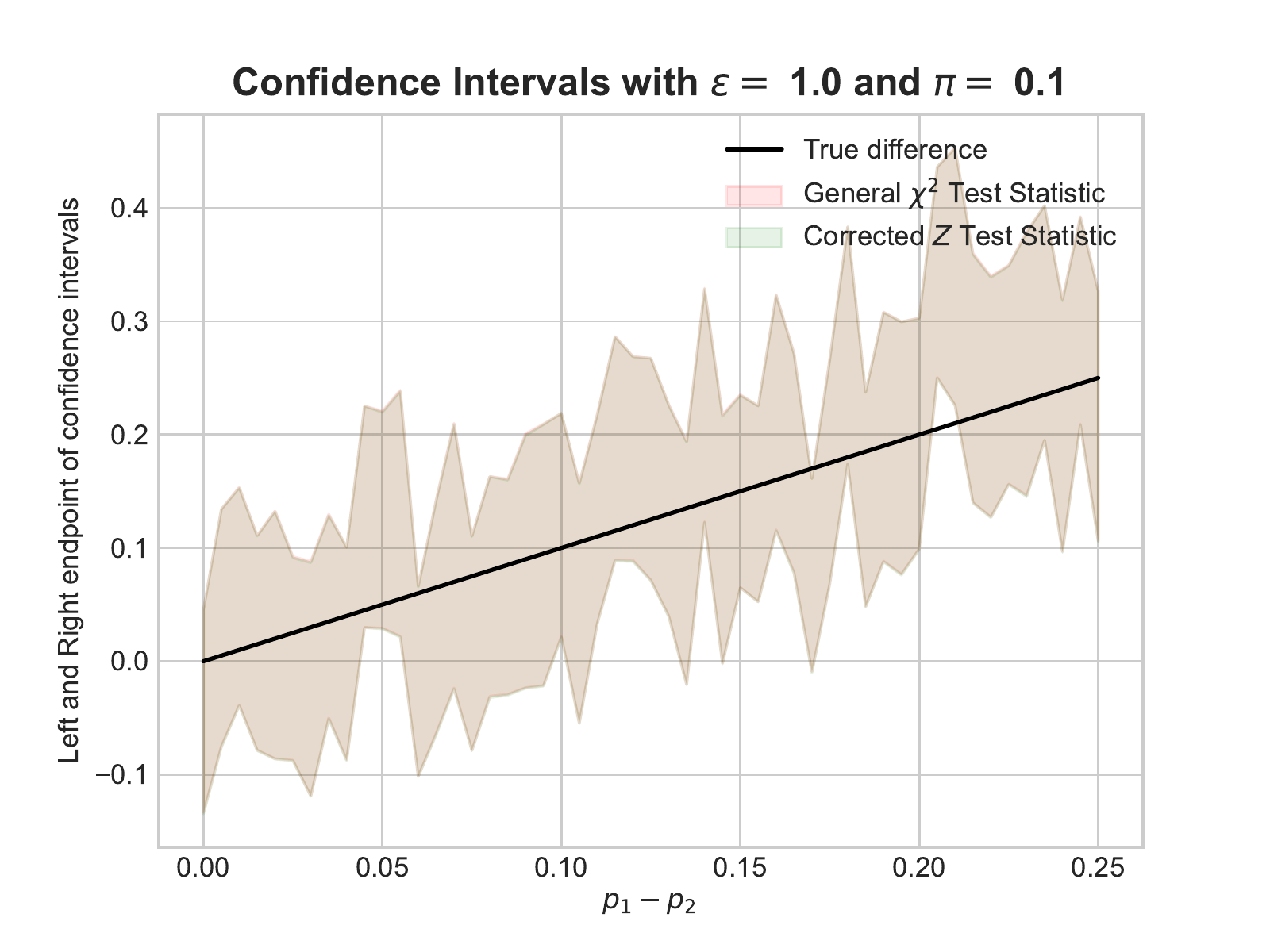}
  \includegraphics[width=0.46\linewidth]{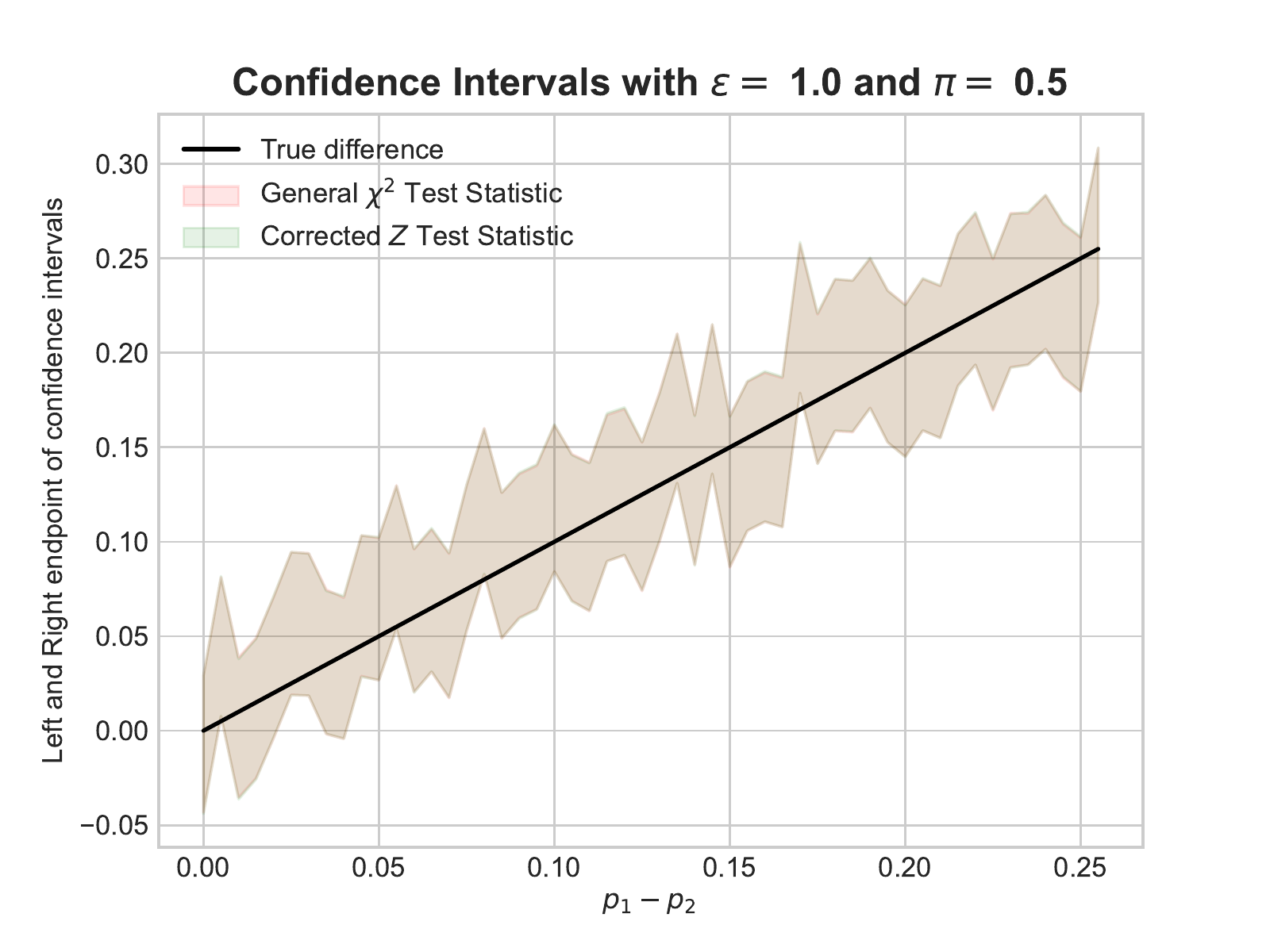}
  \caption{Comparing Confidence Intervals for Group Local DP between using the corrected $Z$-test and the $\chi^2$ tests with $p_1 = 0.25$ and changing $p_1 - p_2$.  The left plot uses group 1 probability $\pi = 0.10$ and the right plot uses $\pi = 0.5$.  Note that we are showing $95\%$ confidence intervals a single trial by generating data with different $p_1- p_2$.}
  \label{fig:DiffProp2}
\end{figure}

\subsection{Results on Adult Dataset}

We also present results on the UCI Adult dataset \cite{Adult}.  Since we are working with two sensitive groups, we will use gender as the sensitive group category, where a sample is either \emph{Male} or \emph{Female}.  We will test whether there is a significant difference in whether males or females make more that \$50k salary.  Using \texttt{adult.data}, we compute confidence intervals for the difference in proportion in salary above \$50k between Males and Females after using randomized response on the gender of each sample.  We then compute the (non-private) sample difference in proportion on the \texttt{adult.test} data.  We compare the traditional proportion test that ignores the additional noise due to privacy, the $Z$-test with the correction given in \eqref{eq:DeltaEpsilon}, and our general $\chi^2$ test in determining confidence intervals at the 95\% significance level.  We present in Figure~\ref{fig:adultDiffPropCI} how often the different methods miss the test data difference in proportion over 1000 trials for various $\epsilon$ parameters.  Note that the randomness in each trial is only over the randomness due to randomized response and the data remains fixed.  Because the non-private confidence interval on the training data does overlap the test data difference in proportion estimate, as $\epsilon$ increases the different methods will rarely miss the test data proportion difference.    

\begin{center}
\begin{figure}
\centering
  \includegraphics[width=0.46\linewidth]{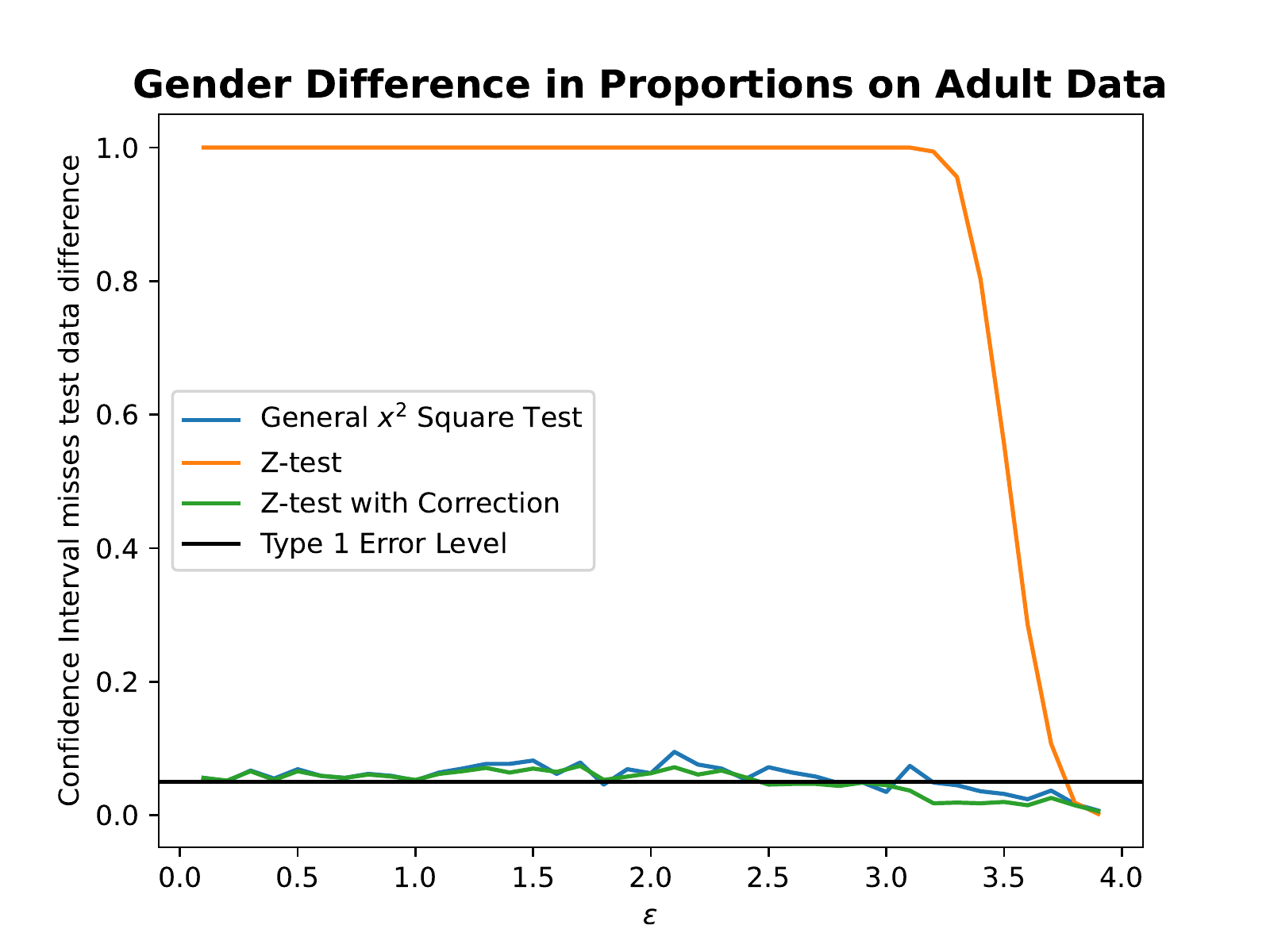}
  \caption{Proportion of times out of 1000 independent trials that the training data difference in proportion confidence interval misses the test data difference in proportion where we treat Gender as the sensitive group and the outcome as the binary outcome of whether a sample makes more than \$50k salary.}
  \label{fig:adultDiffPropCI}
\end{figure}
\end{center}

Thus far, we have detailed an approach to difference in proportions testing in the LGDP setting. First, in the case where we test $H_0: p_1 = p_2$, the classical test does not appear to need modification. Second, we consider testing $H_0: p_1 = p_2 + \Delta$. We show that a standard z-test will give extremely poor confidence intervals, and provide two alternatives: a z-test corrected for privacy, and a general $\chi^2$ test. The two seem to perform quite similarly. We compare the general $\chi^2$ test with analogous tests in the more restrictive local DP setting, and - unsurprisingly but encouragingly - are able to improve power in our less restrictive setting. We also show that the corrected z-test drastically outperforms the uncorrected z-test, correcting for the bias observed in uncorrected confidence intervals.

%% file: independenceTest.tex
\section{Independence Testing with Categorical Data}\label{sect:independence}
We now consider testing whether the success probability is equal across several groups simultaneously.  A common test is to use the classical Pearson $\chi^2$ test for independence to see whether the outcome is independent of the group.  Given that there are multiple groups, rather than just 2, we will compare the three private mechanisms presented in Section~\ref{sect:privacyPrelim}.  
To design private $\chi^2$ tests for determining whether the success probability is the same across $g>2$ groups simultaneously, we follow the general $\chi^2$ test approach outlined in Section~\ref{sect:generalChi}, which was also used to design (fully) local DP $\chi^2$ tests in \cite{GaboardiRo18}. We will compare these tests with those we develop in the less restrictive private setting of LGDP.  

We first set up some notation.  Let $Y_i \sim \text{Multinomial}(1,\theta(\pi,p))$ be the data entry for individual $i$, where $p \in [0,1]$ is the success probability across all groups and $\pi \in [0,1]^g$ is the probability vector over all $g$ groups, so that $\sum_{i=1}^g\pi_i = 1$.  Note that the covariance matrix of $Y_i$ will still be of the form of the covariance matrix given in \eqref{eq:multinomCovar}.  Let $W_i \sim \text{Multinomial}(1,\pi)$, which will determine which group each sample belongs to, and let the outcome for a sample in group $j \in [g]$ be written as $X_{i}[j] \sim \text{Bern}(p_j)$.  Hence, we have the following random vector that we use in the $\chi^2$ statistic.
\begin{align*}
& Y_i = \left(\underbrace{W_{i}[1] \cdot X_{i}[1], W_{i}[2] \cdot X_{i}[2], \cdots, W_{i}[g] \cdot X_{i}[g]}_{\text{successes}}, \right. \\
& \qquad\qquad\qquad \left. \underbrace{W_{i}[1] \cdot(1-X_{i}[1]), W_{i}[2] \cdot(1-X_{i}[2]), \cdots, W_{i}[g] \cdot(1-X_{i}[g]}_{\text{failures}} )\right)^\intercal.
\end{align*}

To match the contingency table format we had in the previous section, we will write the coordinates of $Y_i =  (Y_i[1,1], Y_i[1,2], \cdots, Y_i[1,g], Y_i[2,1], \cdots, Y_i[2,g])$, so that the entries whose first index is 1 are the successes and the entries whose first index is 2 are the failures. Once we use a privatization mechanism $M$, whether it be $g$-randomized response, bit flipping, or the subset mechanism, the resulting outcome will be $Y_i^\diffp$, which will consist of some successes/failures from group $j$ being potentially replicated across various groups, and perhaps removed from group $j$.  We then receive $n$ i.i.d. samples from the distribution of $Y_i^\diffp$ to obtain counts of successes and failures in each group $Y^\diffp = \sum_{i=1}^n Y_i^\diffp$.  To make the randomness in the privacy mechanism explicit, we will write the mechanism as a matrix of noise terms multiplied by the original $Y_i$.  With each mechanism $M(W_i) \in \{ 0,1\}^g$, we write the random matrix $Z_i^\diffp \in \{0,1 \}^{g\times g}$ where column $j$ is the corresponding random entries for $M(j)$.  We can then succinctly write $Y_i^\diffp$ in the following way where $\mathbf{0}$ is a $g$ by $g$ matrix of zeros:
\begin{equation}
Y_i^\diffp=  \begin{bmatrix}
			Z_i^\diffp & \mathbf{0} \\
			\mathbf{0} & Z_i^\diffp
			\end{bmatrix} Y_i
\label{eq:privY}
\end{equation}

Here we can distinguish our work from the local DP setting considered in \citet{GaboardiRo18}.  In particular, the matrix multiplying $Y_i$ in the (fully) local DP setting would include terms where there is a zero block submatrix.  This would correspond to successes within group $j$ being able to switch to a failure in another group $j'$.  In our setting, we are not privatizing outcomes, i.e. successes or failures, hence the zero block matrices.  

With our general set up, we can then compare the various privacy mechanisms for various levels of $\diffp$.  
That is, given privatized data $Y_i^\diffp$ for $i \in [n]$ along with the privacy level $\diffp$ and the mechanism $M$ used, we then form the $\chi^2$ statistic $\hat{D}$ from \eqref{eq:privChiSq} and compare with the critical value $\chi^2_{g-1, 1-\alpha}$ for $1-\alpha$ level of significance.  

\subsection{Randomized Response}
For $g$-randomized response, the matrix $Z^\diffp_i$ will have column $j$ be a multinomial of a single trial with probability vector that is $\tfrac{e^\diffp}{e^\diffp + g-1}$ in position $j$ and $\tfrac{1}{e^\diffp + g - 1}$ in every other coordinate.  Note that the resulting privatized data $Y^\diffp $ will still follow a multinomial distribution, where the probability vector is the following
\[ 
\vec{\theta}^\diffp(\pi, p) = \begin{pmatrix}
\frac{p}{e^\diffp + g - 1} \begin{bmatrix}
e^\diffp & 1 & \cdots & 1 \\
1 & e^\diffp & \cdots & 1 \\
 &  & \ddots &  \\
1 & 1 & \cdots & e^\diffp 
	\end{bmatrix} \pi \\	
\\	
\frac{1-p}{e^\diffp + g -1} \begin{bmatrix}
e^\diffp & 1 & \cdots & 1 \\
1 & e^\diffp & \cdots & 1 \\
 &  & \ddots &  \\
1 & 1 & \cdots & e^\diffp 
	\end{bmatrix} \pi \\
\end{pmatrix}
\]
Note that because the resulting vector $Y^\diffp$ is still a multinomial, we do not need to write out the covariance matrix in order to calculate the general $\chi^2$ statistic, but rather the generalized inverse is the diagonal matrix whose entries are the inverse of the entries in $\vec{\theta}^\diffp(\pi, p)$.  Note that we have $\vec{\theta}^\diffp(\pi, p) = \vec{\theta}(\pi^\diffp,p)$, where $\pi^\diffp = \left( \tfrac{e^\diffp}{e^\diffp+g-1} \pi_j + \tfrac{1}{e^\diffp + g-1} (1-\pi_j) : j \in [g]\right)$, so we can use the same test statistic as in the non-private case, but with different group probabilities that can be estimated from the privatized samples.  We then need to find estimates for $p$ and $\pi^\diffp$ based on the data $Y^\diffp$.  We then use the generalized version of the estimate provided in \eqref{eq:estimatePandPi} with $\Delta = 0$.  
\[
\hat{p} = \frac{\sum_{j \in [g]}Y^\diffp[1,j]}{n}, \qquad \hat{\pi}^\diffp = \left( \hat{\pi}_j^\diffp = \frac{Y^\diffp[1,j] + Y^\diffp[2,j]}{n} : j \in [g] \right).
\]
We can then form the $\chi^2$ statistic $\hat{D}$ as in \eqref{eq:chiSqStat}, where $C(\hat{\theta}_n)^\dagger$ is the diagonal matrix whose entries are the inverse of $\vec{\theta}(\hat{\pi}^\diffp, \hat{p})$, which we write as $\text{Diag}(\vec{\theta}(\hat{p}, \hat{\pi}^\diffp))^{-1}$
\[
\hat{D} = \min_{\substack{p \in (0,1), \pi \in [0,1]^g \\ \text{s.t. } \sum_{j=1
}^g \pi_j = 1 }}\left\{ n \left( Y^\diffp/n - \vec{\theta}(\pi, p)  \right)^\intercal \text{Diag}(\vec{\theta}( \hat{\pi}^\diffp,\hat{p}))^{-1} \left( Y^\diffp/n - \vec{\theta}(\pi, p)\right) \right\}
\]
We then compare $\hat{D}$ with a $\chi^2$ distribution with $g-1$ degrees of freedom, as in the non-private test, because the rank of the covariance matrix is at most $2g-1$ and we are minimizing over $g$ variables ($p$ and $\pi_1, \cdots, \pi_{g-1}$ since $\pi_g = 1 - \sum_{j=1}^{g-1}\pi_{j}$).
\subsection{Bit Flipping}\label{sect:BitFlipBinaryOutcomes}
The bit flipping mechanism will result in a random vector $Y_i^\diffp$ that is not a multinomial, so a little more care will be needed in computing the general $\chi^2$ statistic.
The matrix of noise terms in \eqref{eq:privY} for the bit flipping mechanism will consist of the following entries,
\[
Z^\diffp_i[j,j] \sim \text{Bern}\left(\tfrac{e^{\diffp/2}}{ e^{\diffp/2} + 1} \right), \quad j \in [g] \quad  \text{ and } \quad Z^\diffp_i[j,\ell] \sim \text{Bern}\left(\tfrac{1}{ e^{\diffp/2} + 1} \right), \quad j \neq \ell.
\]

We will first compute $E[Y_i^\diffp] = \vec{\theta}^\diffp(\pi, p)$ when we privatize the group with the bit flipping mechanism
\[
\vec{\theta}^\diffp(\pi, p) = 
\begin{pmatrix}
\frac{p}{e^{\diffp/2} +1} \begin{bmatrix}
e^{\diffp/2} & 1 & \cdots & 1 \\
1 & e^{\diffp/2} & \cdots & 1 \\
\vdots & \vdots & \ddots & \vdots \\
1 & 1 & \cdots & e^{\diffp/2} 
	\end{bmatrix} \pi \\ \\
	\frac{1-p}{e^{\diffp/2} +1} \begin{bmatrix}
e^{\diffp/2} & 1 & \cdots & 1 \\
1 & e^{\diffp/2} & \cdots & 1 \\
\vdots  & \vdots & \ddots & \vdots  \\
1 & 1 & \cdots & e^{\diffp/2} 
	\end{bmatrix} \pi
\end{pmatrix}.
\]
We now compute the covariance matrix, $C(\pi,p;\diffp) = \E[Y_i^\diffp \  (Y_i^\diffp)^\intercal] - \vec{\theta}^\diffp(\pi, p) \vec{\theta}^\diffp(\pi, p)^\intercal$
\[
\E[Y_i^\diffp[1,j]^2]= \frac{p}{e^{\diffp/2}+1} (\pi_j e^{\diffp/2} + (1-\pi_j)), \quad \E[Y_i^\diffp[2,j]^2] = \frac{1-p}{e^{\diffp/2}+1} (\pi_j e^{\diffp/2} + (1-\pi_j)), \qquad \forall j \in [g].
\]
\begin{align*}
\E[Y_i^\diffp[1,j] \cdot Y_i^\diffp[1,\ell]] & = \frac{p}{(e^{\diffp/2} + 1)^2} \left( e^{\diffp/2} (\pi_j + \pi_\ell) + (1 - \pi_j - \pi_\ell ) \right), \qquad &j, \ell \in [g], j \neq \ell
\\
\E[Y_i^\diffp[2,j] \cdot Y_i^\diffp[2,\ell]] &= \frac{1-p}{(e^{\diffp/2} + 1)^2} \left( e^{\diffp/2} (\pi_j + \pi_\ell) + (1 - \pi_j - \pi_\ell ) \right), \qquad & j, \ell \in [g], j \neq \ell
\\
\E[Y_i^\diffp[1,j] \cdot Y_i^\diffp[2,\ell]] &= \E[Y_i^\diffp[2,j] \cdot Y_i^\diffp[1,\ell] ] = 0 & j, \ell \in [g], j \neq \ell
\end{align*}

Unlike the case for $g$-randomized response, the all ones vector is not in the null space.  \citet{GaboardiRo18} showed that with the bit flipping algorithm in the local DP setting, the all ones vector is an eigenvector, whose eigenvalue depends solely on the privacy loss parameter $\diffp$.  Using a technique from \cite{KiferRo17}, they showed how one can project out this eigenvector and the resulting $\chi^2$ statistic will have one fewer degree of freedom (asymptotically).  We will not be able to do a similar technique here in the LGDP setting, since the all ones vector is not an eigenvector for general $\pi, p$.  Hence, we will not be able to reduce the degrees of freedom in its asymptotic distribution, at least with similar techniques although it might be possible another way.  

Next, we need to compute estimates for $p$ and $\pi$.
\[
\hat{p} = \frac{(e^{\diffp/2}+1)\sum_{j=1}^g Y^\diffp[j]}{n \left(e^{\diffp/2} + g-1\right)} ,\qquad \hat{\pi} = \left( \frac{ \frac{Y^\diffp[j] + Y^\diffp[g+j]}{n} - \frac{1}{e^{\diffp/2}+1} }{\frac{2e^{\diffp/2}}{e^{\diffp/2} + 1}-1} : j \in [g]\right) . 
\]
We then plug in our estimates to the covariance matrix and form our $\chi^2$ statistic, as in \eqref{eq:chiSqStat}, and compare it to a $\chi^2$ distribution with $2g - g = g$ degrees of freedom (one larger than the non-private version due to the covariance matrix being non-singular).

\subsection{Subset Mechanism}

The subset mechanism \cite{YeBa17} from Definition~\ref{defn:subset} can also be used to privatize the group membership of each sample $i$, which takes an additional parameter $k < g$.  Column $j$ of $Z_i^\diffp$ will correspond to the outcome of the subset mechanism $M(j)$.  That is. $Z_i^\diffp[j,j] \sim \text{Bern}(\tfrac{k e^\diffp}{k e^\diffp + g - k})$, and then the other coordinates $Z_i^\diffp[\ell,j]$ for $\ell \neq j$ will depend on the realization of $Z_i^\diffp[j,j] $.  So if $Z_i^\diffp[j,j]=1$, then $\left(Z_i^\diffp[\ell,j] : \ell \neq j \right)$ will sample $k-1$  ones uniformly at random without replacement, while if $Z_i^\diffp[j,j]=0$, then $\left(Z_i^\diffp[\ell,j] : \ell \neq j \right)$ will sample $k$  ones uniformly at random without replacement.  Following our framework, we first compute $\vec{\theta}^\diffp(\pi, p; k) = \E[Y_i^\diffp]$ when we use the subset mechanism.
\[
\vec{\theta}^\diffp(\pi, p;k) = 
\begin{pmatrix}
\frac{p}{{g-1\choose k-1} e^\diffp + {g-1\choose k} } 
\begin{bmatrix}
{g-1 \choose k-1}e^\diffp & \left( {g-2 \choose k-2}e^\diffp + {g-2\choose k-1} \right) & \cdots & \left( {g-2 \choose k-2}e^\diffp + {g-2\choose k-1} \right) \\
\left( {g-2 \choose k-2}e^\diffp + {g-2\choose k-1} \right) & {g-1 \choose k-1}e^\diffp & \cdots & \left( {g-2 \choose k-2}e^\diffp + {g-2\choose k-1} \right) \\
\vdots & \vdots & \ddots & \vdots \\
\left( {g-2 \choose k-2}e^\diffp + {g-2\choose k-1} \right) & \left( {g-2 \choose k-2}e^\diffp + {g-2\choose k-1} \right) & \cdots & {g-1 \choose k-1}e^\diffp
	\end{bmatrix} \pi \\ \\
	\frac{1-p}{{g-1\choose k-1} e^\diffp + {g-1\choose k} } 
	\begin{bmatrix}
{g-1 \choose k-1}e^\diffp & \left( {g-2 \choose k-2}e^\diffp + {g-2\choose k-1} \right) & \cdots & \left( {g-2 \choose k-2}e^\diffp + {g-2\choose k-1} \right) \\
\left( {g-2 \choose k-2}e^\diffp + {g-2\choose k-1} \right) & {g-1 \choose k-1}e^\diffp & \cdots & \left( {g-2 \choose k-2}e^\diffp + {g-2\choose k-1} \right) \\
\vdots & \vdots & \ddots & \vdots \\
\left( {g-2 \choose k-2}e^\diffp + {g-2\choose k-1} \right) & \left( {g-2 \choose k-2}e^\diffp + {g-2\choose k-1} \right) & \cdots & {g-1 \choose k-1}e^\diffp
	\end{bmatrix} \pi
\end{pmatrix}
\]

We then compute the covariance matrix $C(\pi,p;\diffp, k) = \E[Y_i^\diffp (Y_i^\diffp)^\intercal] - \vec{\theta}^\diffp(\pi, p) \vec{\theta}^\diffp(\pi, p)^\intercal$
\begin{align*}
\E[Y_i^\diffp[1,j]^2] & = p \frac{{g-1 \choose k-1}e^\diffp\pi_j + \left( {g-2 \choose k-2}e^\diffp  + {g-2 \choose k-1}\right) (1-\pi_j)}{{g-1\choose k-1} e^\diffp + {g-1\choose k} } , & \qquad \forall j \in [g].\\
\E[Y_i^\diffp[2,j]^2] &= (1-p) \frac{{g-1 \choose k-1}e^\diffp\pi_j + \left( {g-2 \choose k-2}e^\diffp  + {g-2 \choose k-1}\right) (1-\pi_j)}{{g-1\choose k-1} e^\diffp + {g-1\choose k} } , & \qquad \forall j \in [g].
\end{align*}
Now for $\ell \neq j$ and $j, \ell \in [g]$,
\begin{align*}
\E[Y_i^\diffp[1,j] \cdot Y_i^\diffp[1,\ell]] & =\frac{p}{{g-1\choose k-1} e^\diffp + {g-1\choose k} }  \left( e^{\diffp} {g-2\choose k-2}(\pi_j +\pi_\ell) + \left(e^\diffp {g-3\choose k-3} + {g-3\choose k-2} \right) (1-\pi_j - \pi_\ell) \right)
\\
\E[Y_i^\diffp[2,j] \cdot Y_i^\diffp[2, \ell]] &= \frac{1-p}{p} \E[Y_i^\diffp[1,j] Y_i^\diffp[1,\ell]] ,
\\
\E[Y_i^\diffp[1,j] \cdot Y_i^\diffp[2,\ell]] &= \E[Y_i^\diffp[2,j] Y_i^\diffp[1,\ell] ] = 0 
\end{align*}
We then consider the rank of this covariance matrix.  Note that the all ones vector is in the null space of the covariance matrix.  
\begin{lemma}\label{lem:singularCovar}
The covariance matrix $C(\pi,p; \diffp, k)$ corresponding to the subset mechanism has the all ones vector $\mathbf{1} \in \R^{2g}$ in its null space, i.e.
\[
C(\pi,p; \diffp, k) \mathbf{1} = \mathbf{0}
\]
\end{lemma}

\begin{proof}
Consider element $j \in [g]$ in $ \E[Y_i^\diffp (Y_i^\diffp)^\intercal] \mathbf{1}$, where we ignore the coefficient $ \frac{p}{{g-1\choose k-1} e^\diffp + {g-1\choose k} }$,
\begin{align*}
&  \pi_j \left(e^\diffp {g-1 \choose k-1 } + {g-1 \choose k }  \right) + (1-\pi_j) \left( e^\diffp {g-2 \choose k-2} +{g-2 \choose k-1} \right) \\
 & \qquad\qquad  +  \sum_{j \neq \ell} \left( (\pi_j + \pi_\ell)  e^\diffp {g-2 \choose k-2 }  + (1-\pi_j - \pi_\ell) \left( e^\diffp {g-3 \choose k-3} +{g-3 \choose k-2}\right) \right)  \\ 
 & = \pi_j e^\diffp \left(  {g-1 \choose k-1 } + (g-1)  {g-2 \choose k-2 } \right) + (1-\pi_j) e^\diffp \left(  2 {g-2 \choose k-2 } + (g-2)  {g-3 \choose k-3 }\right) \\
 & \qquad \qquad+ (1-\pi_j) \left( {g-2 \choose k-1 } + (g-2){g- 3 \choose k-2 }\right)
\end{align*}
We then use the following properties
\[
x {x-1 \choose y-1} = y {x\choose y}, \qquad x {x-1 \choose y} = (y +1) {x \choose y+1}
\]
Hence, after simplifying, we get
\begin{align*}
E[Y_i^\diffp (Y_i^\diffp)^\intercal] \mathbf{1} & =\frac{1}{{g-1\choose k-1} e^\diffp + {g-1\choose k} } \left[ \pi_j e^\diffp k {g-1\choose k-1} + (1-\pi_j) \left(e^\diffp k{g-2\choose k-2} + k{g-2\choose k-1} \right)  \right]  \\
& = k \vec{\theta}^\diffp(\pi,p;k)
\end{align*}
Furthermore, we have
\[
\vec{\theta}^\diffp(\pi,p;k) ^\intercal \mathbf{1} = k.
\]
Putting everything together, we have
\[
C(\pi,p; \diffp, k) \mathbf{1} = \left( \E[Y_i^\diffp (Y_i^\diffp)^\intercal] - \vec{\theta}^\diffp(\pi, p) \vec{\theta}^\diffp(\pi, p)^\intercal \right) \mathbf{1} =  k \vec{\theta}^\diffp(\pi,p;k) - k \vec{\theta}^\diffp(\pi,p;k) = \mathbf{0}.
\]

\end{proof}

Next, we compute estimates for $p$ and $\pi$ based on the sample $Y^\diffp$ as well as $\diffp$ and $k$.
\begin{align*}
\hat{p} &= 
\frac{\left( {g-1\choose k-1} e^\diffp + {g-1\choose k} \right) \cdot \sum_{j=1}^g Y[1,j]^\diffp}{n \left( { g-1 \choose k-1} e^ \diffp + (g-1) \left( {g-2 \choose k-2} e^\diffp + {g-2\choose k-1} \right)   \right)} = \frac{\sum_{j=1}^g Y[1,j]^\diffp}{n k} ,
\\
 \hat{\pi} &= \left( \frac{ \left( {g-1\choose k-1} e^\diffp + {g-1\choose k} \right)\cdot \left( \frac{Y[1,j]^\diffp + Y[2,j]^\diffp}{n} \right)- \left(e^\diffp {g-2\choose k-2} + {g-2\choose k-1} \right) }{e^\diffp \left({ g-1 \choose k-1 } - { g-2 \choose k-2 } \right) - { g-2 \choose k-1 } } : j \in [g]\right) . 
\end{align*}

We can then use the $\chi^2$ statistic in \eqref{eq:chiSqStat} with the covariance matrix $C(\pi, p; \diffp, k)$ that we computed along with the estimates $\hat{p}$ and $\hat{\pi}$. Note that we showed in Lemma~\ref{lem:singularCovar} that the covariance is not full rank, so we will use $2g-1$ as the rank of the covariance matrix and $g-1$ degrees of freedom in the $\chi^2$ distribution in our test.

\subsection{Results}\label{sect:MultiplePropDiffResults}

In our results, we start by comparing with the existing (fully) local DP $\chi^2$ tests for independence from \cite{GaboardiRo18} in Figure~\ref{fig:IndependenceTests}.  As stated earlier, the difference between the local DP and group local DP setting is that in the latter the outcomes (successes or failures) cannot be changed but they can in the former.  We can see that the power can be drastically improved in the less restrictive model with the same privacy loss parameter $\diffp$.   
\begin{figure}
  \includegraphics[width=0.31\linewidth]{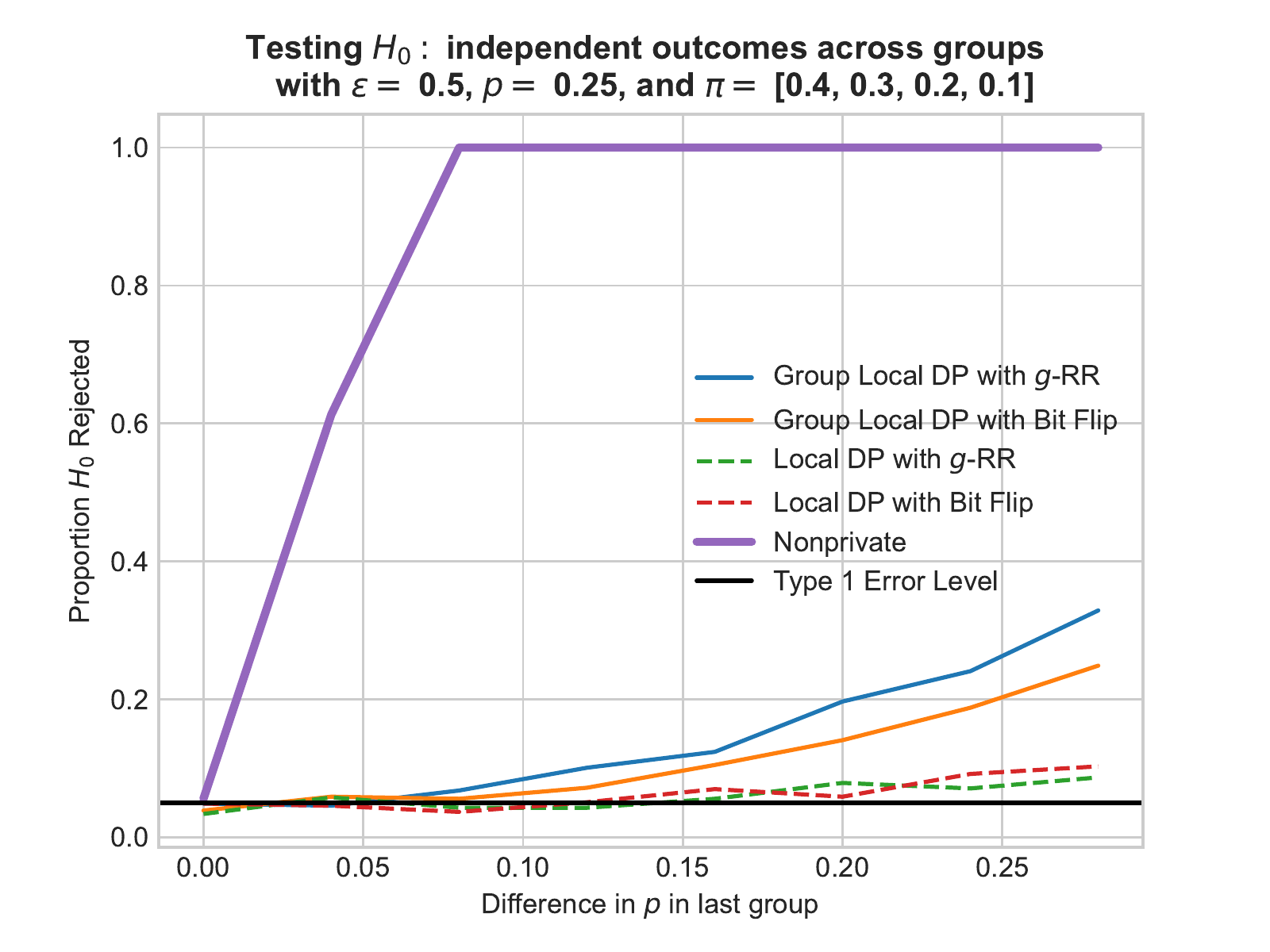}
  \includegraphics[width=0.31\linewidth]{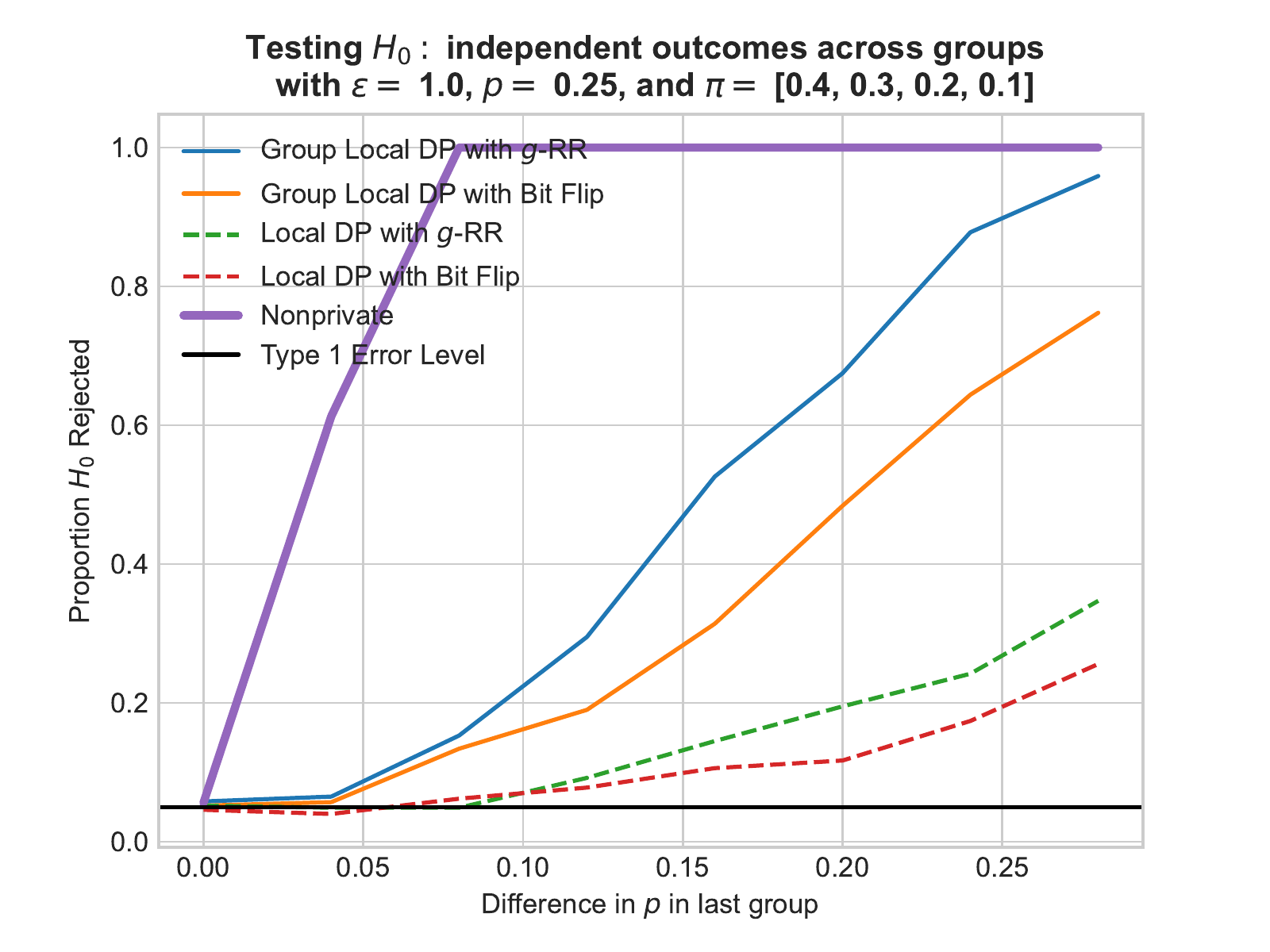}
    \includegraphics[width=0.31\linewidth]{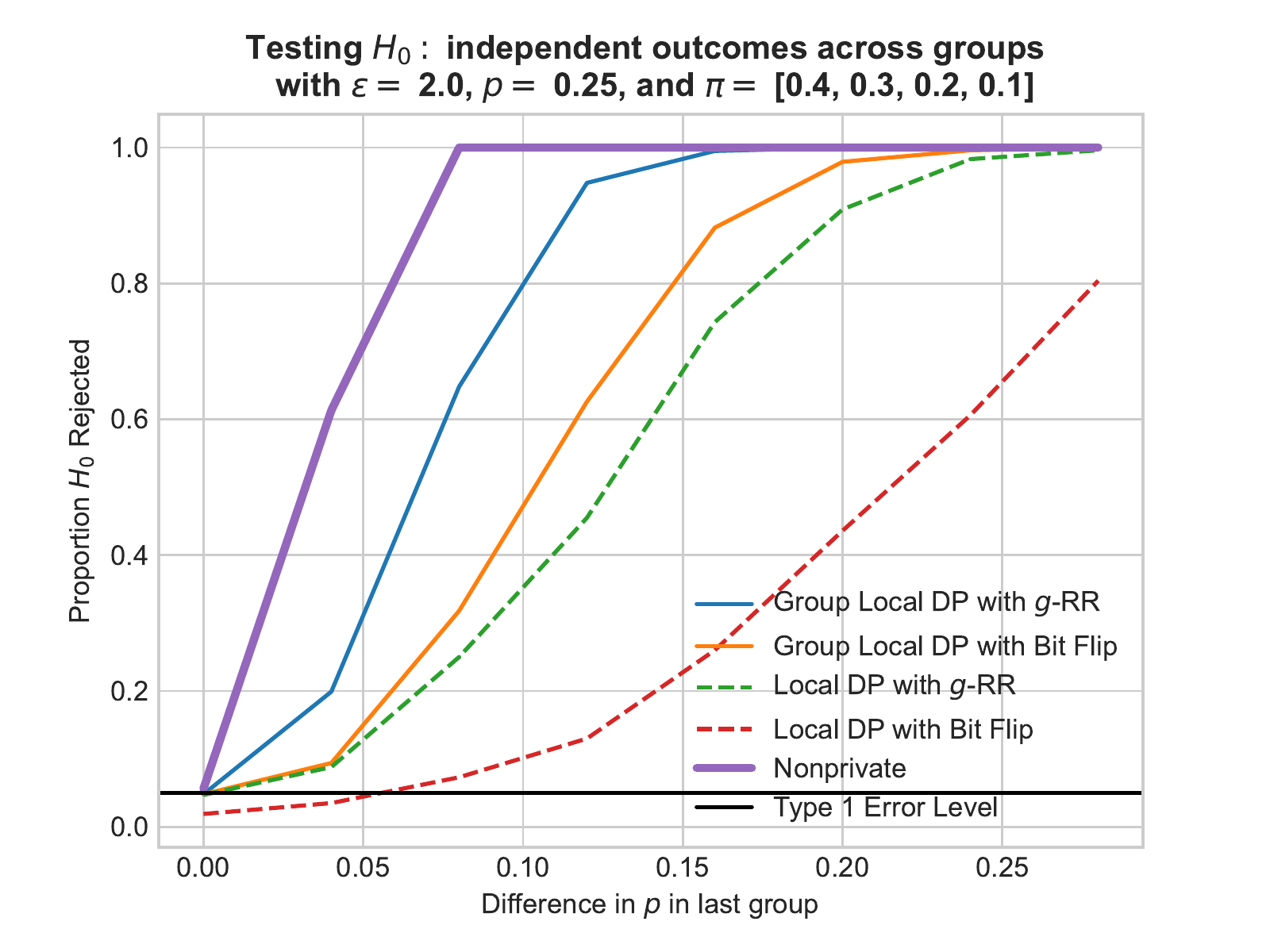}
  \caption{Comparing non-private (top bold line), Group Local DP (solid lines), and original LDP $\chi^2$ tests (dashed lines) for independence testing with the probability vector over the $g = 4$ groups to be $\pi = (0.4, 0.3, 0.2, 0.1)$ with various $\diffp$ and $n = 10000$.}
  \label{fig:IndependenceTests}
\end{figure}

We include experiments for the subset mechanism \cite{YeBa17}, which were not considered before for local DP $\chi^2$ independence tests.  The main takeaway for the subset mechanism is that it seems to strictly dominate over $g$-randomized response and bit flipping for the same level of privacy, while in \citet{GaboardiRo18} there were privacy levels where the $g$-randomized response algorithm outperformed bit flipping for the high $\diffp$ regime and vice versa in the low $\diffp$ regime.  See Figure~\ref{fig:multiplePropDiff} for experiments with various parameter settings where $g = 10$ and $\pi$ is uniform over the $g$ groups.  We can see that the subset mechanism dominates the other two at various different privacy levels, which is to be expected as the subset mechanism is known to be optimal for certain tasks.  Recall that as $\diffp$ gets large, the Subset Mechanism and $g$-randomized response are the same.

\begin{figure}
  \includegraphics[width=0.31\linewidth]{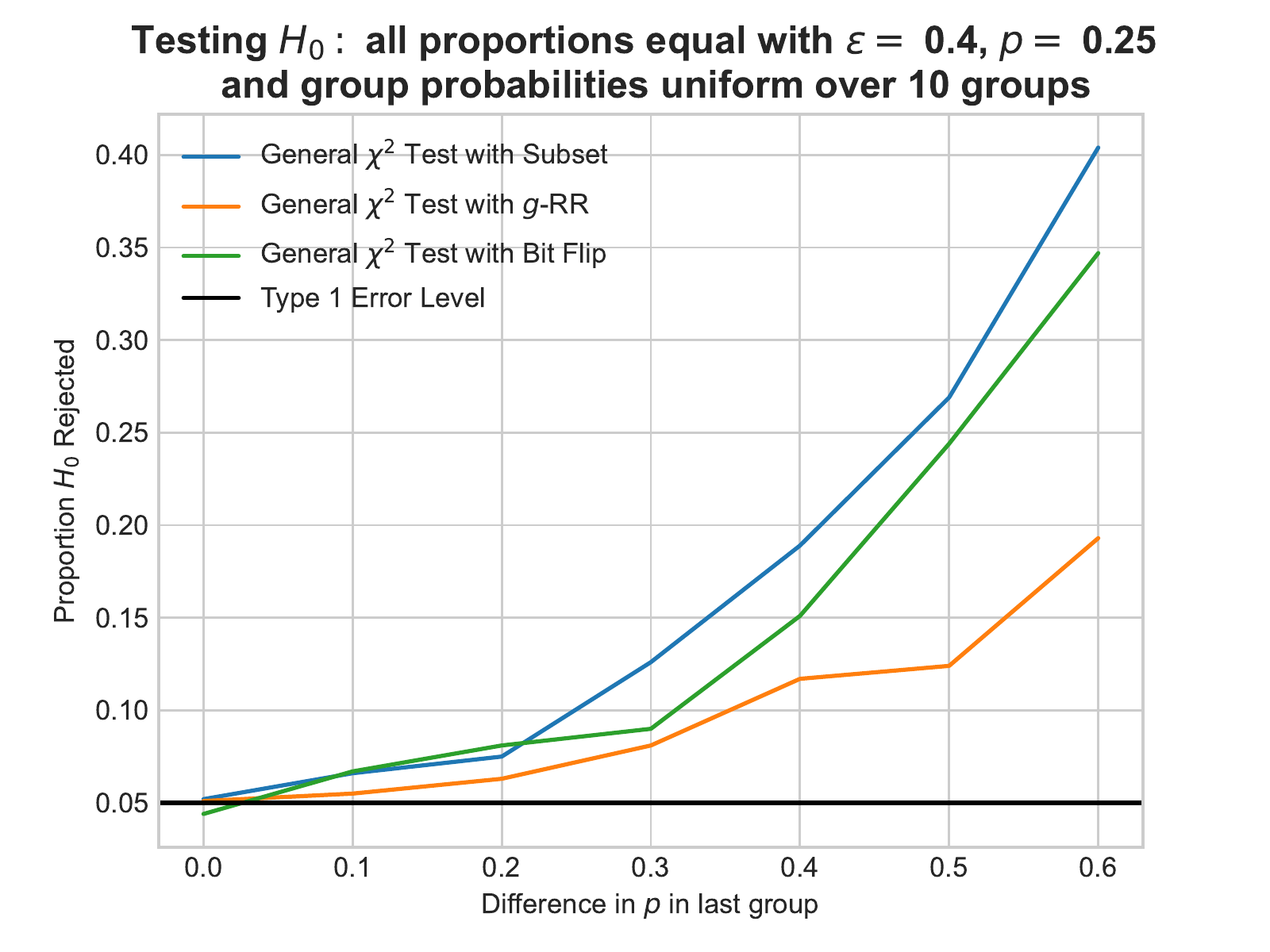}
  \includegraphics[width=0.31\linewidth]{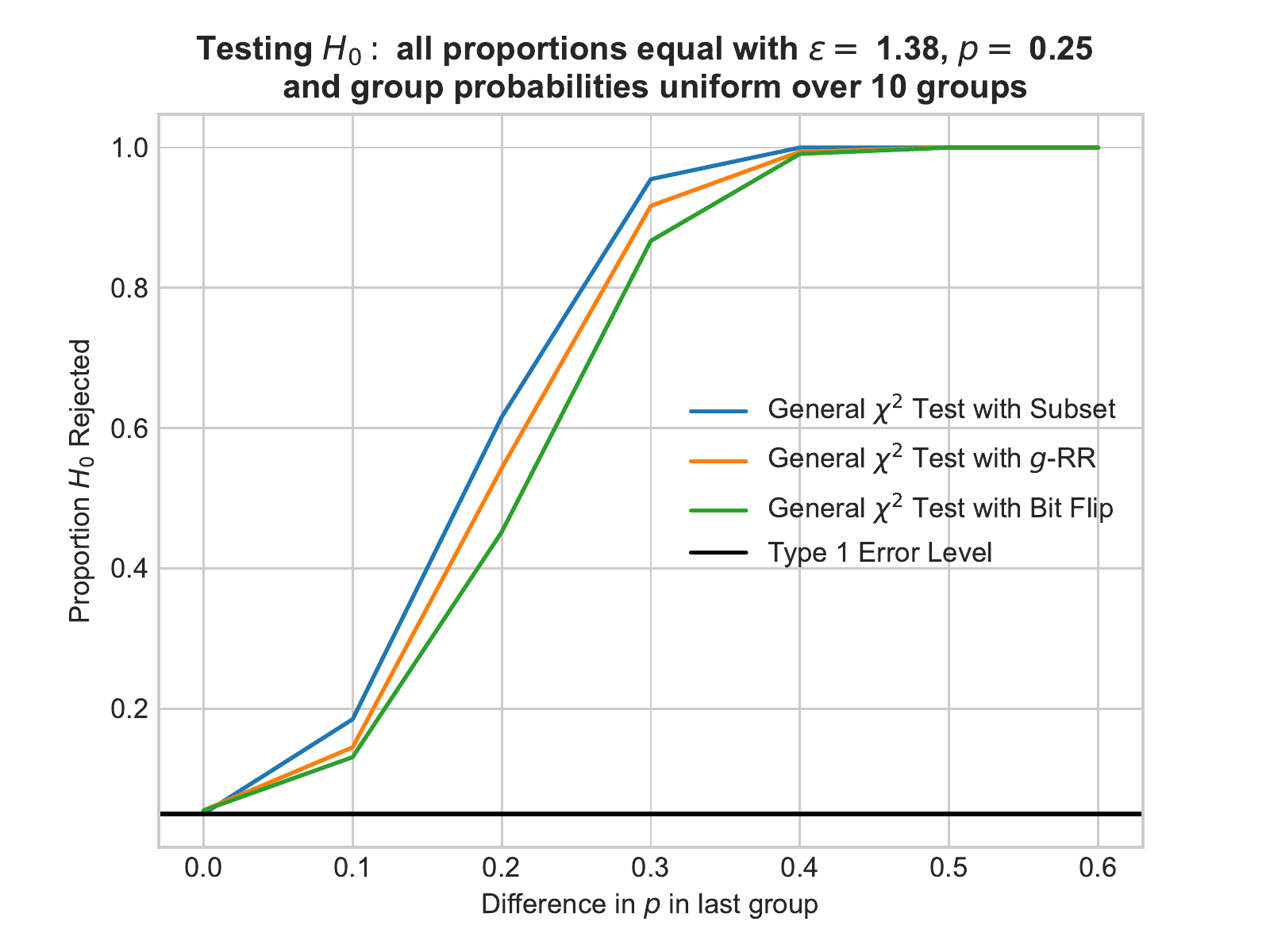}
    \includegraphics[width=0.31\linewidth]{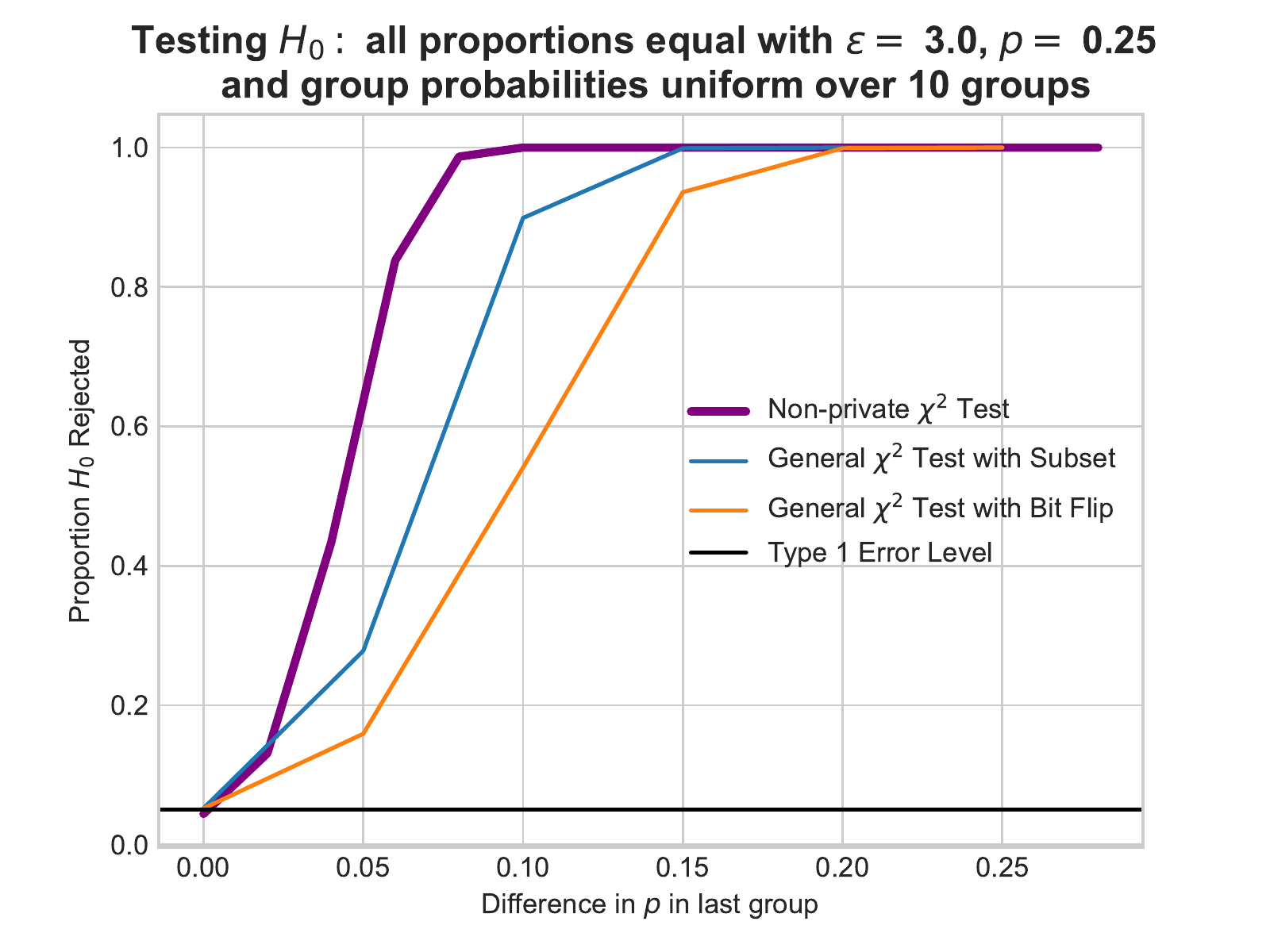}
  \caption{Comparing various Local Group DP mechanisms with corresponding $\chi^2$ test for testing whether there is a difference in success probability across different sensitive groups with various $\diffp$ and $n = 10000$.  Note that with $\diffp = 3$, we get $k = \lceil g/(e^\diffp + 1) \rceil = 1$, which is equivalent to $g$-randomized response.}
  \label{fig:multiplePropDiff}
\end{figure}

The tests we develop achieve higher empirical power than simply using the classical $\chi^2$ tests after privatizing the groups.  We present plots in Figure~\ref{fig:compareClassicMultipleProp} that shows for data generated with the subset mechanism at various privacy levels, the general $\chi^2$ test that accounts for the subset mechanism outperforms using the classic $\chi^2$ test, which does not account for the privacy mechanism.  We point out that when $\diffp$ gets larger, the two tests seem to perform similarly.  All plots consist of the proportion of times the null hypothesis was rejected over 1000 trials.  

\begin{figure}
  \includegraphics[width=0.31\linewidth]{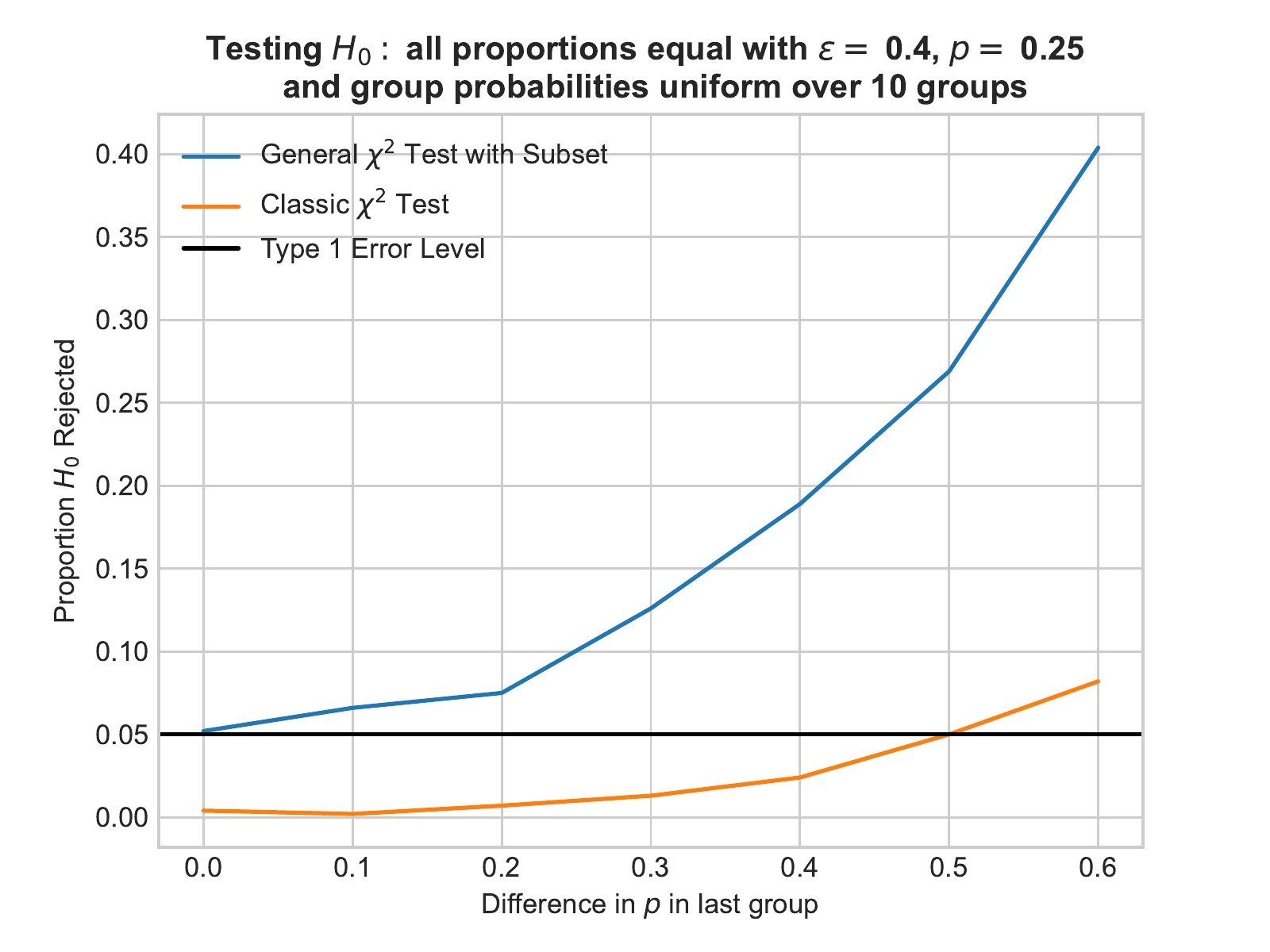}
  \includegraphics[width=0.31\linewidth]{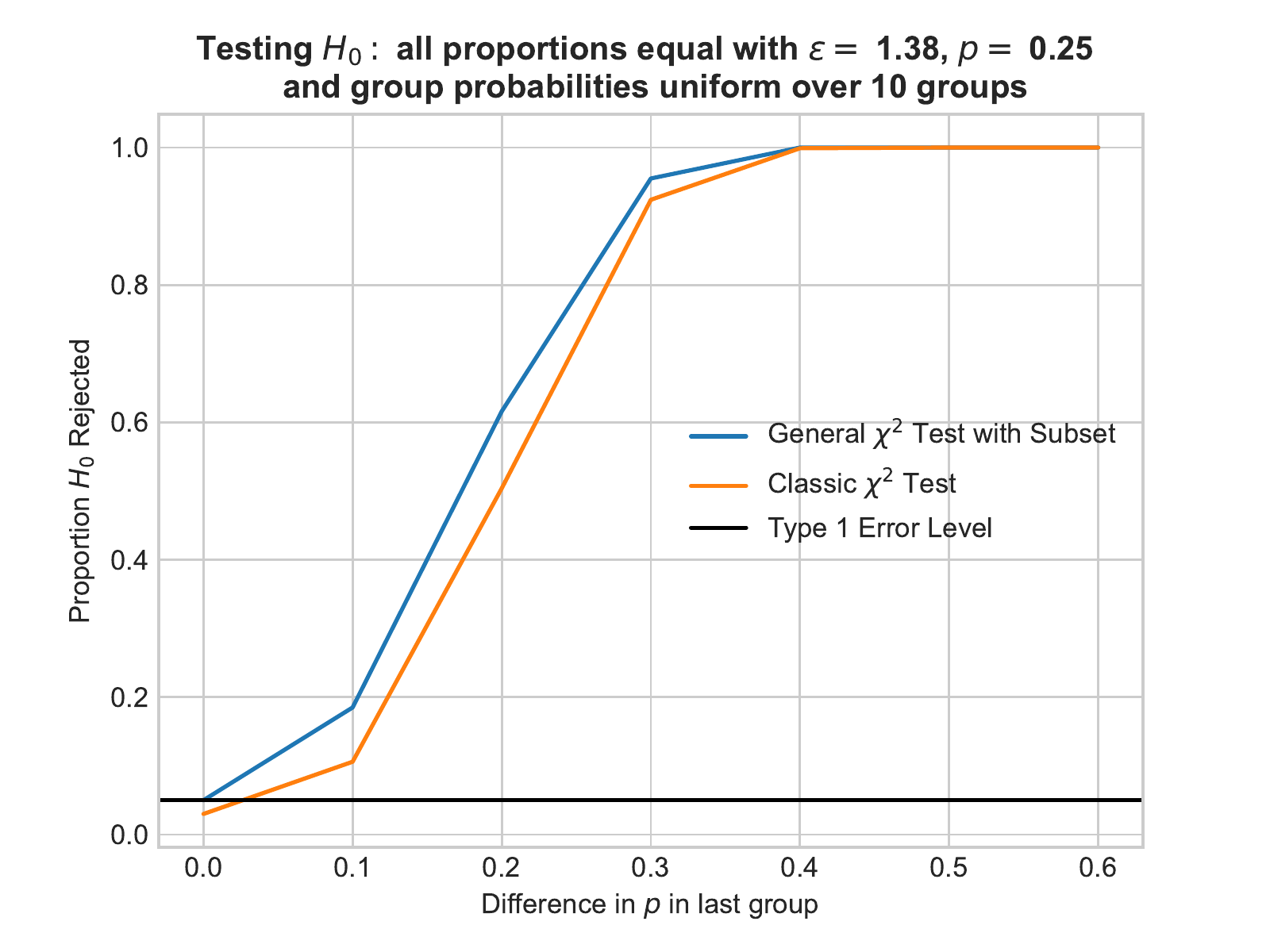}
    \includegraphics[width=0.31\linewidth]{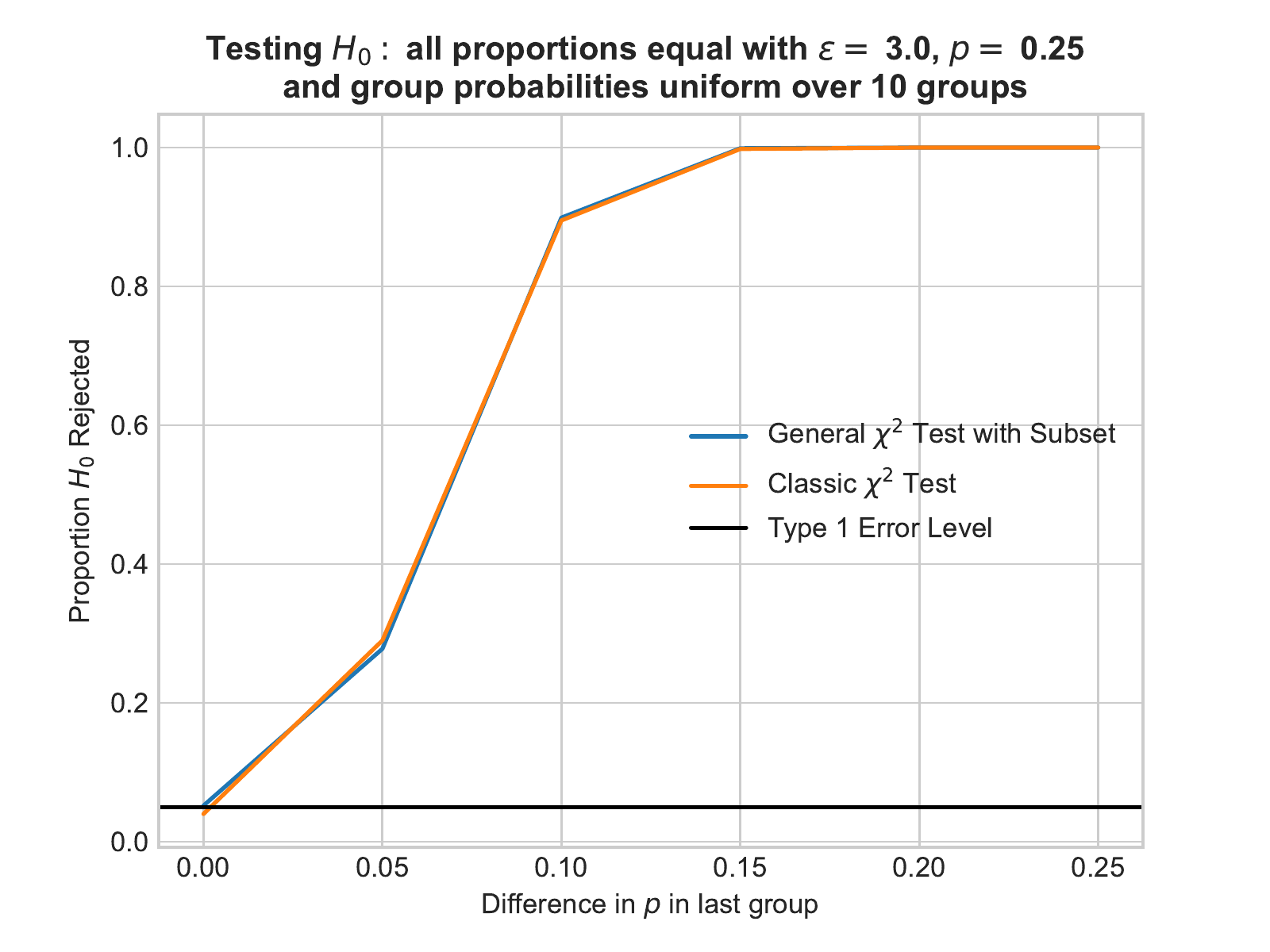}
  \caption{Comparing various Local Group DP mechanisms with corresponding $\chi^2$ test for testing whether there is a difference in success probability across different sensitive groups with various $\diffp$ and $n = 10000$.  Note that with $\diffp = 3$, we get $k = \lceil g/(e^\diffp + 1) \rceil = 1$, which is equivalent to $g$-randomized response.}
  \label{fig:compareClassicMultipleProp}
\end{figure}

We also evaluate our method on the UCI Adult dataset \cite{Adult} \texttt{adult.data}, where we will use Race as the sensitive group and the binary outcome as whether a sample makes more than \$50k salary.  Note that race contains 5 groups, with labels White, Asian-Pac-Islander, Amer-Indian-Eskimo, Other, and Black.  We now want to arrive at the same conclusion after privatizing the race of each sample as we would if we had not privatized it.  Figure~\ref{fig:AdultRaceTest} gives our results, which considers various levels of privacy and for each privacy level we compute 1000 independent trials of the subset mechanism on each sample's race and use our general $\chi^2$ test while comparing it to the traditional $\chi^2$ test for independence, which ignores the privacy mechanism.  We see that we can achieve more power for stronger levels of privacy, but since there is such a strong difference between proportions in the groups, i.e. we should reject the null hypothesis that there is no difference in proportions across all groups, for even moderate levels of privacy we arrive at the same conclusion as the non-private test almost all the time.  

\begin{center}
\begin{figure}
\centering
  \includegraphics[width=0.46\linewidth]{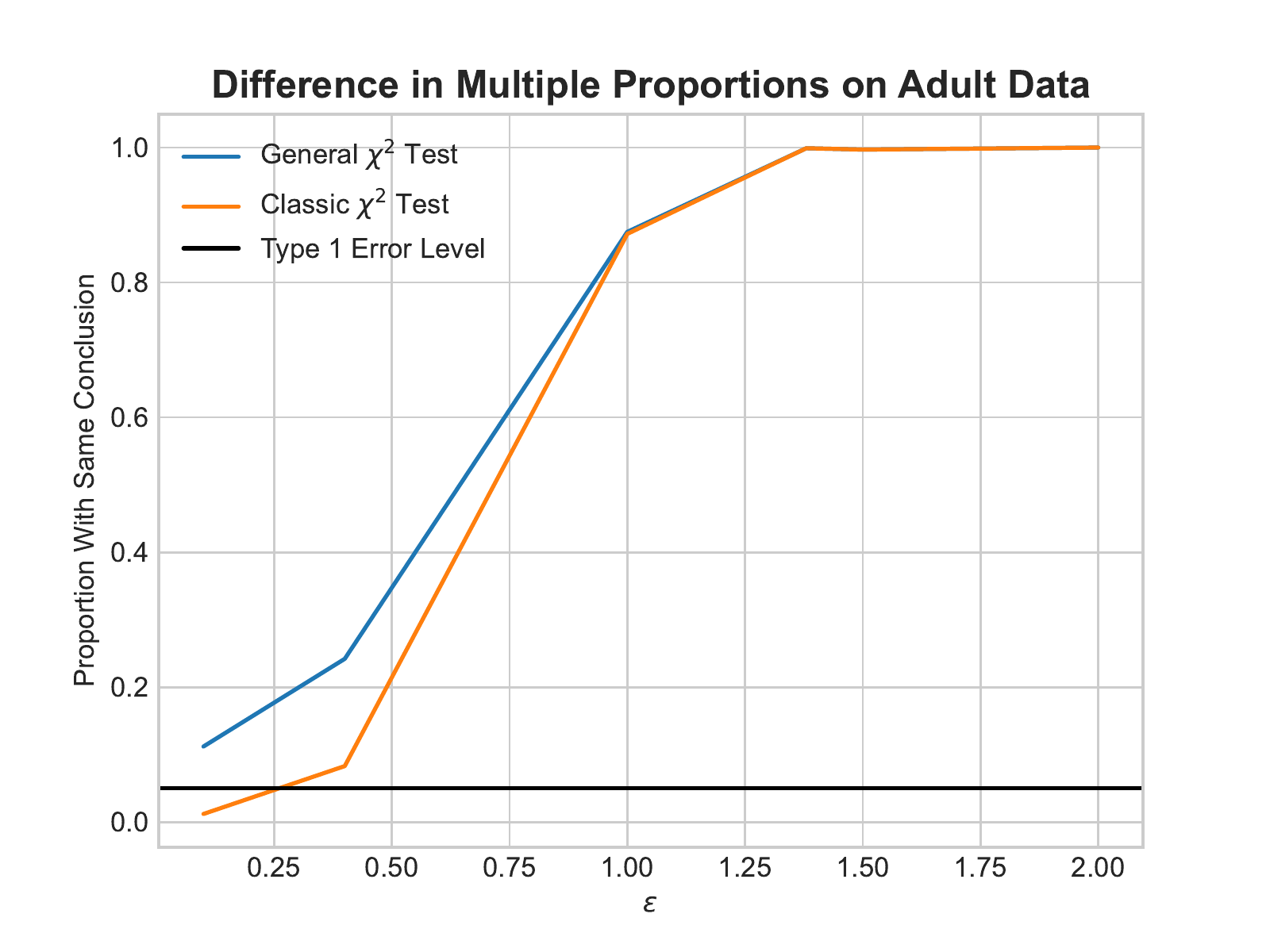}
  \caption{Proportion of times out of 1000 independent trials that we reject the null hypothesis of there being no difference in proportions across all groups on the UCI Adult dataset.  We use race as the sensitive group and the outcome is whether a sample makes more thant \$50k salary or not.}
  \label{fig:AdultRaceTest}
\end{figure}
\end{center}

%% file: tTest.tex
\section{Testing Difference in Two Means: t-tests}
Up to this point, we have dealt with binary outcomes, but now move on to real valued outcomes.  We first want to test whether the mean is different between two groups, with the following section testing differences across more than two groups. As in the earlier section on testing the difference between proportions, we will also be interested in confidence intervals for the difference between two means. Note that in the fully local DP setting, we would need to privatize the group membership of each individual as well as the outcome.  Because the outcome is real valued, we would need to introduce additional parameters and/or assumptions on the data in order to ensure privacy, such as clipping the magnitude of any outcome and then adding noise proportional to this clip parameter.  Such approaches may heavily bias the data and also add an additional parameter to set.  Setting clip values too large makes noise swamp the signal and setting clip values too small will heavily bias outcomes before noise is added, creating a bias-variance tradeoff.  In our less restrictive privacy model we do not need to make such considerations.

Typically, one would use the classical t-test to test the difference between two means between samples $\{X_i[j]\}_{i=1}^{n_j} \stackrel{i.i.d.}{\sim}  \Normal{\mu_j}{\sigma^2_j}$ for $j \in \{1,2\}$. That is, we use the t-test statistic defined as follows where $s_1^2, s_2^2$ are the sample variances for groups $1$ and $2$, respectively.
\[
T = \frac{ \frac{1}{n_1} \sum_{i=1}^{n_1} X_i[1] - \frac{1}{n_2} \sum_{i=1}^{n_2} X_i[2] }{\sqrt{s_1^2/n_1 + s_2^2/n_2}{}}
\]
We would then compare the test statistic to a t distribution, which converges to a standard Gaussian if $n_1, n_2$ are large.
When introducing privacy, we want to know whether the t-test can still be used or whether a different test should be used.\footnote{Observe that in a standard t-test, we assume that samples are i.i.d. Gaussian, and the sample mean and sample standard deviation are independent; once we have added privacy, these assumptions no longer hold. In practice, however, the t-test appears robust to these assumption violations.}  In order to use the general $\chi^2$ statistic when privacy is included, we will need to have $\chi^2$ test that can handle real outcomes, which is what we will cover next.

\subsection{A $\chi^2$ Test for Difference in Means}

For binary outcomes, we were able to move from the $Z$-test to a $\chi^2$ test, and in fact $\chi^2$ tests are typically used instead of $Z$-tests in many common statistical packages.  However, for continuous outcomes, there is not a standard $\chi^2$ version of the t-test.  We then present a way to formulate the t-test as a $\chi^2$ and show similar performance.  Note that for binary outcomes, we could form a contingency table where the rows were outcomes (success/failure) and the columns were groups (group 1 or 2).  To fit this framework, a first approach would be to discretize the outcomes into bins, and form the contingency table with $r$ different rows, where $r$ is a predetermined value for the number of bins the outcomes will be placed in.  Unfortunately, this introduces additional complexity to the hypothesis test when adding privacy, and it is not clear how to discretize the outcome set and how this might impact statistical power.  

Instead, we present a way to form a contingency table for continuous outcomes without discretizing, by considering the moments of the samples, which we present in Table~\ref{table:momentContingency}.  It is easy to see that the contingency table for binary outcomes in Table~\ref{table:contingencyTable}, where instead of having failure outcomes, we could replace it with the 0th order moment, which would give the marginals.  Recall that we use $W_i \sim \text{Bern}(\pi)$ to determine the group of sample $i$ and $X_i[j] \sim \Normal{\mu_j}{\sigma^2_j} $ for $j \in \{ 1,2\}$.  The entries in the table will suffice for estimating the population parameters $\mu_{1}, \mu_2 \in \R$, $\sigma_1, \sigma_2 >0$, and $\pi \in (0,1)$.
\begin{table}[htbp]
\centering\setcellgapes{4pt}\makegapedcells
\begin{tabular}{ |c|c|c| } 
 \hline
 Sample Orders & \shortstack{Group $0$ \\ w.p. $\pi$} & \shortstack{Group $1$ \\ w.p. $1-\pi$} \\ 
 \hline
$0$-th  & $ \sum_{i=1}^n W_i$ & $ \sum_{i=1}^n(1-W_i)$ \\ 
 \hline
$1$-st & $ \sum_{i=1}^n W_i X_{i}[1]$ & $ \sum_{i=1}^n(1-W_i)X_{i}[2]$ \\ 
 \hline
 $2$-nd & $ \sum_{i=1}^n W_i X_{i}^2[1]$ & $ \sum_{i=1}^n(1-W_i) X_{i}^2[2]$ \\ 
 \hline
\end{tabular}
\caption{Contingency Table for continuous outcomes $\{X_{i}[j]\}_{i=1}^n \stackrel{i.i.d.}{\sim} \Normal{\mu_j}{\sigma^2_j}$ with $j \in \{ 1,2\}$ and group variable $\{ W_i\}_{i=1}^n \stackrel{i.i.d.}{\sim} \text{Bern}(\pi)$.} \label{table:momentContingency}
\end{table}

With this setup, we then want to test $H_0: \mu_ 1 = \mu_2 + \Delta$, where $\Delta = 0$ is common.  We will assume that the standard deviation of each group is different for each other as well. We then consider the random vector $Y = \sum_{i=1}^n Y_i$ which will consist of the entries from the contingency table above.  Note that we do not require entries for both $\sum_{i=1}^nW_i$ as well as $\sum_{i=1}^n(1-W_i)$ as one can be written in terms of the other.  Hence, we have
\begin{equation}
Y = 
\begin{pmatrix}
Y[1] =& \sum_{i=1}^n W_i \\
Y[2] = & \sum_{i=1}^n W_i X_{i,0} \\
Y[3] = & \sum_{i = 1}^n (1-W_i) X_{i,1} \\
Y[4] = & \sum_{i=1}^n W_i X_{i,0}^2 \\
Y[5] = & \sum_{i=1}^n (1-W_i) X_{i,1}^2 \\
\end{pmatrix}
\label{eq:nonPrivateDiffMeanStat}
\end{equation}
We then consider the individual i.i,d. samples $Y_i$ where $Y = \sum_{i=1}^n Y_i$ so that we can compute the expectation of $Y_i$ under the null hypothesis $\mu_1 = \mu_2 + \Delta$ for some $\Delta \in \R$
\[
\vec{\theta}(\pi, \mu_1, \mu_2, \sigma_1, \sigma_2) = \E[Y_i] = \left( \pi, \pi \mu_1, (1-\pi) \mu_2, \pi \left( \mu_1^2 + \sigma^2_1\right), (1-\pi) \left(\mu_2^2 + \sigma^2_2 \right) \right)^\intercal
\]

Note that we can compute the covariance matrix and estimates for $\pi, \mu_1, \mu_2, \sigma_1, \sigma_2$, but it will help to simplify the $\chi^2$ test first.  In particular, when we write out the $\chi^2$ statistic, the minimization will lead to the second moment terms (the last two entries in $Y$) to contribute nothing to the $\chi^2$ value, since we can zero out those coordinates by setting $\sigma_1^2 = Y[4]/n - \mu_1^2$ as long as $\mu_1^2 \leq Y[4]/n$, and similarly for $\sigma_2$.  Note that if it does turn out that for a particular $\mu_1$ we have $Y[4]/n - \mu_1^2$, we should reject, i.e. return a large statistic of say $10 g$ or so, since that would mean that zero variance would be our best estimate for that group.  Hence, we will only consider $Y =  (Y[1], Y[2], Y[3])^\intercal$ in our test.  If it can be assumed that the variances are equal across groups, then we can keep the coordinates $Y[3], Y[4]$ in our test statistic.  Hence, we will continue with $Y_i$ denoting the first three coordinates of the random vector in \eqref{eq:nonPrivateDiffMeanStat}.

We next calculate the covariance matrix $C(\pi, \mu_1, \mu_2, \sigma_1, \sigma_2)$ for the first 3 coordinates in $Y_i$
\[
C(\pi, \mu_1, \mu_2, \sigma_1, \sigma_2) = 
\begin{bmatrix}
\pi(1-\pi) & \pi \mu_1 - \pi^2 \mu_1 & - \pi (1-\pi) \mu_2 \\
\pi \mu_1 ( 1- \pi \mu_1) & \pi (\mu_1^2 + \sigma_1^2) -\pi \mu_1 & - \pi \mu_1 (1-\pi) \mu_2 \\
 \pi (1-\pi) \mu_2 & - \pi \mu_1 (1-\pi) \mu_2 & (1- \pi) (\mu_2^2 + \sigma_2^2) (1- (1-\pi) \mu_2)
\end{bmatrix}
\]
Under the null hypothesis $\mu_1 = \mu_2 + \Delta$, we will use the following estimates
\begin{align*}
\hat{\pi} &= Y[1]/n, \quad \hat{\mu} = \frac{Y[2] + Y[3]}{n},  \quad \hat{\mu}_1 = \hat{\mu} + (1-\hat{\pi}) \Delta, \quad \hat{\mu}_2 = \hat{\mu} - \hat{\pi} \Delta.\\
\sigma^2_1 &= \frac{Y[4]/n}{\hat{\pi}} - \hat{\mu_1}^2, \quad \sigma^2_2 = \frac{Y[5]/n}{1-\hat{\pi}} - \hat{\mu_2}^2
\end{align*}

We are now ready to calculate the $\chi^2$ statistic
\begin{equation}
D = \min_{\substack{\pi \in (0,1), \\ \mu_1, \mu_2 : \mu_1 = \mu_2 + \Delta}} \left\{ 
\begin{pmatrix}
Y[1]/n - \pi \\
 Y[2]/n - \pi \mu_1 \\
 Y[3]/n - (1-\pi) \mu_2  
 \end{pmatrix} ^\intercal 
 C(\hat{\pi}, \hat{\mu_1}, \hat{\mu_2}, \hat{\sigma_1}, \hat{\sigma_2})^{-1}
  \begin{pmatrix}
Y[1]/n - \pi \\
 Y[2]/n - \pi \mu_1 \\
 Y[3]/n - (1-\pi) \mu_2  
 \end{pmatrix} 
 \right\}
 \label{eq:diffMeanChiSqStat}
\end{equation}
We will evaluate the test statistic $D$ against a $\chi^2$ with $(3 - 2) = 1$ degree of freedom. We now compare this $\chi^2$ based hypothesis test with the traditional t-test described in the previous subsection, at a significance level of 95\%.  We consider the null hypothesis $H_0: \mu_1 = \mu_2 = \mu$ and generate data in two groups with equal variances.  First, we modify the shared mean $\mu$ as well as the common variance and the probability $\pi$ of being in group $1$.  The left plot of Figure~\ref{fig:tVsChi} shows confidence intervals using the standard $t$-test statistic (assuming unequal variances) and the general $\chi^2$ statistic using an approach similar to when we had binary outcomes in Section~\ref{sect:CI_prop}.  When computing the left and right end points of the confidence interval for proportions, we could simply consider $\Delta \in [-1,1]$.  However, for differences in means, we will use the data to determine lower and upper bounds on the candidate confidence interval region.   We can simply take the sample mean in both groups and add or subtract say 10 standard deviations within each group. 
The results show overlapping confidence intervals.  
The right plot of Figure~\ref{fig:tVsChi} shows the average number of times over 1000 trials that each test rejected the null hypothesis $H_0: \mu_1 = \mu_2$ as we change the difference $\mu_1 - \mu_2 > 0$, so that we should reject more frequently.  The proportion of null hypotheses rejected is indistinguishable between the two tests.

 \begin{figure}
  \begin{center}
\includegraphics[width=0.45\textwidth]{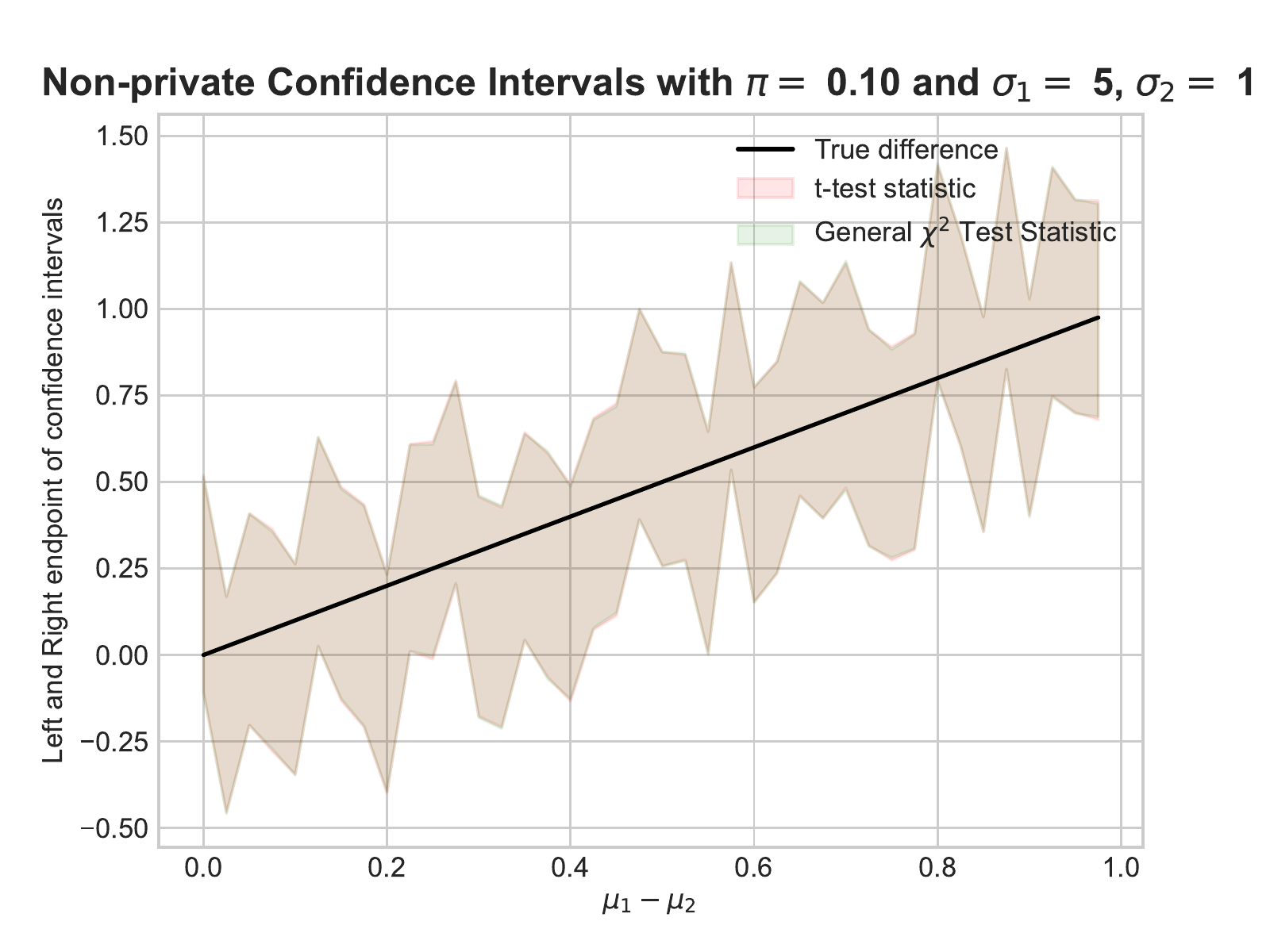}
\includegraphics[width=0.45\textwidth]{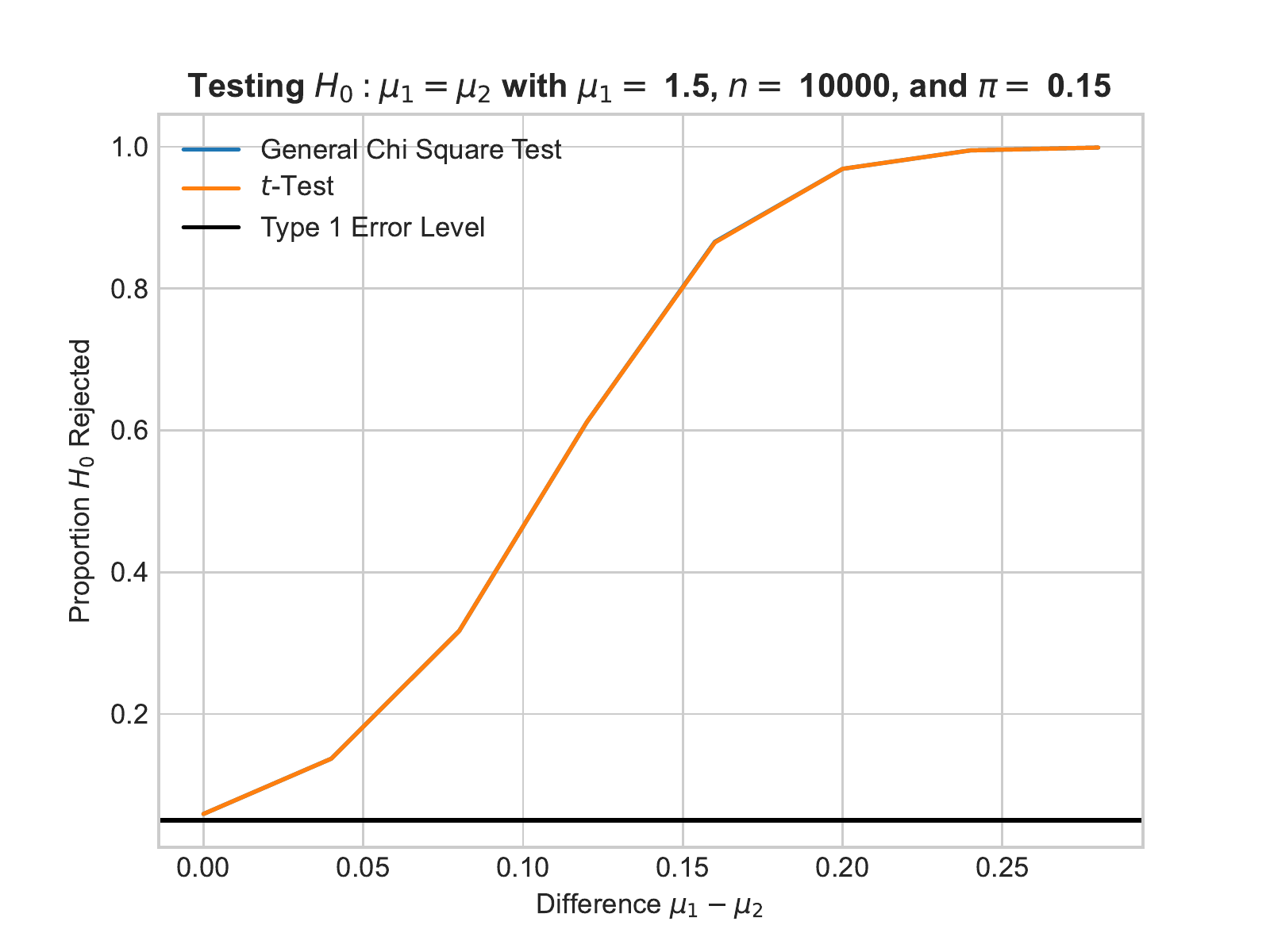}
\caption{(Left) Comparing confidence intervals using the standard t-test statistic and the general $\chi^2$ statistic in \eqref{eq:diffMeanChiSqStat} with $\mu_2 = 0$.  The bounds are completely overlapping, resulting in identical tests. (Right) Comparing power curves of the t-test and the general $\chi^2$ test for $H_0: \mu_1 = \mu_2$, which are on top of each other. In this case we use $\sigma_1 = 2, \sigma_2 = 1$ when we generate data.}
\label{fig:tVsChi}
\end{center}
\end{figure}

\subsection{Private Tests for Difference in Means}\label{sect:privateTTest}
Recall that we use the traditional randomized response when privatizing group $j \in \{1,2 \}$.  We will use the general $\chi^2$ approach to design a private test for differences between two means.  This test can then be used to calculate confidence intervals, in the same way that confidence intervals were derived from the $\chi^2$ statistic for binary outcomes in Section~\ref{sect:CI_prop}.  We will then consider the privatized version of the random vector $Y$ from \eqref{eq:nonPrivateDiffMeanStat}, but we will only use the first 3 coordinates, as the last two coordinates will be eliminated in the minimization of the statistic, as we mentioned earlier.  

For randomized response, we will write $Z^\diffp_{i}[j] \sim \text{Bern}(\tfrac{e^\diffp}{e^\diffp + 1})$ for $i \in [n]$ and $j \in \{1,2\}$, and then write the privatized vector $Y_i^\diffp$ in terms of $Z^\diffp_i[j]$,

\begin{align*}
Y_i^\diffp & = 
\begin{pmatrix} Z^\diffp_{i}[1] \cdot W_i  + (1-Z^\diffp_{i}[2]) \cdot (1-W_i) \\  
			Z^\diffp_{i}[1] \cdot W_i \cdot X_{i}[1] + (1-Z^\diffp_{i}[2]) \cdot (1-W_i) \cdot X_{i}[2] \\
			(1-Z^\diffp_{i}[1])\cdot W_i \cdot X_{i}[1] + Z^\diffp_{i}[2] \cdot (1-W_i) \cdot X_{i}[2]
\end{pmatrix} 
\end{align*}
Next, we compute its expectation $\E[Y_i^\diffp] = \vec{\theta}^\diffp(\pi,\mu_1, \mu_2)$ where
\[
\vec{\theta}^\diffp(\pi,\mu_1, \mu_2) = 
\begin{pmatrix}  
\tfrac{e^\diffp}{e^\diffp + 1}  \pi  +  \tfrac{1}{e^\diffp + 1} \left(1-\pi\right) \\
\tfrac{e^\diffp}{e^\diffp + 1}  \pi \mu_1 +  \tfrac{1}{e^\diffp + 1} \left(1-\pi\right) \mu_2\\
\tfrac{1}{e^\diffp + 1}  \pi \mu_1 +  \tfrac{e^\diffp}{e^\diffp + 1} \left(1-\pi\right) \mu_2\\
\end{pmatrix}
\]

We can then compute $Y_i^\diffp$ covariance matrix $C(\pi, \mu_1, \mu_2, \sigma_1, \sigma_2; \diffp)$.  
\begin{align*}
& C(\pi, \mu_1, \mu_2, \sigma_1, \sigma_2; \diffp)  = \E\left( Y_i^\diffp (Y_i^\diffp)^\intercal \right) - \E\left[ Y_i^\diffp \right] \E\left[ Y_i^\diffp \right]^\intercal \\
& \qquad = \frac{1}{e^\diffp + 1}
\begin{bmatrix}
e^\diffp \pi  + \left(1-\pi\right) & \pi e^\diffp \mu_1 + (1-\pi) \mu_2  & 0 \\
\pi e^\diffp \mu_1 + (1-\pi) \mu_2 & \pi e^\diffp (\mu_1^2 + \sigma_1^2)  + (1-\pi) (\mu_2^2 + \sigma_2^2) & 0 \\
 0 & 0 & \pi (\mu_1^2 + \sigma_1^2)  + (1-\pi) e^\diffp (\mu_2^2 + \sigma_2^2) 
\end{bmatrix} \\
& \qquad \qquad 
- \vec{\theta}^\diffp(\pi,\mu_1, \mu_2) \vec{\theta}^\diffp(\pi,\mu_1, \mu_2)^\intercal
\end{align*}

Now, we need to assign estimates for the parameters $\pi, \mu_1, \mu_2, \sigma_1, \sigma_2$. Here, we have the null hypothesis $H_0: \mu_1 = \mu_2 + \Delta$ and sample data $Y^\diffp$. Beginning with $\hat{\pi}, \hat{\mu}_1$, and $\hat{\mu}_2$, we use 
\begin{align*}
\hat{\pi} &= (e^\diffp + 1) \left( \frac{Y^\diffp[1]/n - \tfrac{1}{e^\diffp + 1}}{e^\diffp - 1} \right) \\
\begin{bmatrix}
\hat{\pi} \tfrac{e^{\diffp}}{e^\diffp + 1} & (1-\hat{\pi}) \tfrac{1}{e^\diffp + 1} \\ 
\hat{\pi} \tfrac{1}{e^\diffp + 1} & (1-\hat{\pi}) \tfrac{e^\diffp}{e^\diffp + 1}
\end{bmatrix} 
\begin{pmatrix}
\hat{\mu}_1 \\
\hat{\mu}_2
\end{pmatrix} 
& = \begin{pmatrix}
Y^\diffp[2]/n \\
Y^\diffp[3] /n 
\end{pmatrix}
\end{align*}
We then use the null hypothesis to replace $\hat{\mu}_1 = \hat{\mu}_2 + \Delta$ and then solve the over constrained system of equations via least squares in which case we have
\[
\hat{\mu}_2 = \frac{\hat{\pi} \tfrac{e^{\diffp}}{e^\diffp + 1}  \left( Y^\diffp[2]/n - \hat{\pi} \tfrac{e^\diffp}{e^\diffp + 1} \Delta \right)  +(1-\hat{\pi}) \tfrac{1}{e^\diffp + 1} \left( Y^\diffp[3]/n - \hat{\pi} \tfrac{1}{e^\diffp + 1} \Delta \right)}{\left(\hat{\pi} \tfrac{e^{\diffp}}{e^\diffp + 1} \right)^2 + \left( (1-\hat{\pi}) \tfrac{1}{e^\diffp + 1} \right)^2}
\]

For $\sigma_1, \sigma_2$, we will use different estimates in the two diagonal entries in which they appear, corresponding to the estimation sample. We replace the $\sigma_1$ in the $C[2,2]$ diagonal entry with the sample variance $s_1^2$ of $\{Y_{i}^\diffp[2] \}_{i=1}^n$. Analogously, for the $C[3,3]$ diagonal entry, we replace $\sigma_2$ with the sample variance $s_2^2$ of $\{Y_i^\diffp[3] \}_{i=1}^n$.  We do point out that under the estimates $\hat{\mu}_1 = \hat{\mu}_2 + \Delta$, there might not be a possible $\sigma_1, \sigma_2\geq 0$ that can achieve the sample variances $s_1^2, s_2^2$.  We then compute the theoretical variance of our observations $Y^\diffp[2], Y^\diffp[3]$
\begin{align*}
\textrm{Var}(Y^\diffp[2]) & =  \mu_1^2 \left(1- \tfrac{\pi e^\diffp}{e^\diffp +1}\right)  \tfrac{\pi e^\diffp}{e^\diffp +1} + \mu_2^2 \left(1- \tfrac{(1-\pi)}{e^\diffp +1}\right) \tfrac{(1-\pi)}{e^\diffp +1}  - 2 \tfrac{ \pi (1-\pi) e^\diffp}{(e^\diffp + 1)^2} \mu_1\mu_2 + \tfrac{\pi e^\diffp}{e^\diffp + 1} \sigma_1^2 +  \tfrac{(1-\pi)}{e^\diffp + 1} \sigma_2^2   \\
\textrm{Var}(Y^\diffp[3]) & =  \mu_1^2 \left(1- \tfrac{\pi}{e^\diffp +1}\right) \tfrac{\pi}{e^\diffp +1} + \mu_2^2 \left(1- \tfrac{(1-\pi)e^\diffp}{e^\diffp +1}\right)  \tfrac{(1-\pi)e^\diffp}{e^\diffp +1}  - 2 \tfrac{ \pi (1-\pi)e^\diffp}{(e^\diffp + 1)^2} \mu_1\mu_2 + \tfrac{\pi}{e^\diffp + 1} \sigma_1^2 +  \tfrac{(1-\pi)e^\diffp}{e^\diffp + 1} \sigma_2^2 
\end{align*} 

Because $\sigma_j^2 \geq 0$, we need to ensure our estimate $s_j$ for $\textrm{Var}(Y^\diffp[j])$ for $j \in \{1,1\}$.  Thus we check if the following inequalities are satisfied.
\begin{align*}
s_1^2 & \geq  \hat{\mu}_1^2 \left(1- \tfrac{\hat{\pi} e^\diffp}{e^\diffp +1}\right)  \tfrac{\hat{\pi} e^\diffp}{e^\diffp +1} + \hat{\mu}_2^2 \left(1- \tfrac{(1-\hat{\pi})}{e^\diffp +1}\right) \tfrac{(1-\hat{\pi})}{e^\diffp +1}  - 2 \tfrac{ \hat{\pi} (1-\hat{\pi}) e^\diffp}{(e^\diffp + 1)^2} \hat{\mu}_1 \hat{\mu}_2
\\
s_2^2 & \geq  \hat{\mu}_1^2 \left(1- \tfrac{\hat{\pi}}{e^\diffp +1}\right)  \tfrac{\hat{\pi} }{e^\diffp +1} + \hat{\mu}_2^2 \left(1- \tfrac{(1-\hat{\pi})e^\diffp}{e^\diffp +1}\right) \tfrac{(1-\hat{\pi})e^\diffp}{e^\diffp +1}  - 2 \tfrac{ \hat{\pi} (1-\hat{\pi}) e^\diffp}{(e^\diffp + 1)^2} \hat{\mu}_1 \hat{\mu}_2
\end{align*}
If they are not satisfied, we replace $s_j^2$ with the corresponding right hand side, essentially using $\sigma_1 = \sigma_2 = 0$ in our estimate.  

Putting this all together, we have the following $\chi^2$ statistic, $D^\diffp$, which will use $Y^\diffp$ instead of $Y$ in \eqref{eq:nonPrivateDiffMeanStat} which we compare to a $\chi^2$ with 1 degree of freedom to base our hypothesis test.
\begin{equation}
D^\diffp = 
 \min_{\substack{\pi \in (0,1), \\ \mu_1, \mu_2 : \mu_1 = \mu_2s + \Delta}} \left\{ 
\left(  Y^\diffp  - \vec{\theta}^\diffp\left(\pi, \mu_1, \mu_2 \right) \right)^\intercal 
 C(\hat{\pi}, \hat{\mu_1}, \hat{\mu_2}, \hat{\sigma_1}, \hat{\sigma_2})^{-1}
\left( 
Y^\diffp  - \vec{\theta}^\diffp\left(\pi, \mu_1, \mu_2 \right)
\right)
 \right\}
 \label{eq:tTestChiSqStat}
\end{equation}
where 
\[
Y^\diffp  - \vec{\theta}^\diffp\left(\pi, \mu_1, \mu_2 \right) = 
\begin{pmatrix}
Y^\diffp[1]/n - \tfrac{e^\diffp}{e^\diffp + 1}  \pi  +  \tfrac{1}{e^\diffp + 1} \left(1-\pi\right) \\
 Y^\diffp[2]/n - \tfrac{e^\diffp}{e^\diffp + 1}  \pi \mu_1 +  \tfrac{1}{e^\diffp + 1} \left(1-\pi\right) \mu_2 \\
 Y^\diffp[3]/n - \tfrac{1}{e^\diffp + 1}  \pi \mu_1 +  \tfrac{e^\diffp}{e^\diffp + 1} \left(1-\pi\right) \mu_2  
 \end{pmatrix} .
\]

\subsection{Results}
We now present results of our privatized tests for testing the difference in means using our general $\chi^2$ framework.  In Figure~\ref{fig:tTestPowerComparisons}, we compare our test with the naive approach of using the t-test as if no privacy has been introduced, as well as the baseline t-test on non-privatized groups.  We chose a fairly extreme setting of parameters in the data distribution to show that the $\chi^2$ approach can achieve similar power to the classical t-test approach, while in more symmetric settings, i.e. $\sigma_1 \approx \sigma_2$ and $\pi \approx 0.5$, they perform similarly.  We apply the general rule of thumb\footnote{There are similar rule of thumb in traditional statistical tests, where if a group size is too small then the test is inconclusive.}  that if at any point our estimate group size $\hat{\pi} \cdot n$ in either group is less than 5, we simply return a zero statistic and hence fail to reject the null.
 \begin{figure}
  \begin{center}
\includegraphics[width=0.45\textwidth]{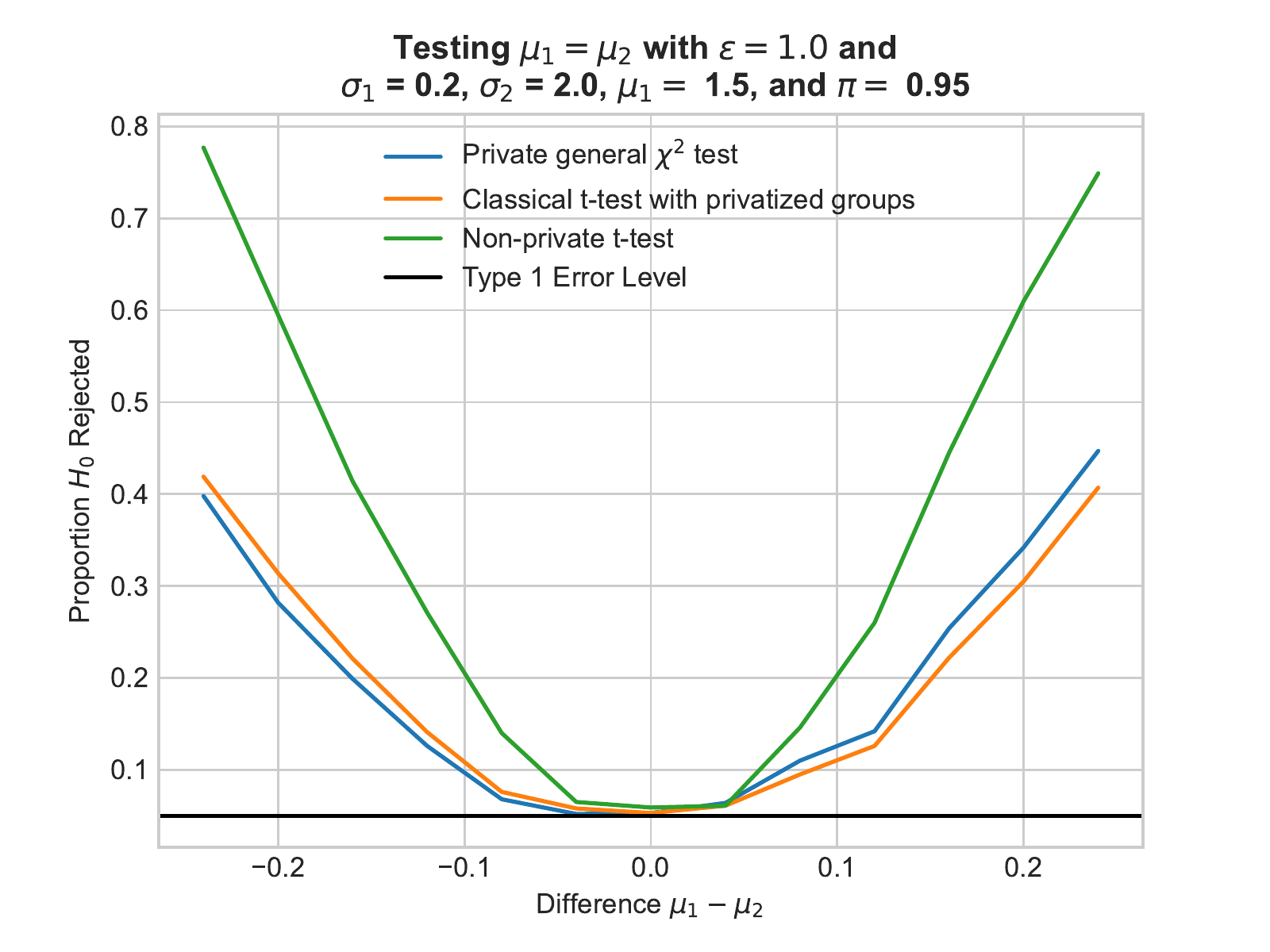}
\includegraphics[width=0.45\textwidth]{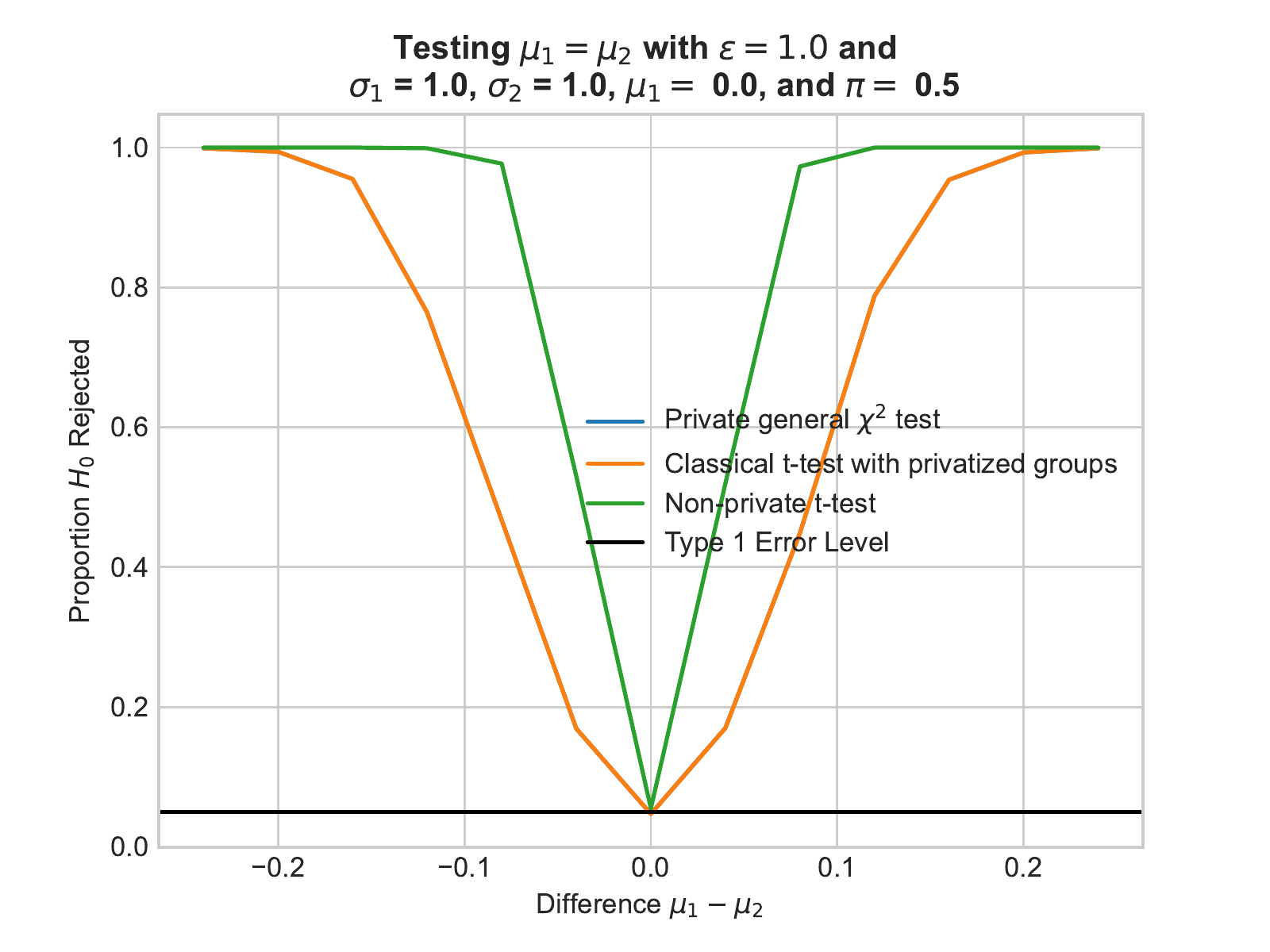}
\caption{Comparing power of t-test without modification on privatized data and the $\chi^2$ test with test statistic given in \eqref{eq:tTestChiSqStat} with $\epsilon = 1.0$ and $n = 10000$ samples. The left plot has distribution parameters $\pi = 0.05, \mu_1 = 1.5$, $\sigma_1 = 2$ and $\sigma_2 = 0.2$., while the right plot has distribution parameters  $\pi = 0.5, \mu_1 = 0$, $\sigma_1 = 1$ and $\sigma_2 = 1$ }
\label{fig:tTestPowerComparisons}
\end{center}
\end{figure}

We also present results on the confidence intervals of the difference in means, as we did in Section~\ref{sect:CI_prop} for binary outcomes.  As previously stated, we will use a similar approach to the binary outcome case and use the $\chi^2$ statistic for various $\Delta = \mu_1 - \mu_2$.  

For the t-test based confidence intervals, we correct the difference in means in the same way as in \eqref{eq:DeltaEpsilon}, and we compare the confidence intervals we get with the general $\chi^2$ approach. We present our results in Figure~\ref{fig:tTestConfidenceIntervals}, where we fix $\mu_1 = 0$, $\diffp =1$, and $n = 1000$, while we vary $\mu_1 - \mu_2$ in each plot and change $\pi, \sigma_1, \sigma_2$ in the different plots.  Note that the confidence intervals from the general $\chi^2$ statistics sometimes produces wider confidence intervals, although they are very close to the confidence intervals from the t-test statistic with a correction.  

 \begin{figure}
  \begin{center}
\includegraphics[width=0.31\textwidth]{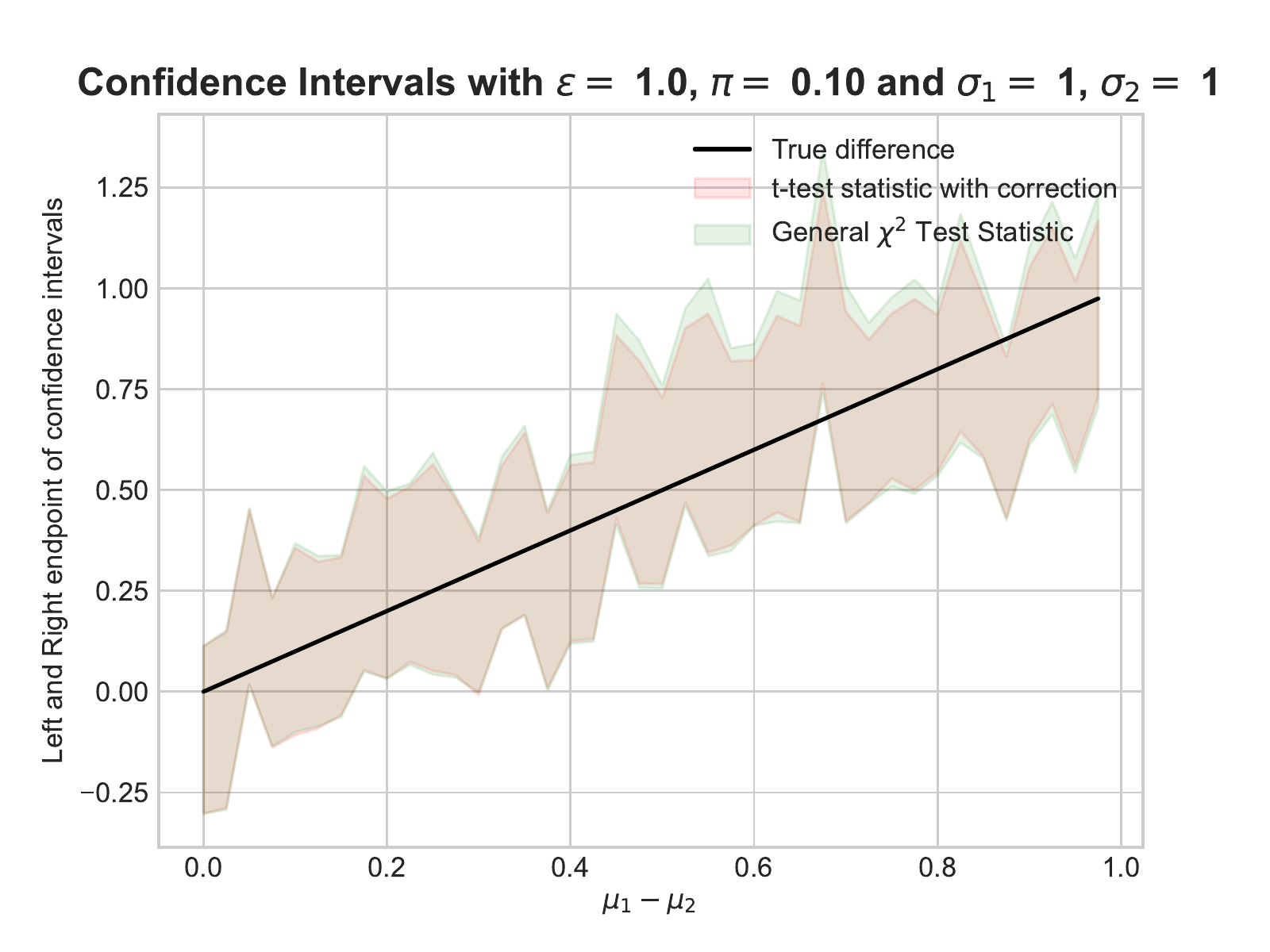}
\includegraphics[width=0.31\textwidth]{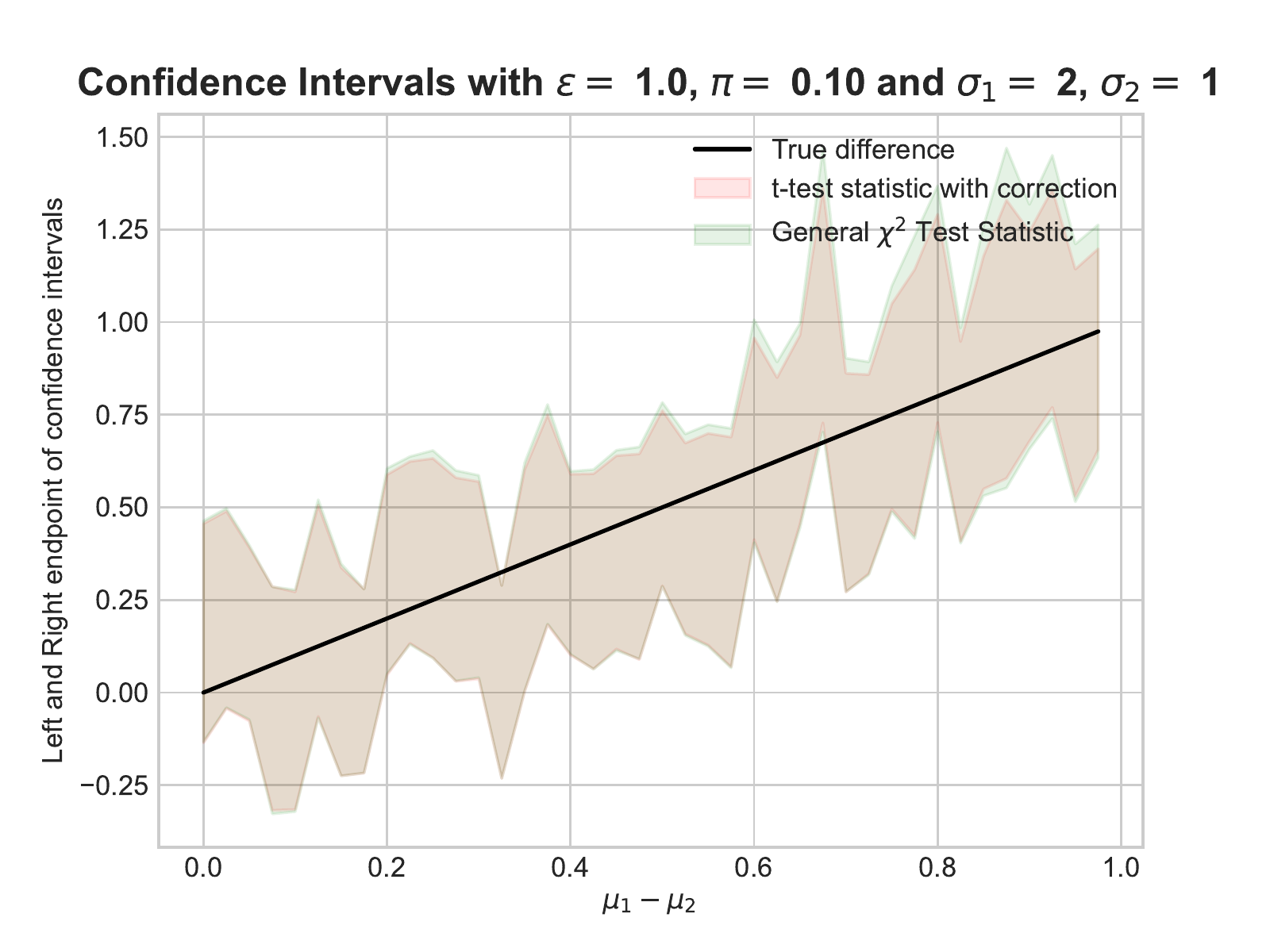}
\includegraphics[width=0.31\textwidth]{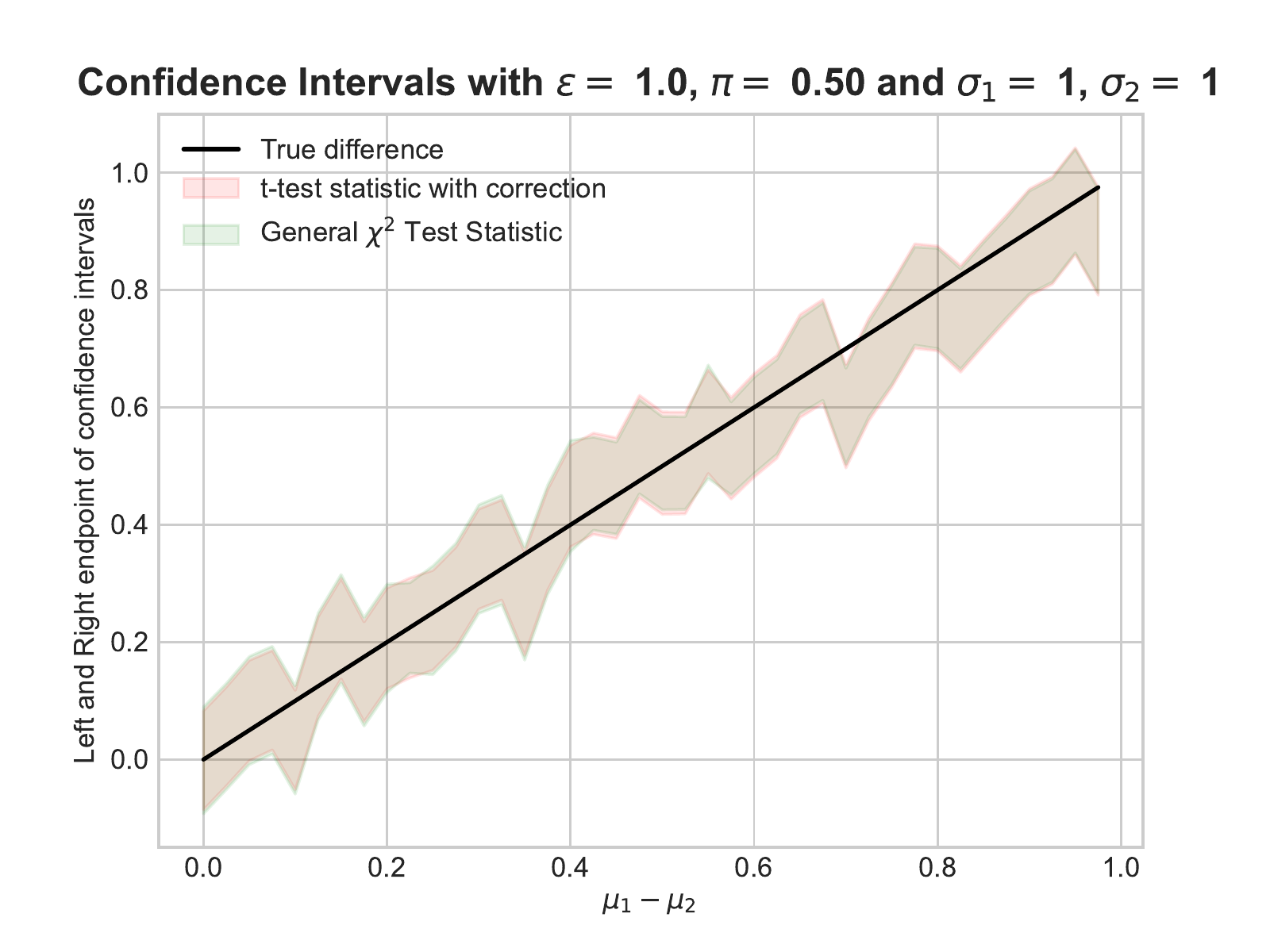}
\caption{Comparing confidence intervals using either the t-test or with the approach we outline for the $\chi^2$ statistic.  We use $\mu_1 = 0$,  $\epsilon = 1.0$, and $n = 10000$ for the plots and change $\sigma_1, \sigma_2$ as well as $\pi$.}
\label{fig:tTestConfidenceIntervals}
\end{center}
\end{figure}

We also present results in how often the various approaches provide confidence intervals that overlap the true difference in means in Figure~\ref{fig:tTestPropCI}.  In each case we privatize the data using randomized response on the group and compute confidence intervals using the classical t-test with a correction from \eqref{eq:DeltaEpsilon} and confidence intervals produced with the general $\chi^2$ approach.  We see that the general $\chi^2$ approach produces slightly more conservative confidence intervals as the privacy parameter increases and we consistently achieve close to the target $5\%$ proportion of missing the true difference in means.  
 \begin{figure}
  \begin{center}
\includegraphics[width=0.31\textwidth]{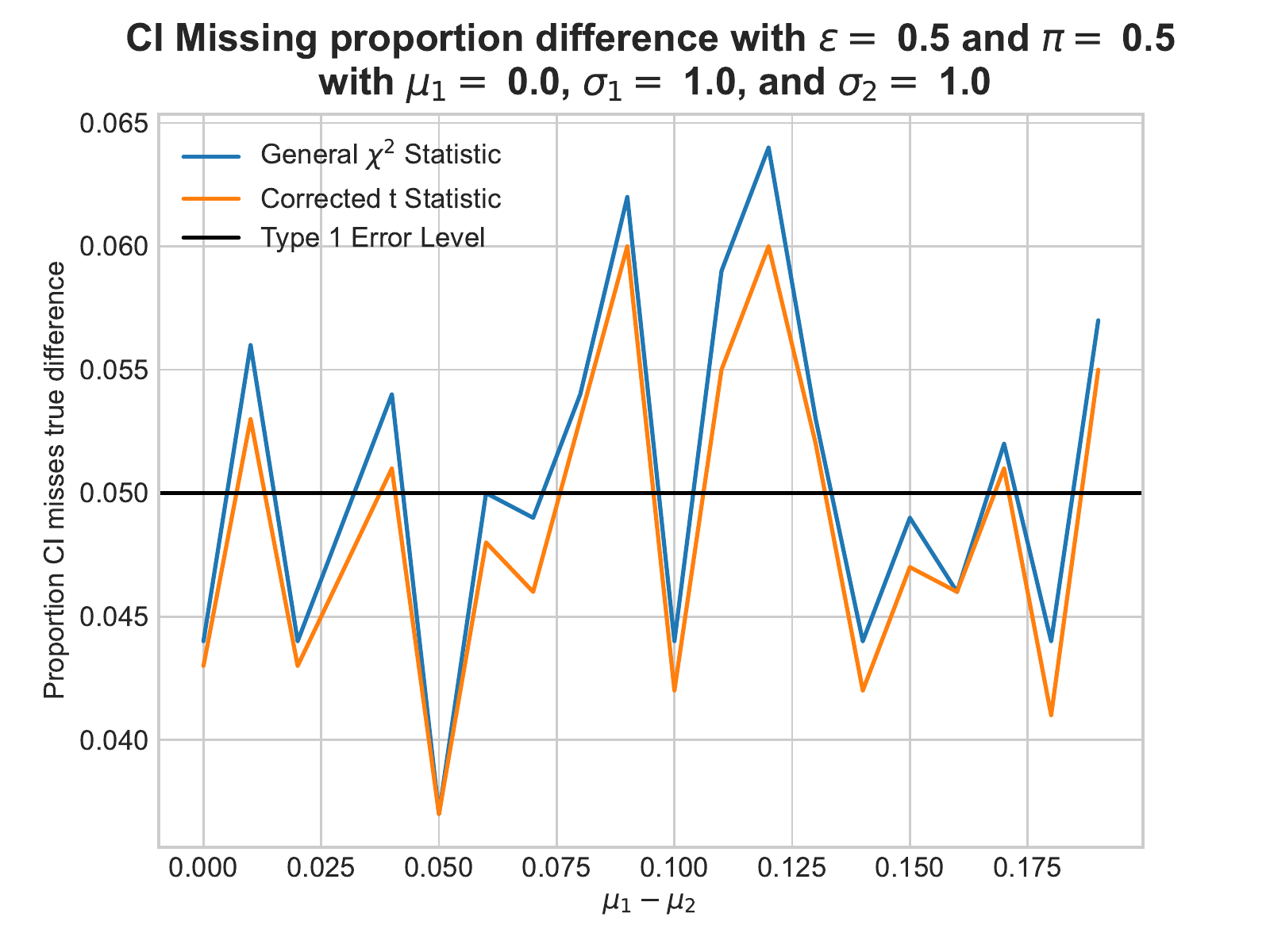}
\includegraphics[width=0.31\textwidth]{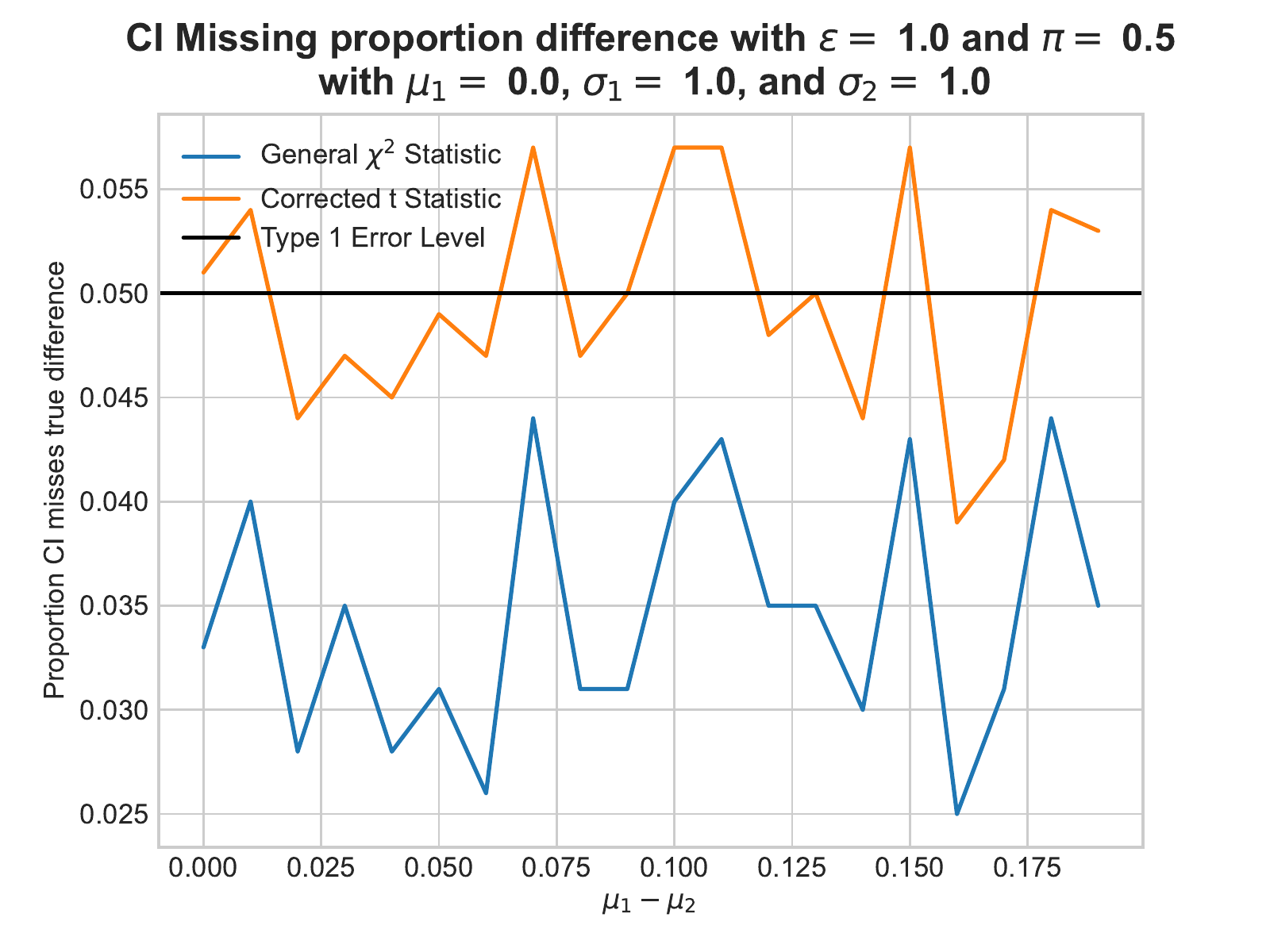}
\includegraphics[width=0.31\textwidth]{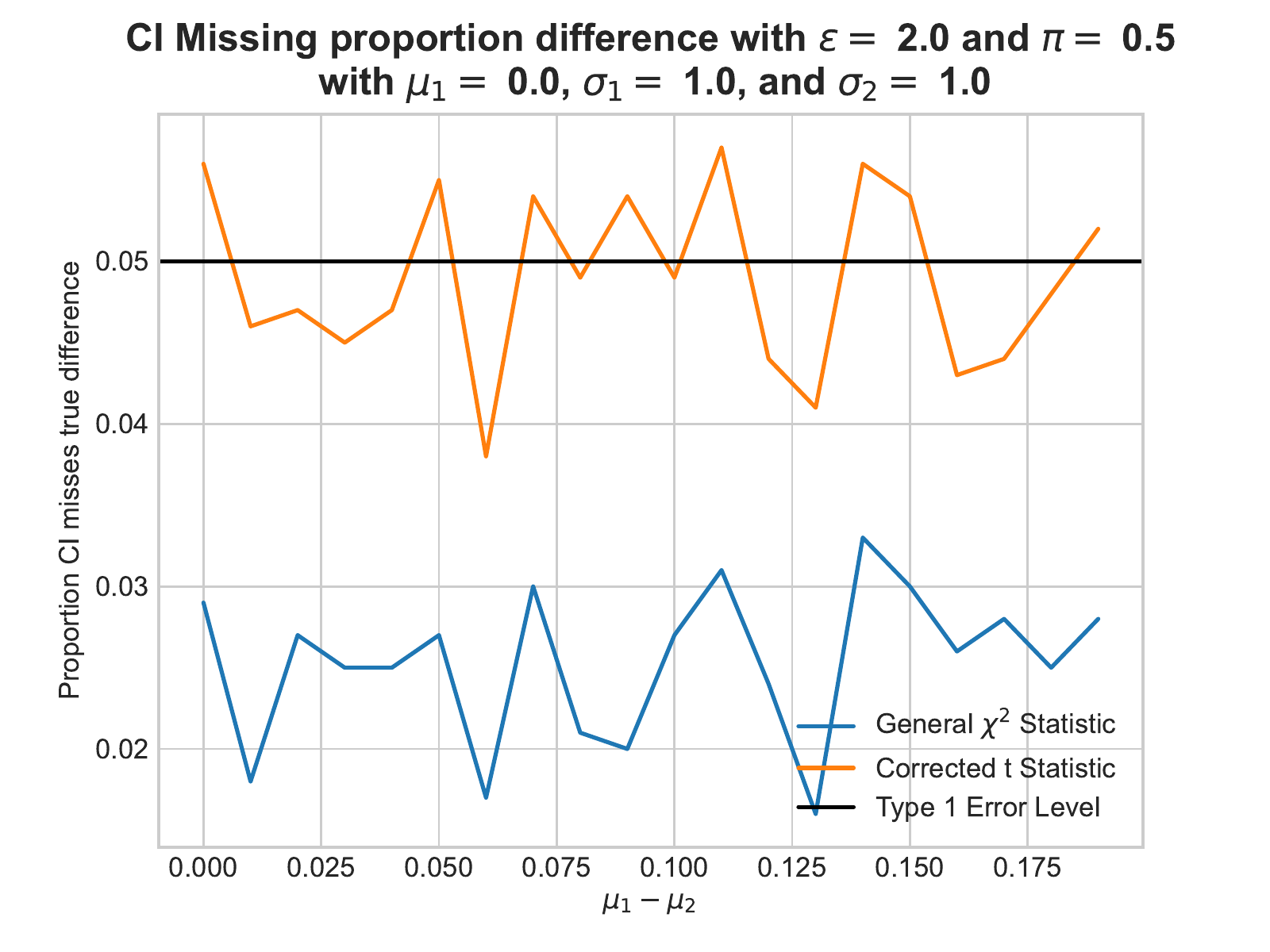}
\includegraphics[width=0.31\textwidth]{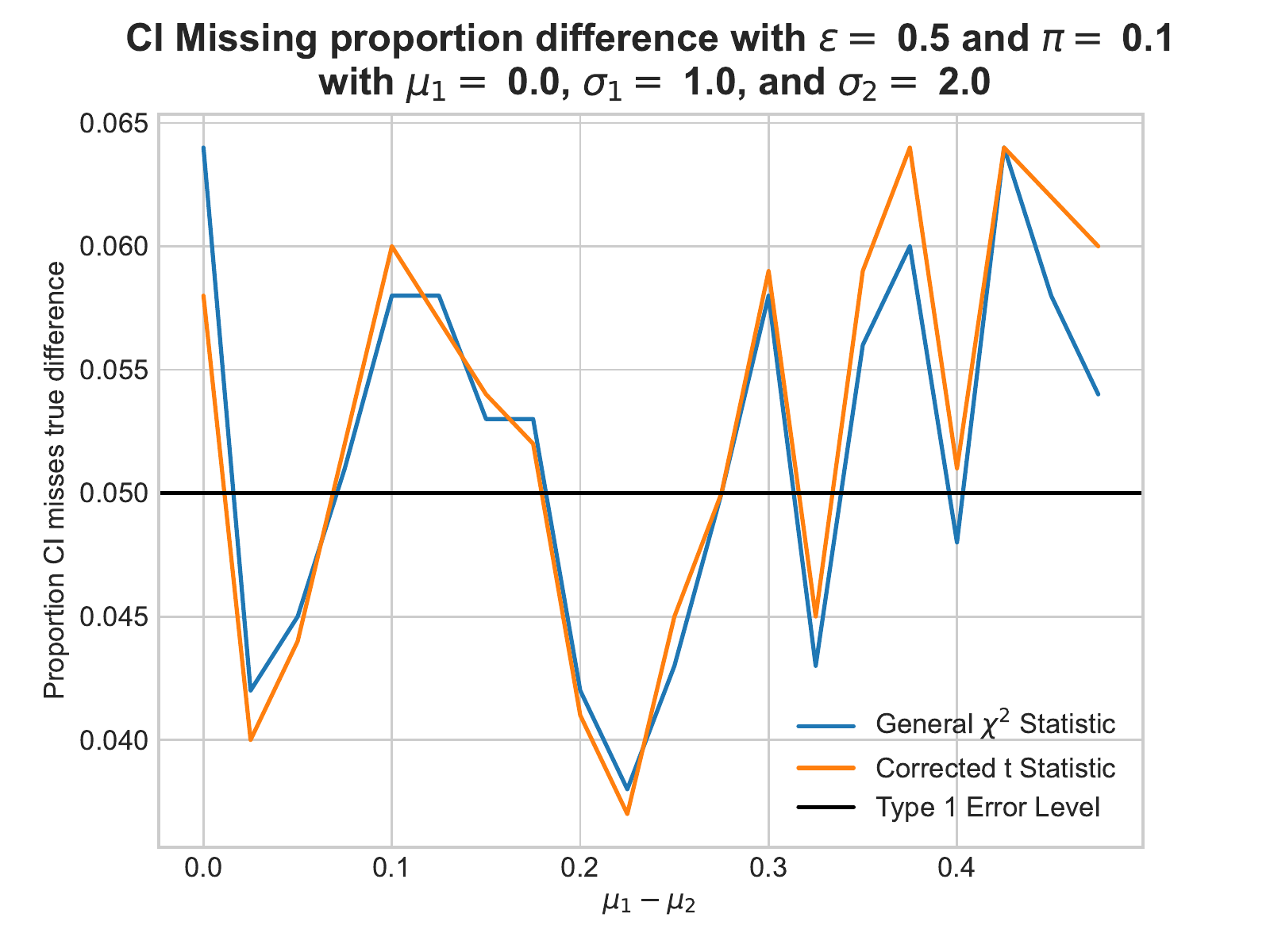}
\includegraphics[width=0.31\textwidth]{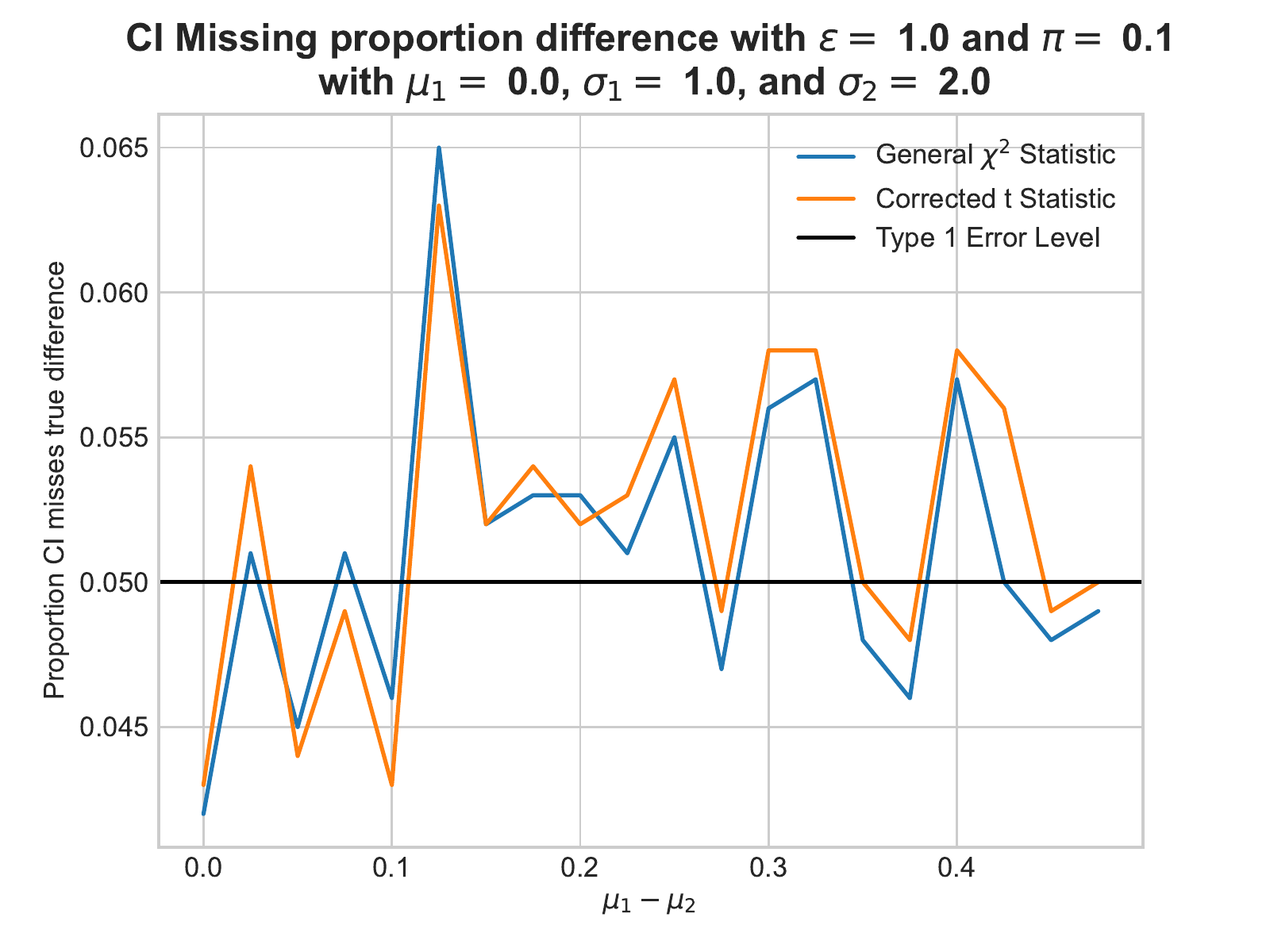}
\includegraphics[width=0.31\textwidth]{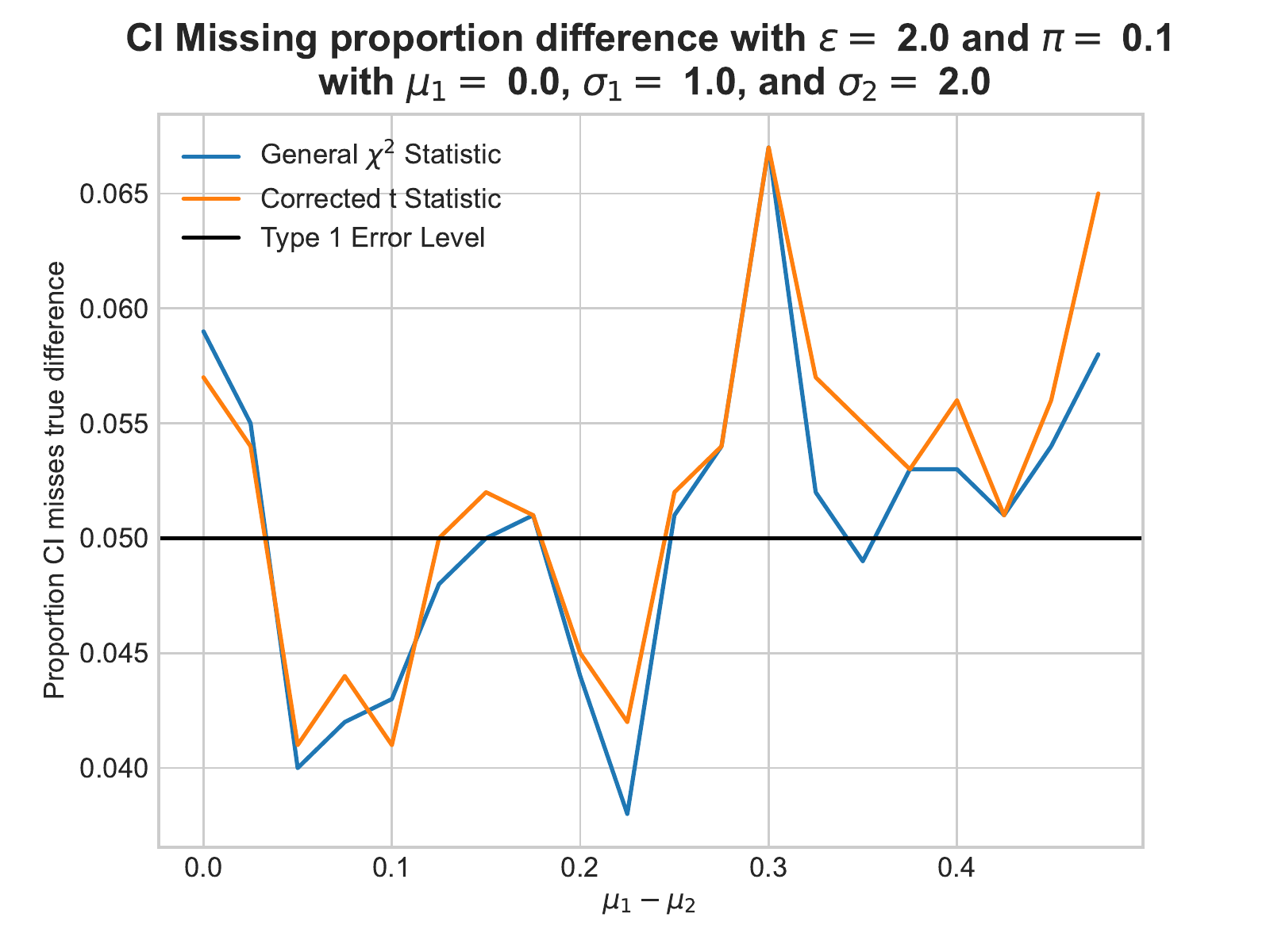}
\caption{We plot the proportion of times different tests will produce confidence intervals that miss the true difference $\mu_1 - \mu_2$.  The general $\chi^2$ test that accounts for the privacy mechanism consistently achieves the target type I error of $5\%$.  We compare our approach with using the classical t-test with correction in \eqref{eq:DeltaEpsilon}.  All plots use data size $n = 10000$ over 1000 trials.}
\label{fig:tTestPropCI}
\end{center}
\end{figure}

%% file: ANOVA.tex
\section{Testing Differences in Means across Several Groups}

We now consider testing whether there is a difference in means across $g > 2$ groups.  That is, let $X_{i}[j]$ be sampled i.i.d. from $\Normal{\mu_j}{ \sigma^2_j}$ for $i \in [n]$ with $j \in [g]$ and we test $H_0: \mu_1 = \mu_2 = \cdots \mu_g$.  In this case, we would perform a one-way ANOVA test, which makes the assumption that all $\sigma_j$ are equal, i.e. $\sigma_j = \sigma$ for all $j \in [g]$.  The one-way ANOVA test compares $H_0$ to the alternative hypothesis $H_1$: not all means are equal. Observe that the one-way ANOVA does not allow us to conclude specifically {\it which} mean or means may be different across the groups, and would usually be followed up with either a set of t-tests or a method like Tukey's method to formalize the comparison, or a reasonable conclusion from clear deviations observed in the data \cite{rice2006mathematical}. 

It is straightforward to fit this hypothesis test to our generalized $\chi^2$ test framework by considering the variable $W_i$, which will determine the group that sample $i$ is in. That is, $W_i \sim \text{Multinomial}(n,\pi)$ for $i \in [n]$ and $\pi \in [0,1]^g$ is a probability vector.  We can then generalize Table~\ref{table:momentContingency} for multiple groups in Table~\ref{table:momentContingencyANOVA}, which we will use in our privacy model.    

\begin{table}[htbp]
\centering\setcellgapes{4pt}\makegapedcells
\begin{tabular}{ |c|c|c|c|c| } 
 \hline
 Sample Moments & \shortstack{Group $1$ \\ w.p. $\pi_1$} & \shortstack{Group $2$ \\ w.p. $\pi_2$} & $\cdots$ & \shortstack{Group $g$ \\ w.p. $\pi_g$} \\ 
 \hline
$0$-th  & $ \sum_{i=1}^n W_{i}[1]$ & $ \sum_{i=1}^n W_{i}[2]$ & $\cdots$ & $ \sum_{i=1}^n W_{i}[g]$ \\ 
 \hline
$1$-st & $ \sum_{i=1}^n W_{i}[1] \cdot X_{i}[1]$ & $ \sum_{i=1}^n W_{i}[2]  \cdot X_{i}[2]$ & $\cdots$ & $ \sum_{i=1}^n W_{i}[g] \cdot X_{i}[g]$ \\ 
 \hline
 $2$-nd & $ \sum_{i=1}^n W_{i}[1] \cdot X_{i}^2[1]$ & $ \sum_{i=1}^n W_{i}[2]\cdot  X_{i}^2[2]$ & $\cdots$ & $ \sum_{i=1}^n W_{i}[g] \cdot X_{i}^2[g]$ \\
 \hline
\end{tabular}
\caption{Contingency Table for continuous outcomes $\{X_{i}[j]\}_{i=1}^n \stackrel{i.i.d.}{\sim} \text{N}(\mu_j, \sigma^2_j)$ with $j \in [g]$ and group variable $\{ W_i\}_{i=1}^n \stackrel{i.i.d.}{\sim} \text{Multinomial}(1,\pi)$.} \label{table:momentContingencyANOVA}
\end{table}

We can then use our general $\chi^2$ framework to design a $\chi^2$ statistic based on the data in each cell of Table~\ref{table:momentContingencyANOVA} in a similar way to testing the difference in means.  Recall that for the difference in two means, we only considered the 0th and 1st sample orders in $Y$, which we will also consider.  We use the random vector $Y = \sum_{i}^n Y_i$ in our test and then compute the expectation of $Y_i = (Y_i^\diffp[1], Y_i^\diffp[2], \cdots, Y_i^\diffp[2g])^\intercal$ and its  covariance matrix. We then find estimates for the distribution parameters with common mean $\mu_j = \mu$ for all $j \in [g]$, group probability $\pi$, and standard deviations $\sigma_j$ for $j \in [g]$, which will be used in the covariance matrix in place of the population parameters.  
\[
Y = \sum_{i=1}^n Y_i = 
\sum_{i=1}^n 
\begin{pmatrix}
W_{i}[1] &
 W_{i}[2] &
\dots &
W_{i}[g-1] &
W_{i}[1]  \cdot X_{i}[2] &
W_{i}[2] \cdot X_{i}[2] &
 \dots &
 W_{i}[g]  \cdot X_{i}[g] 
\end{pmatrix}^T .
\]

After minimizing the resulting $\chi^2$ statistic over $\mu \in \R$ and $\pi \in [0,1]^g$ such that $\sum_{i=1}^g \pi_j = 1$, we then compare the statistic to a $\chi^2$ distribution with $g-1$ degrees of freedom.  Figure~\ref{fig:nonPrivANOVA} shows that the $\chi^2$ approach performs equivalently to the traditional one-way ANOVA test with equal variance $\sigma^2 = \sigma_j^2$ for $j \in [g]$.  

\begin{figure}
\begin{center}
  \includegraphics[width=0.45\linewidth]{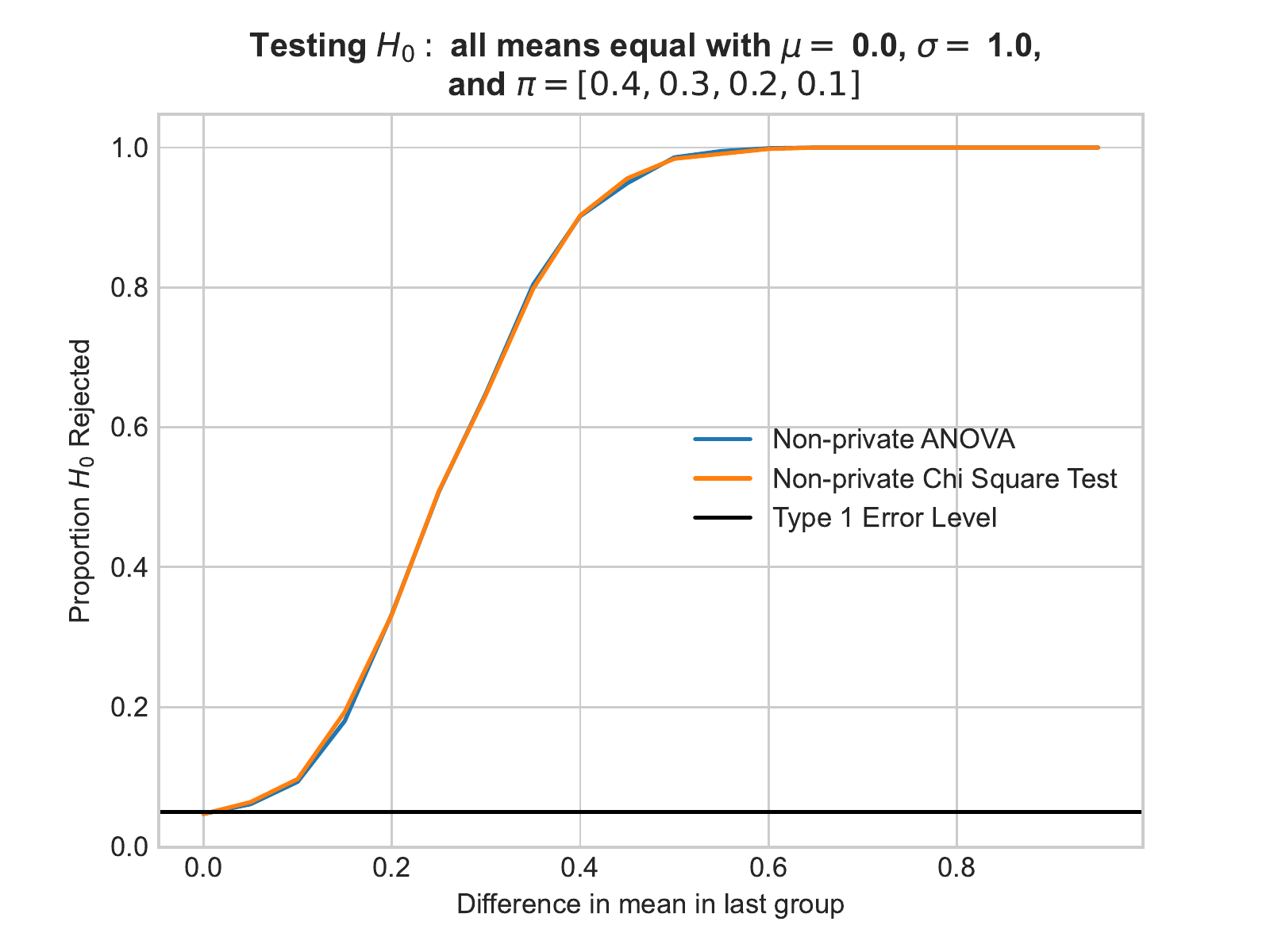}
  \includegraphics[width=0.45\linewidth]{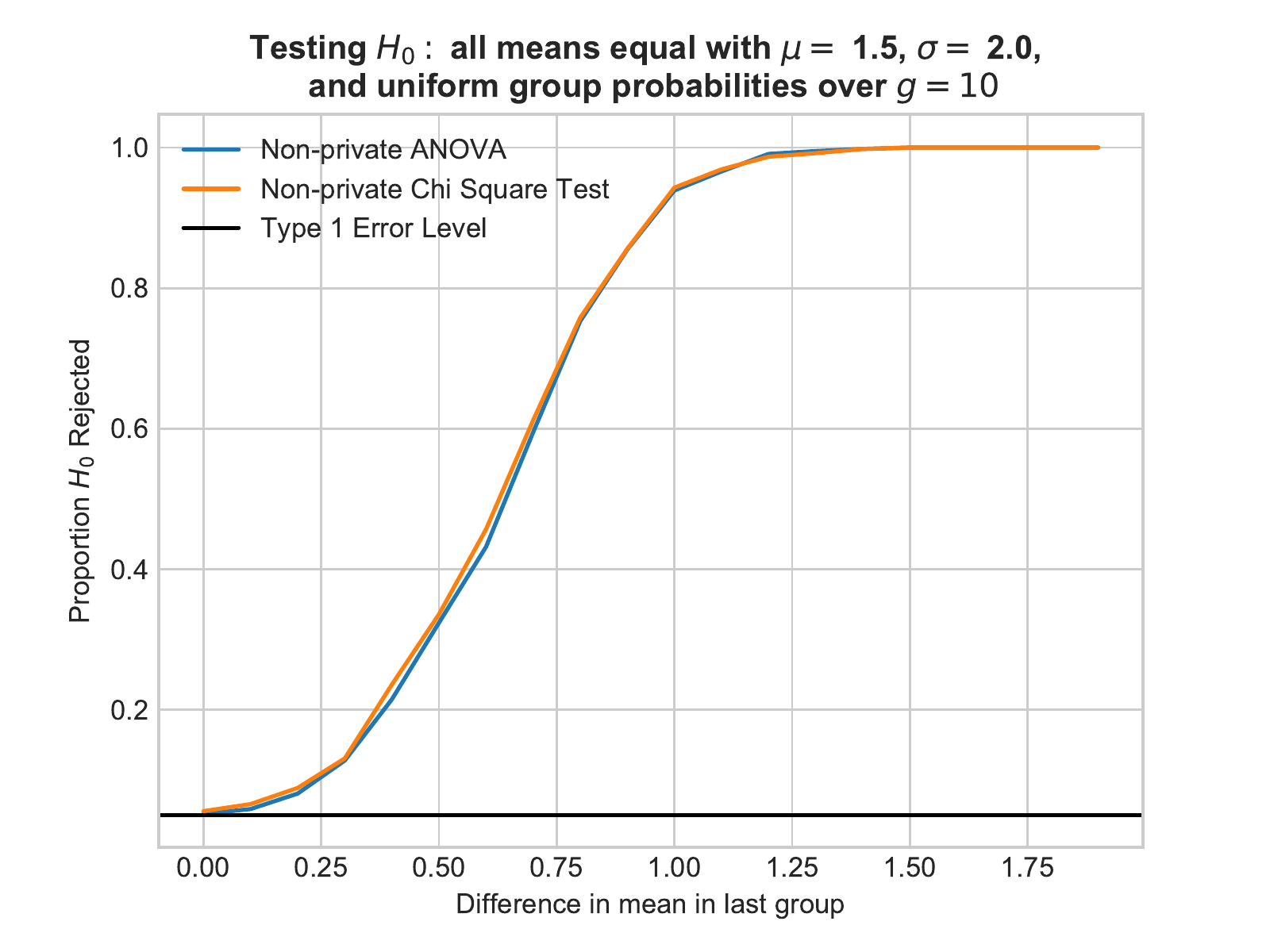}
  \caption{Comparing power curves for one-way ANOVA and the general $\chi^2$ approach with $n = 1000$ samples.}
  \label{fig:nonPrivANOVA}
\end{center}
\end{figure}

To then introduce privacy of the group of each sample, but not the outcome, we will include the random matrix $Z_i^\diffp \in \{0,1 \}^{g\times g}$, where column $j$ corresponds to the outcome for $M(j)$ for various local DP mechanisms $M$.  Hence, we can write out the privatized vector $Y^\diffp = \sum_{i=1}^n Y_i^\diffp$ in terms of $Z_i^\diffp$ as we did for binary outcomes in \eqref{eq:privY}, where $\mathbf{0}$ is the $g\times g$ zero matrix.
\begin{equation}
Y_i^\diffp=  \begin{bmatrix}
			Z_i^\diffp & \mathbf{0} \\
			\mathbf{0} & Z_i^\diffp
			\end{bmatrix} 
\begin{pmatrix}
W_{i}[1] &
 W_{i}[2] &
\dots &
W_{i}[g] &
W_{i}[1] \cdot X_{i}[1] &
W_{i}[2] \cdot X_{i}[2] &
 \dots &
 W_{i}[g]\cdot X_{i}[g] 
\end{pmatrix}^T 
\label{eq:privYANOVA}
\end{equation}
Note that we include the coordinate for $W_{i}[g]$ because the privacy mechanism will modify the probability of being in the last group.  For the $g$-randomized response mechanism, we can eliminate this entry as the first $g$ coordinates of $Y_i^\diffp$ will still form a multinomial distribution.  However, for the bit flip mechanism, we will need to keep  the full $2g$ dimensional vector, in which case we compare the resulting statistic to a $\chi^2$ distribution with $g$ degrees of freedom, one more than the non-private test.  Lastly, for the subset mechanism, the resulting covariance matrix of $Y_i^\diffp$ will have the vector comprising of ones in the first $g$ coordinates and zeros in the second $g$ coordinates be in its null space.  We will cover each in turn in this section.

For each mechanism, we need to compute the expected vector $\vec{\theta}^\diffp(\pi, \mu)$ under the null hypothesis and the covariance matrix $C(\pi, \mu, \sigma_1^2, \cdots, \sigma_g^2; \diffp)$.   Note that the mean vector $\vec{\theta}^\diffp(\pi, \mu) \in \R^{g\times g}$ will have the same form as the vectors in Section~\ref{sect:independence}, except the first $g$ entries will have success probability $p = 1$ and the second $g$ entries will have $\mu$ instead of $1-p$ multiplied.  
We will be able to write the covariance matrix in the following block form with matrices $\Sigma, \Sigma' \in \R^{g\times g}$,
\begin{equation}\label{eq:generalCovar}
C(\pi, \mu, \sigma_1^2, \cdots, \sigma_g^2; \diffp) = 
\begin{bmatrix}
\Sigma & \mu \Sigma \\
\mu \Sigma & \Sigma'
\end{bmatrix}
\end{equation}
Note that $\Sigma$ will simply be the top left $g \times g$ submatrix of the covariance matrix for each mechanism in Section~\ref{sect:independence}, with probability of success $p = 1$.  That is for $j ,\ell \in [g]$,
\[
\Sigma[j,\ell] = \sum_{m =1}^g  \E[ W_i[m] \cdot Z_i^\diffp[j,m] \cdot Z_i^\diffp[\ell,m] ] -  \E[Y_i^\diffp][j]  \cdot \E[Y_i^\diffp][\ell] 
\]
Further, we can compute $\Sigma'$ in terms of the expectation of $Y_i^\diffp$ and the specific mechanism we use for $j, \ell \in[g]$
\begin{align*}
\Sigma'[j,\ell] & =  \sum_{m =1}^g (\mu^2 + \sigma^2_m) \E[ W_i[m] \cdot Z_i^\diffp[j,m] \cdot Z_i^\diffp[\ell,m] ] - \mu^2 \E[Y_i^\diffp][j]  \cdot \E[Y_i^\diffp][\ell] 
\end{align*} 

We also need to form estimates for the population parameters in the covariance matrix, including $\mu, \pi, \sigma_1, \cdots, \sigma_g$.  We will use our sample and the particular privacy mechanism to form these estimates.  Note that if our estimates result in an expected group size to be less than 5, then we will simply fail to reject the null hypothesis, as we have stated earlier. We now cover each mechanism below.

\subsection{Randomized Response}
We first cover the $g$-randomized response mechanism we have $Y_i^\diffp = (Y_i^\diffp[1], \cdots, Y_i^\diffp[2g])^\intercal$, which will have the following mean vector
\[ 
\vec{\theta}^\diffp(\pi, \mu) = \E[Y_i^\diffp]= \begin{pmatrix}
\frac{1}{e^\diffp + g - 1} \cdot \begin{bmatrix}
e^\diffp & 1 & \cdots & 1 \\
1 & e^\diffp & \cdots & 1 \\
 &  & \ddots &  \\
1 & 1 & \cdots & e^\diffp 
	\end{bmatrix} \pi \\	
\\	
\frac{\mu}{e^\diffp + g -1}\cdot \begin{bmatrix}
e^\diffp & 1 & \cdots & 1 \\
1 & e^\diffp & \cdots & 1 \\
 &  & \ddots &  \\
1 & 1 & \cdots & e^\diffp 
	\end{bmatrix} \pi \\
\end{pmatrix}
\]
Using our general form of the covariance matrix in \eqref{eq:generalCovar} we can simplify the terms for $\Sigma[j,\ell]$ due to $Z_i^\diffp[j,m] \cdot Z_i[\ell, m] = 0$ for any $\ell \neq j$ and is 1 otherwise.  We also need to form estimates to use in the $\chi^2$ statistic, which gives us
\[
\hat{\mu} = \frac{\sum_{j \in [g]} \sum_{i=}^nY_i^\diffp[g+j]}{n}, \qquad \hat{\pi} = \left(  (e^\diffp + g-1) \left(\frac{\frac{\sum_{i=1}^nY_i^\diffp[j]}{n} - \frac{1}{e^\diffp + g-1}}{e^\diffp - 1} \right): j \in [g] \right).
\]
We also need to estimate the variance for each group, so we will use the sample standard deviation for each group, i.e. $\{Y_i^\diffp[g+1:2g] : i \in [n] \}$ and use this for the main diagonal of $\Sigma'$.  The terms of $\Sigma'$ on the off diagonal will only consist of terms $- \E[ Y_i^\diffp][j] \cdot \E[ Y_i^\diffp][\ell]$ for $j \neq \ell$.

We also point out that for randomized response, the first $g$ entries will still follow a multinomial distribution, where only the first $g-1$ entries are needed.  Hence, we remove the $g$-th entry in $Y_i^\diffp$ while also removing the $g$-th row and column of the covariance matrix.  The result is then a covariance matrix of rank at most $2g - 1$ and we optimize over $g$ variables, with $\pi_1, \cdots, \pi_{g-1}$ and $\mu$.  We then form the $\chi^2$-statistic and compare it to a $\chi^2$ distribution with $g-1$ degrees of freedom.  

\subsection{Bit Flipping}

We next turn to the bit flipping mechanism, where it is possible for a sample to be in multiple groups simultaneously.  We first compute the expected vector of $Y_i^\diffp$, as we did for the binary outcome case in Section~\ref{sect:BitFlipBinaryOutcomes}
\[
\vec{\theta}^\diffp(\pi,\mu) = 
\begin{pmatrix}
\frac{1}{e^{\diffp/2} +1} \cdot \begin{bmatrix}
e^{\diffp/2} & 1 & \cdots & 1 \\
1 & e^{\diffp/2} & \cdots & 1 \\
 &  & \ddots &  \\
1 & 1 & \cdots & e^{\diffp/2} 
	\end{bmatrix} \pi \\ \\
	\frac{\mu}{e^{\diffp/2} +1} \cdot\begin{bmatrix}
e^{\diffp/2} & 1 & \cdots & 1 \\
1 & e^{\diffp/2} & \cdots & 1 \\
 &  & \ddots &  \\
1 & 1 & \cdots & e^{\diffp/2} 
	\end{bmatrix} \pi
\end{pmatrix}
\]
We next compute the covariance matrix $\Sigma'$, which consists of terms with $\E[Z_i^\diffp[j,m] Z_i^\diffp[\ell,m]]$ for $j,\ell, m \in [g]$.  From the bit flip mechanism, we know each coordinate of $Z_i^\diffp$ is independent of each other, so we have for $j \neq \ell$
\[
\E[Z_i^\diffp[j,m] Z_i^\diffp[\ell,m]] = \E[Z_i^\diffp[j,m] ] \cdot \E[Z_i^\diffp[\ell,m]] = \left\{ \begin{array}{lr} 
																\tfrac{1}{(e^{\diffp/2} + 1)^2} & j,\ell \neq m \\
																\tfrac{e^{\diffp/2}}{(e^{\diffp/2} + 1)^2} & j = m, \text{ or } \ell = m
															\end{array}\right.
\]
and if $j = \ell$
\[
\E[Z_i^\diffp[j,m] Z_i^\diffp[\ell,m]] = \E[Z_i^\diffp[j,m] ] = \left\{ \begin{array}{lr} 
																\tfrac{1}{e^{\diffp/2} + 1} & j \neq m \\
																\tfrac{e^{\diffp/2}}{e^{\diffp/2} + 1} & j = m
															\end{array}\right.
\]

We now form estimates for the population parameters in the covariance matrix.  
\[
\hat{\mu} = (e^{\diffp/2}+1) \cdot \frac{\sum_{i=1}^n \sum_{j=1}^g Y_i^\diffp[g+j]}{n (e^{\diffp/2} + (g-1) )}, \quad \hat{\pi} = \left( (e^{\diffp/2}+1) \cdot \left(\frac{\tfrac{ \sum_{i=1}^n Y_i^\diffp[j] }{n} - \tfrac{1}{e^{\diffp/2} + 1} }{e^{\diffp/2} - 1}\right) : j \in [g]\right)
\]
Next, we need to form an estimate for the variance.  To help simplify things, we will assume equal variance across groups in our estimate, i.e $\sigma_j = \sigma$ for all $j \in [g]$.  We point out that unequal variances can be used, but it just complicates the estimate we use.  Further, we will write $s_j^2$ to denote the sample variance computed within each group, so that $s_j^2$ is an estimate for the variance of $Y_i^\diffp[g+j]$.  This gives us the following estimate for $\sigma$,
\[
\hat{\sigma}^2 = (e^{\diffp/2} + 1) \cdot \frac{\sum_{j=1}^g s_j^2 - \hat{\mu}^2\left( \tfrac{e^{\diffp/2}}{e^{\diffp/2}+1} + (g-1) \tfrac{1}{e^{\diffp/2}+1} - \sum_{j=1}^g( \pi_j \tfrac{e^{\diffp/2}}{e^{\diffp/2}+1} + (1-\pi_j) \tfrac{1}{e^{\diffp/2}+1} )^2 \right) }{e^{\diffp/2} + (g-1)}.
\]

We now use the covariance matrix with these estimates in the general $\chi^2$ statistic and compare it with a $\chi^2$ distribution with $g$ degrees of freedom.  Note that there is an extra degree of freedom compared to the previous test and the non-private version.  This is because the covariance matrix may be full rank and some combination of elements in $Y_i^\diffp$ cannot be used to determine other elements, as was possible when we had the $g$-randomized response mechanism.  Hence, the covariance matrix is of rank at most $2g$ and we are minimizing over $g$ variables $\pi_1, \cdots, \pi_g, \mu$.  
\begin{remark}
Recall that in the binary outcome case, we were using a technique from \cite{KiferRo17, GaboardiRo18} to \emph{project out} the eigenvector associated with noise, which reduced the degrees of freedom by 1.  A similar approach cannot be adopted here, as in the non-private case, the vector whose first $g$ coordinates are 1 and the latter $g$ coordinates are zero is in the null space of the covariance matrix, but the same vector is no longer an eigenvector of the covariance matrix after applying the bit flipping mechanism for general $\mu$.  
\end{remark}

\subsection{Subset Mechanism}

Lastly we cover the subset mechanism, which was shown to be the most powerful of the other privacy mechanisms for various privacy levels in the binary outcomes case.  To ease notation, we will assume the variance is equal across all groups, i.e. $\sigma_j = \sigma$ for all $j \in [g]$.  We follow the same procedure as for the other mechanisms, where we first give the expected value of $Y_i^\diffp$, denoted as $\vec{\theta}^\diffp(\pi, p;k)$ where the random entries in $Z_i^\diffp$ from \eqref{eq:privYANOVA} come from the subset mechanism with parameter $k$.  
\[
\vec{\theta}^\diffp(\pi, p;k) = 
\begin{pmatrix}
\frac{1}{{g-1\choose k-1} e^\diffp + {g-1\choose k} } 
\begin{bmatrix}
{g-1 \choose k-1}e^\diffp & \left( {g-2 \choose k-2}e^\diffp + {g-2\choose k-1} \right) & \cdots & \left( {g-2 \choose k-2}e^\diffp + {g-2\choose k-1} \right) \\
\left( {g-2 \choose k-2}e^\diffp + {g-2\choose k-1} \right) & {g-1 \choose k-1}e^\diffp & \cdots & \left( {g-2 \choose k-2}e^\diffp + {g-2\choose k-1} \right) \\
 &  & \ddots &  \\
\left( {g-2 \choose k-2}e^\diffp + {g-2\choose k-1} \right) & \left( {g-2 \choose k-2}e^\diffp + {g-2\choose k-1} \right) & \cdots & {g-1 \choose k-1}e^\diffp
	\end{bmatrix} \pi \\ \\
	\frac{\mu}{{g-1\choose k-1} e^\diffp + {g-1\choose k} } 
	\begin{bmatrix}
{g-1 \choose k-1}e^\diffp & \left( {g-2 \choose k-2}e^\diffp + {g-2\choose k-1} \right) & \cdots & \left( {g-2 \choose k-2}e^\diffp + {g-2\choose k-1} \right) \\
\left( {g-2 \choose k-2}e^\diffp + {g-2\choose k-1} \right) & {g-1 \choose k-1}e^\diffp & \cdots & \left( {g-2 \choose k-2}e^\diffp + {g-2\choose k-1} \right) \\
 &  & \ddots &  \\
\left( {g-2 \choose k-2}e^\diffp + {g-2\choose k-1} \right) & \left( {g-2 \choose k-2}e^\diffp + {g-2\choose k-1} \right) & \cdots & {g-1 \choose k-1}e^\diffp
	\end{bmatrix} \pi
\end{pmatrix}
\]

We write the $j, \ell$ entry of the submatrix $\Sigma' = \E[Y_i^\diffp[g+j] \cdot Y_i^\diffp[g+\ell]] - \E[Y_i^\diffp] [g+j] \cdot  \E[Y_i^\diffp][\ell]$, where the first term can be computed with $j = \ell$
\[
\E[Y_i^\diffp[g+j]^2]= (\mu^2 + \sigma^2) \frac{{g-1 \choose k-1}e^\diffp\pi_j + \left( {g-2 \choose k-2}e^\diffp  + {g-2 \choose k-1}\right) (1-\pi_j)}{{g-1\choose k-1} e^\diffp + {g-1\choose k} } ,  \qquad \forall j \in [g].\\
\]
and for $\ell \neq j$,
\begin{align*}
& \E[Y_i^\diffp[g+j] Y_i^\diffp[g+\ell]] \\
& \qquad  =\frac{(\mu^2 + \sigma^2)}{{g-1\choose k-1} e^\diffp + {g-1\choose k} }  \left( e^{\diffp} {g-2\choose k-2}(\pi_j +\pi_\ell) + \left(e^\diffp {g-3\choose k-3} + {g-3\choose k-2} \right) (1-\pi_j - \pi_\ell) \right)
\end{align*}

We make the following observation about the covariance matrix of $Y_i^\diffp$ that will show that it is not full rank, and hence we will lose a degree of freedom in the asymptotic $\chi^2$ distribution of the test statistic.  

\begin{lemma}
The covariance matrix $C(\pi, \mu, \sigma^2_1, \cdots, \sigma^2_g; \diffp, k)$ of $Y_i^\diffp$ in \eqref{eq:privYANOVA} for the subset mechanism has a nontrivial null space.
\end{lemma}
\begin{proof}
We can write the covariance matrix as  a block matrix, where we will not care about the bottom right $g\times g$ block, since it will not be touched with the vector $(1, \cdots, 1, 0, \cdots, 0)$.  Let $\Sigma$ be the covariance matrix of $Z_i^\diffp W_i$, which we have actually computed in Lemma~\ref{lem:singularCovar} with success probability $p = 1$.  The top left $g\times g$ block matrix in $C(\cdot)$ will then be $\Sigma$, which has the all ones vector in its null space.  Furthermore, the bottom left block matrix (as well as the top right, since the covariance matrix is symmetric) will be $\mu\Sigma$ where $\mu \in \R$ is the common mean across all groups under the null hypothesis.   This gives us what we need for the lemma statement.
\[
C(\pi, \mu, \sigma^2_1, \cdots, \sigma^2_g; \diffp) 
\begin{pmatrix}
1 \\
\vdots \\
1 \\
0 \\
\vdots \\
 0
\end{pmatrix} = 
\begin{bmatrix}
\Sigma \mathbf{1} + \mu \Sigma \mathbf{0} \\
\mu \Sigma \mathbf{1}  + 0
\end{bmatrix} = 0
\]
\end{proof}

We now turn to computing estimates for the population parameters to plug into our covariance matrix.
\begin{align*}
\hat{\mu} & = \frac{\sum_{j=1}^g \sum_{i=1}^n Y_i^\diffp[g+j]}{n k}, \\
\hat{\pi} & = \left( \frac{ \left( {g-1\choose k-1} e^\diffp + {g-1\choose k} \right)\cdot \left( \frac{\sum_{i=1}^nY_i^\diffp[j]}{n} \right)- \left(e^\diffp {g-2\choose k-2} + {g-2\choose k-1} \right) }{e^\diffp \left({ g-1 \choose k-1 } - { g-2 \choose k-2 } \right) - { g-2 \choose k-1 } } : j \in [g]\right) . 
\end{align*}

Lastly, we need to estimate the variance for each group, which we will again assume for ease of notation that $\sigma_j = \sigma$ for all groups $j \in [g]$.  We will use the sample variance $s_j^2$ for $Y_i^\diffp[g+j]$ within each group $j \in [g]$.  We then use the following estimate
\[
\hat{\sigma}^2 = \frac{\sum_{j=1}^g s_j^2 - \mu^2 \left( k -  \tfrac{1}{\left( {g-1 \choose k-1} e^\diffp + {g-1 \choose k} \right)^2} \sum_{j=1}^g\left( {g-1 \choose k-1} e^\diffp \pi_j  + (1- \pi_j) \left( {g-2 \choose k-2 } e^\diffp  + {g-2 \choose k-1}\right) \right)^2 \right) }{k}
\]

Hence, we will compare the resulting $\chi^2$ statistic with the $\chi^2$ distribution with $g-1$ degrees of freedom.  As noted earlier, if any test computes a group probability $\hat{\pi}_j$ so that $n \cdot \hat{\pi} \leq 5$, we simply reject the null hypothesis.  

\subsection{Results}

We present results of our various private tests in Figure~\ref{fig:PrivANOVA} with $\mu = \mu_j$ for all $j \in [g]$ with $g = 10$ but we vary the mean in the last coordinate $\mu_{10}$ to see the fraction of times we reject the null.    We compare the tests for the 3 different privacy mechanisms, $g$-randomized response, bit flipping, and the subset mechanism, as well as the approach of using the one-way ANOVA test as is on data that has been privatized with $g$-randomized response.   We see that the subset mechanism outperforms the various tests at the different privacy levels.  Recall that the subset mechanism has a free parameter $k$ which we fix to be the nearest integer to $g/(e^\diffp + 1)$ and when $\diffp$ is large, the subset mechanism becomes the $g$-randomized response mechanism.  Throughout our experiments, we assume that the variances across all groups are the same in our estimates in the covariance matrix of $Y_i^\diffp$, as is the setting for the one-way ANOVA.  

\begin{figure}
\begin{center}
  \includegraphics[width=0.31\linewidth]{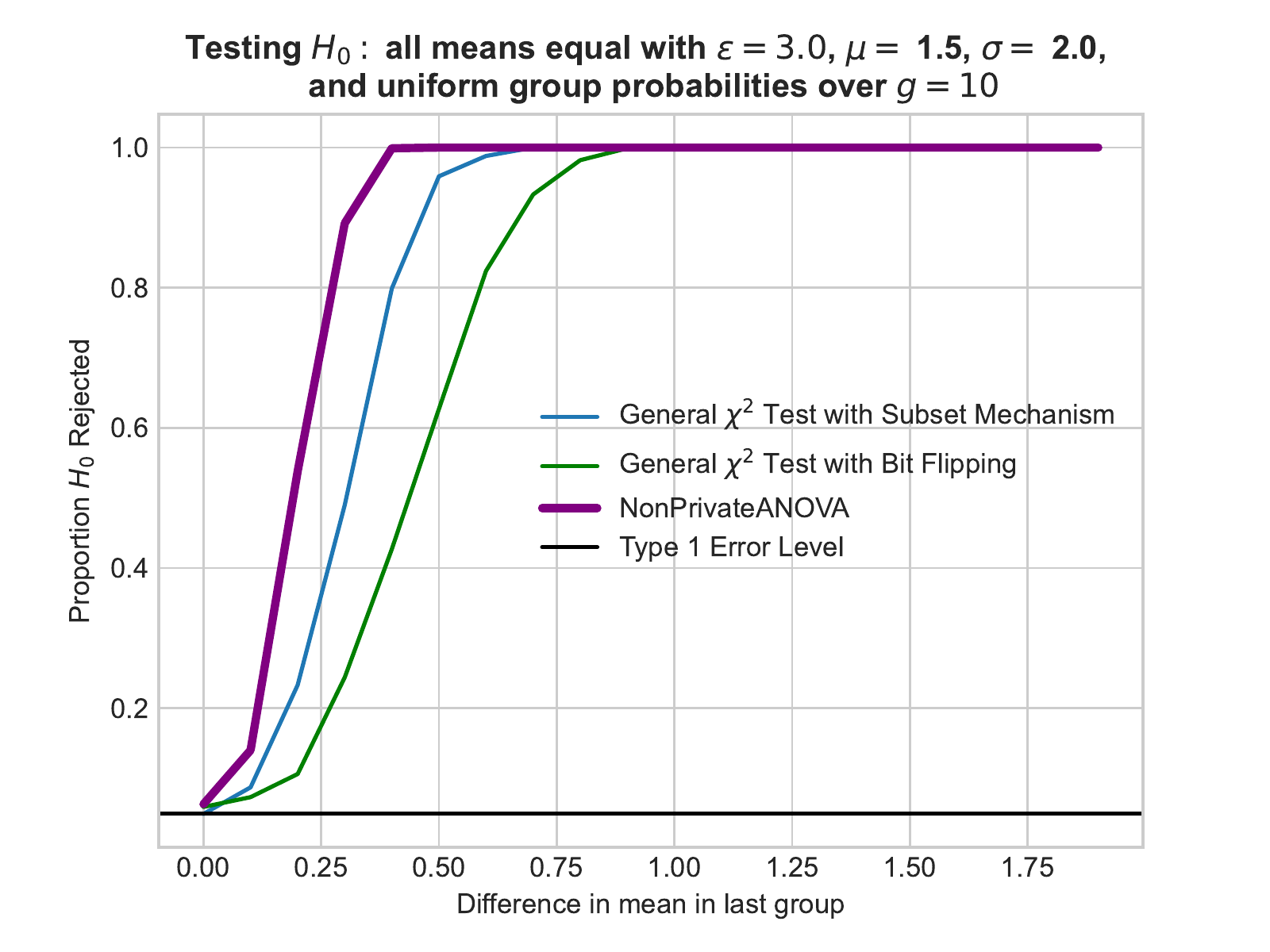}
    \includegraphics[width=0.31\linewidth]{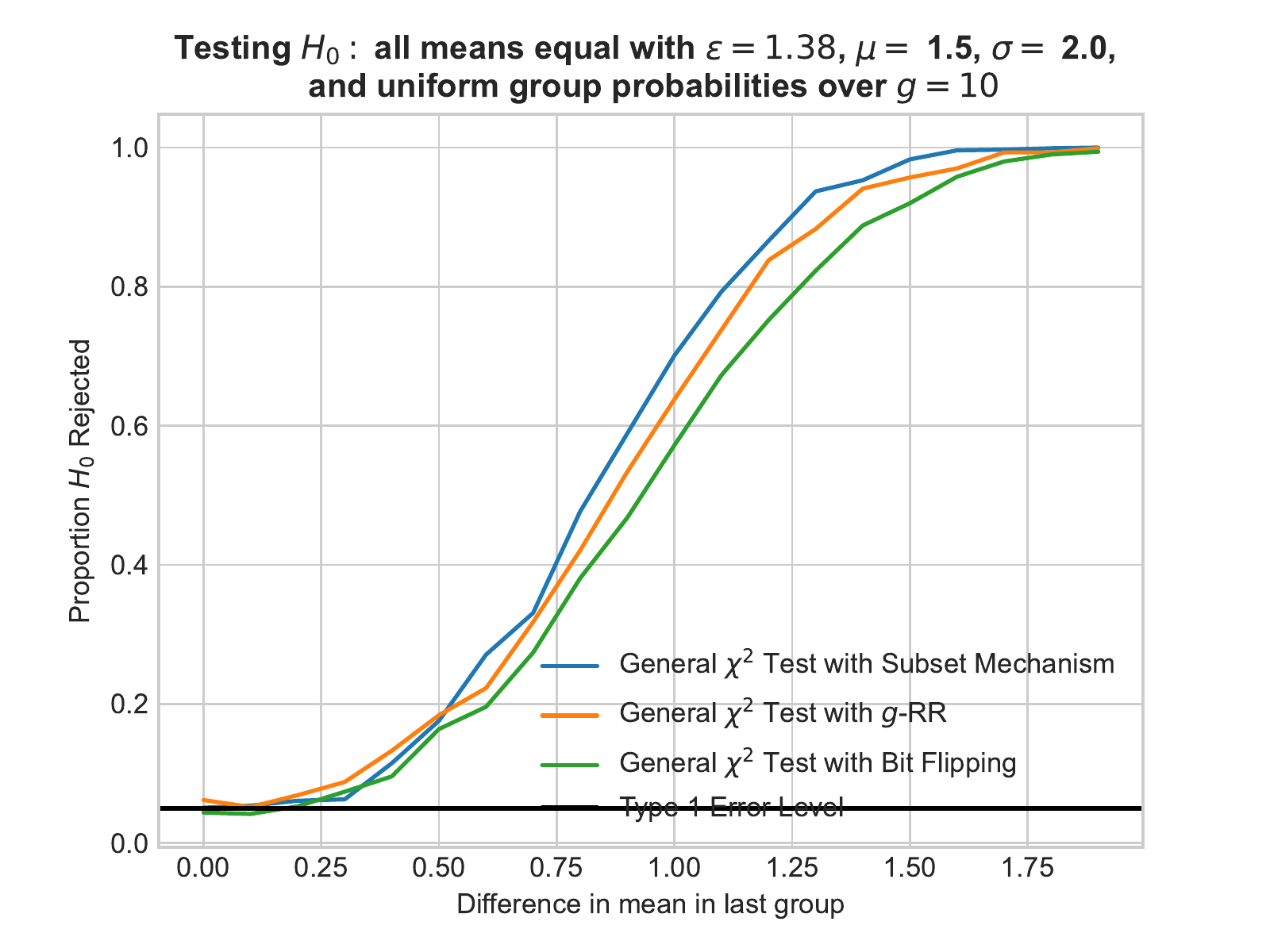}
      \includegraphics[width=0.31\linewidth]{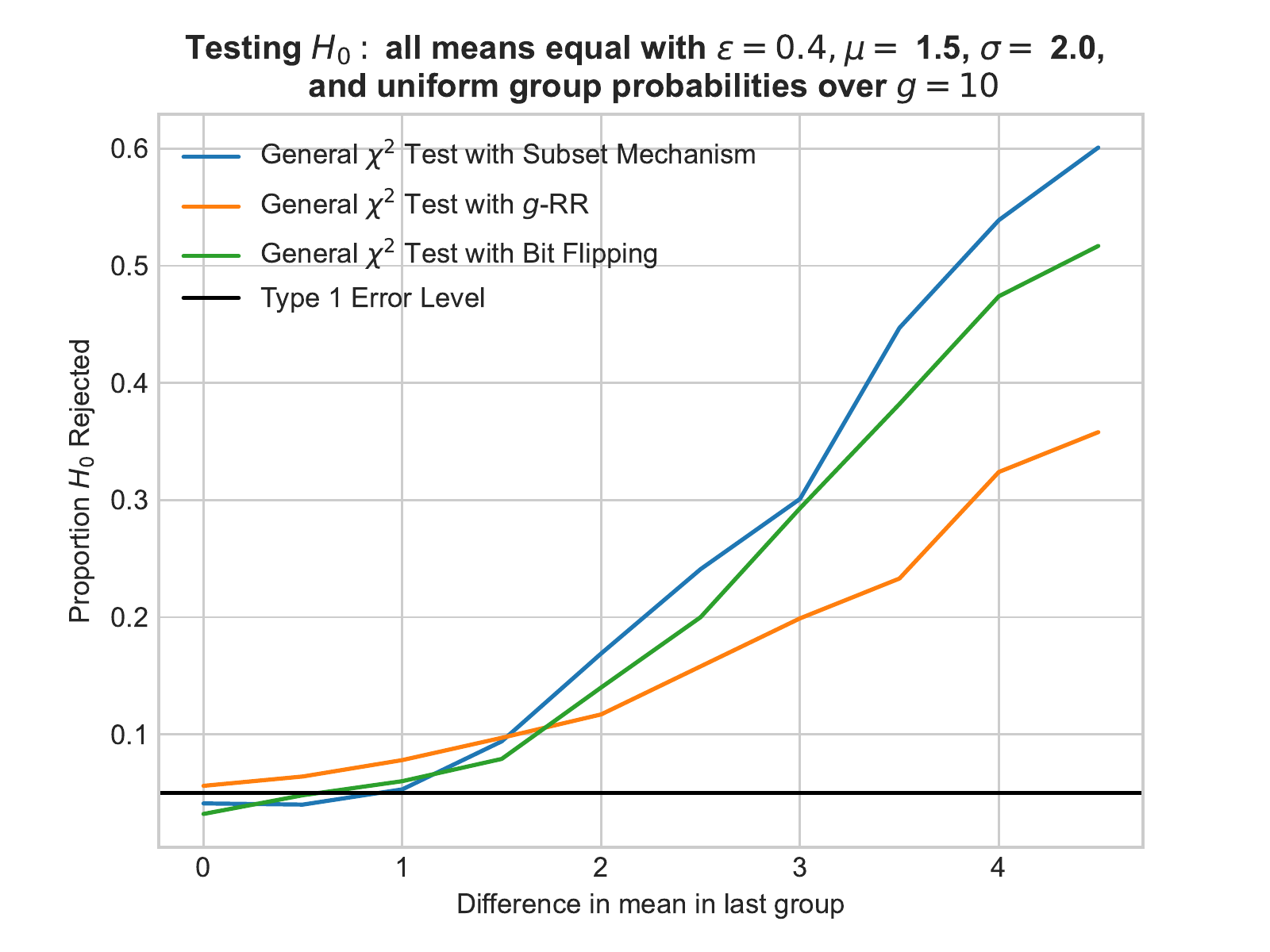}
  \caption{Comparing power curves for one-way ANOVA and the general $\chi^2$ approach with $n = 10000$ samples and various levels of privacy.}
  \label{fig:PrivANOVA}
\end{center}
\end{figure}

We also show that the tests we develop achieve higher empirical power rather than simply using the classical one-way ANOVA tests after privatizing the groups.  We present plots in Figure~\ref{fig:compareClassicMultipleMeans} that shows for data generated with the subset mechanism at various privacy levels, the general $\chi^2$ test that accounts for the subset mechanism outperforms using the classic ANOVA test, which does not account for the privacy mechanism.  We point out that when $\diffp$ gets larger, the two tests seem to perform similarly, similar to what we saw in Section~\ref{sect:MultiplePropDiffResults} with testing multiple proportions.  All plots consist of the proportion of times the null hypothesis was rejected over 1000 trials.  

\begin{figure}
  \includegraphics[width=0.46\linewidth]{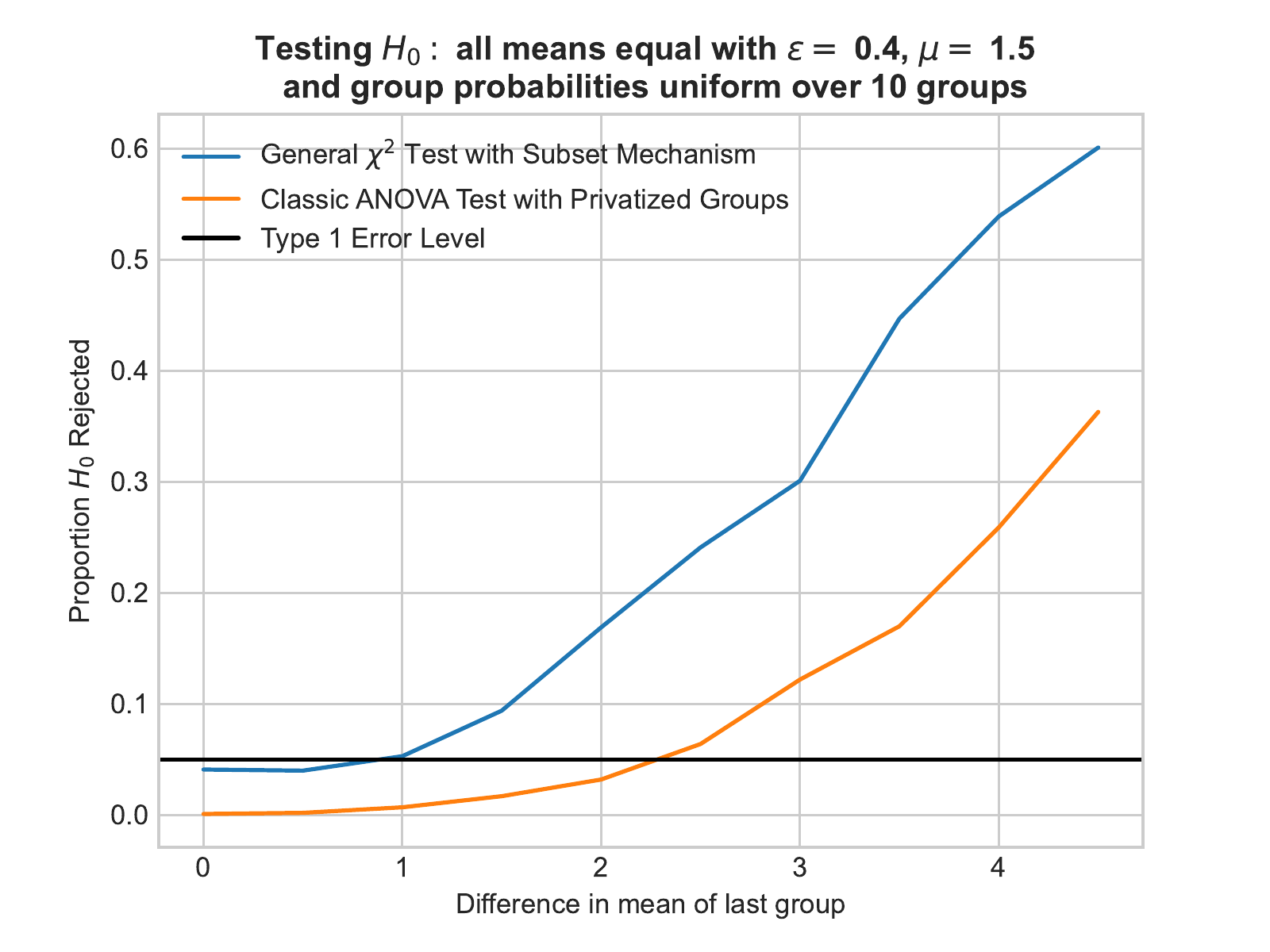}
  \includegraphics[width=0.46\linewidth]{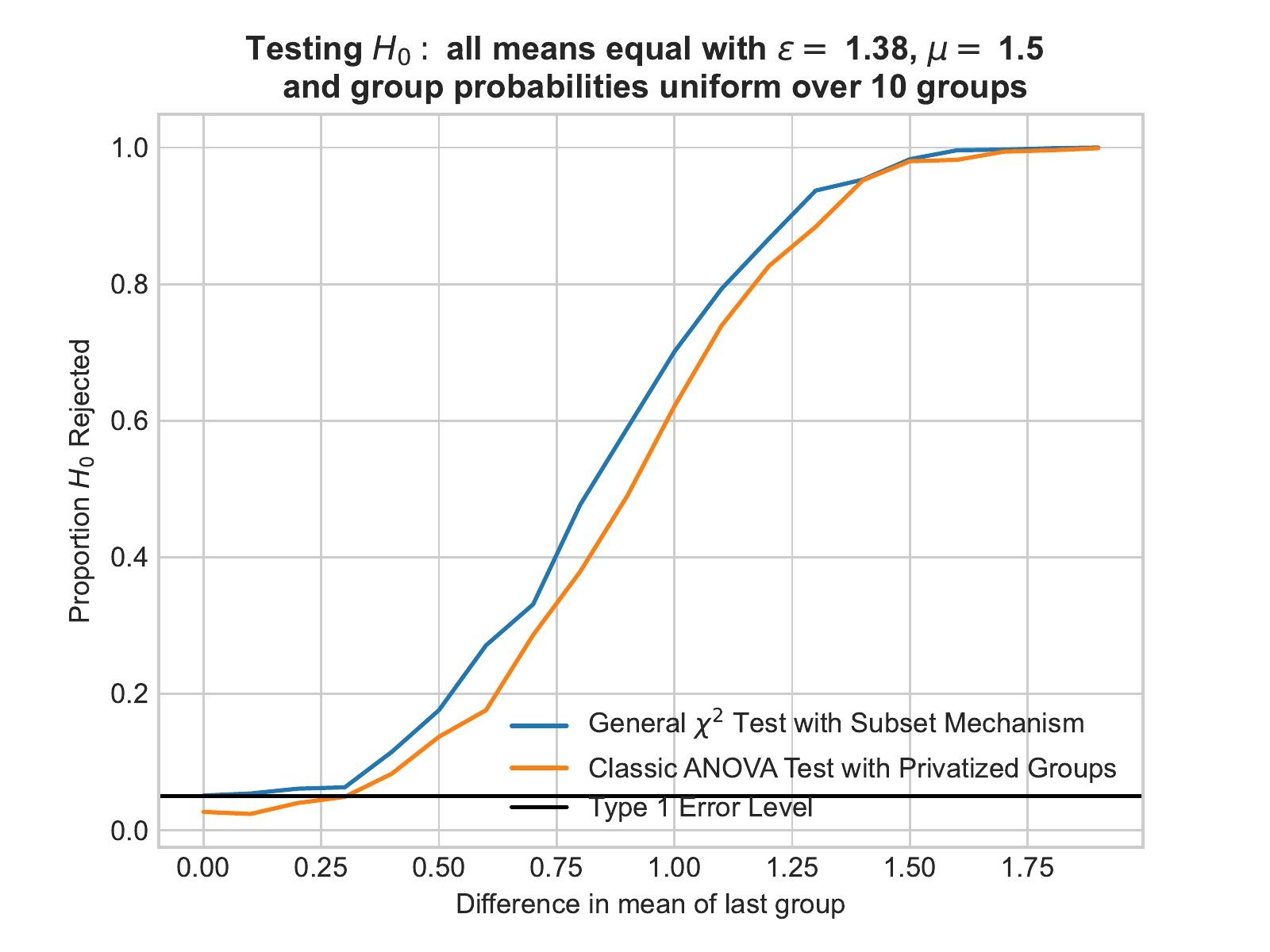}
  \caption{Comparing various Local Group DP mechanisms with corresponding one-way ANOVA test for testing whether there is a difference in means across different sensitive groups with various $\diffp$ and $n = 10000$.}
  \label{fig:compareClassicMultipleMeans}
\end{figure}

\subsection{Testing Difference in Two Groups}
Once we have determined that there is indeed a difference across all group means, one typically wants to compute confidence intervals for the mean between two specific groups.  In fact, one may want to directly compute a confidence interval between two groups, although the data has been privatized over several groups.  In our privacy setting, we do not want to privatize the group membership each time we want to conduct a test, hence we will have samples with privatized groups which will mix with the two specific groups we want to compute a confidence interval for their difference.  Ideally, we would privatize only the two groups we are interested in, but this would increase the privacy loss, something we want to avoid.  Hence, we show how we can still obtain valid confidence intervals between two specific groups although the data has been privatized over $g> 2$ groups.   

Since we are only interested in the difference in means between two groups, say $H_0: \mu_j = \mu_\ell + \Delta$ for $j, \ell \in [g]$, we can change the optimization of the general $\chi^2$ statistic to allow for any mean $\mu_m$ where $m \neq j,\ell$ based on the data samples.  This will reduce our degrees of freedom of the asymptotic $\chi^2$ by $g-1$, thus if we privatize the groups with the Subset Mechanism, then the $\chi^2$ statistic, after optimizing over all $\mu_m$ for $m \neq j, \ell$ and $\mu_j = \mu_\ell + \Delta$, should be compared to a $\chi^2$ with 1 degree of freedom.   

We show in Figure~\ref{fig:SignificanceSubset} that testing whether two means are equal can lead to invalid results if we were to use the classical t-test after the groups have been privatized.  This is in contrast to when we would privatize only two groups, where the classical t-test empirically achieve the target level of Type 1 error, see Figure~\ref{fig:tTestPowerComparisons}.  The general $\chi^2$ approach can then be used to achieve the target level of Type 1 error.  We also give confidence intervals for the difference in mean in Figure~\ref{fig:tTestCIs} where the classical t-test misses the true difference while the general $\chi^2$ test overlaps the true difference.  The difference between the two tests becomes more pronounced when the group probabilities between the two groups of interest differ from each other.

\begin{figure}
  \includegraphics[width=0.31\linewidth]{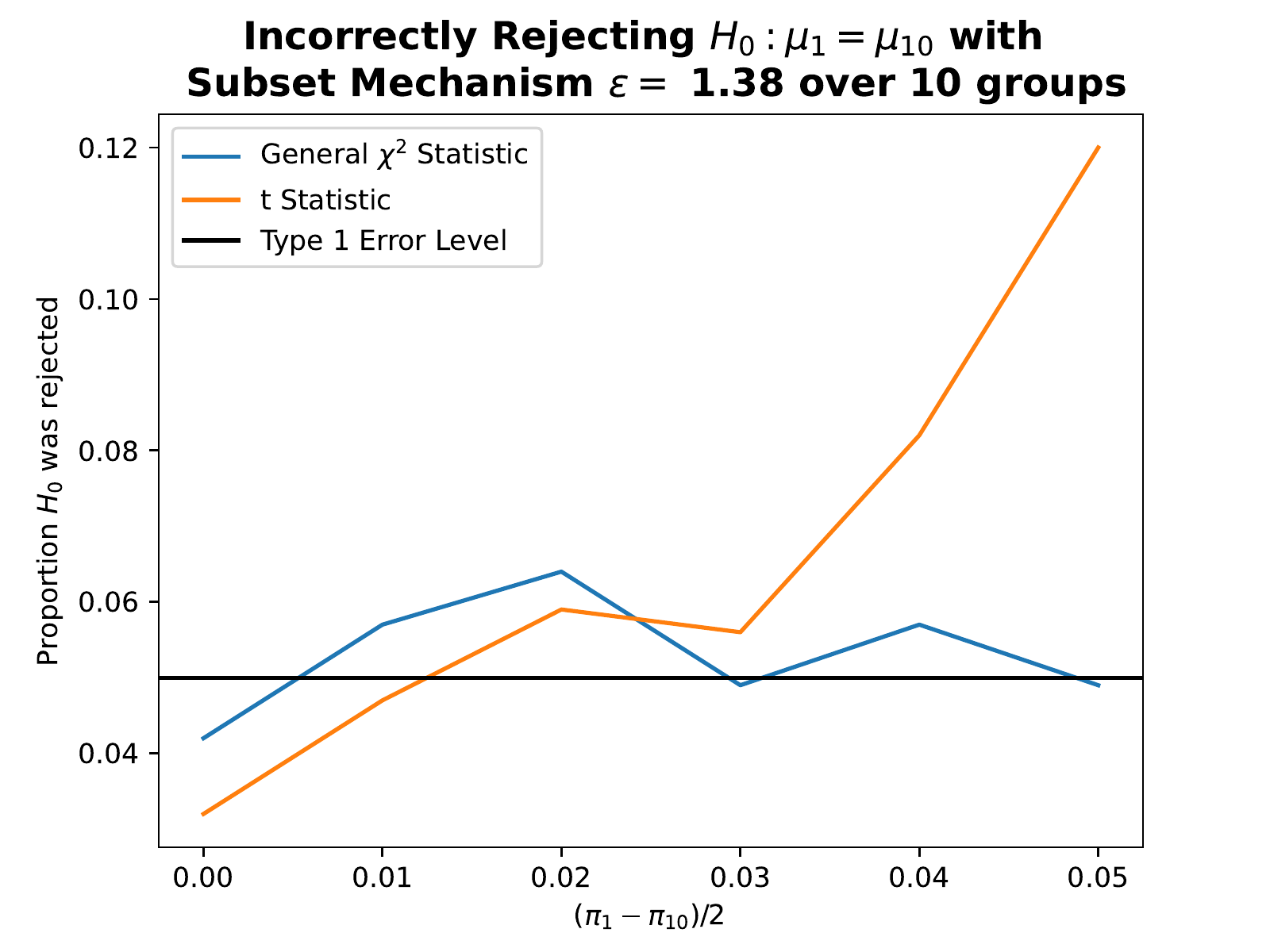}
  \includegraphics[width=0.31\linewidth]{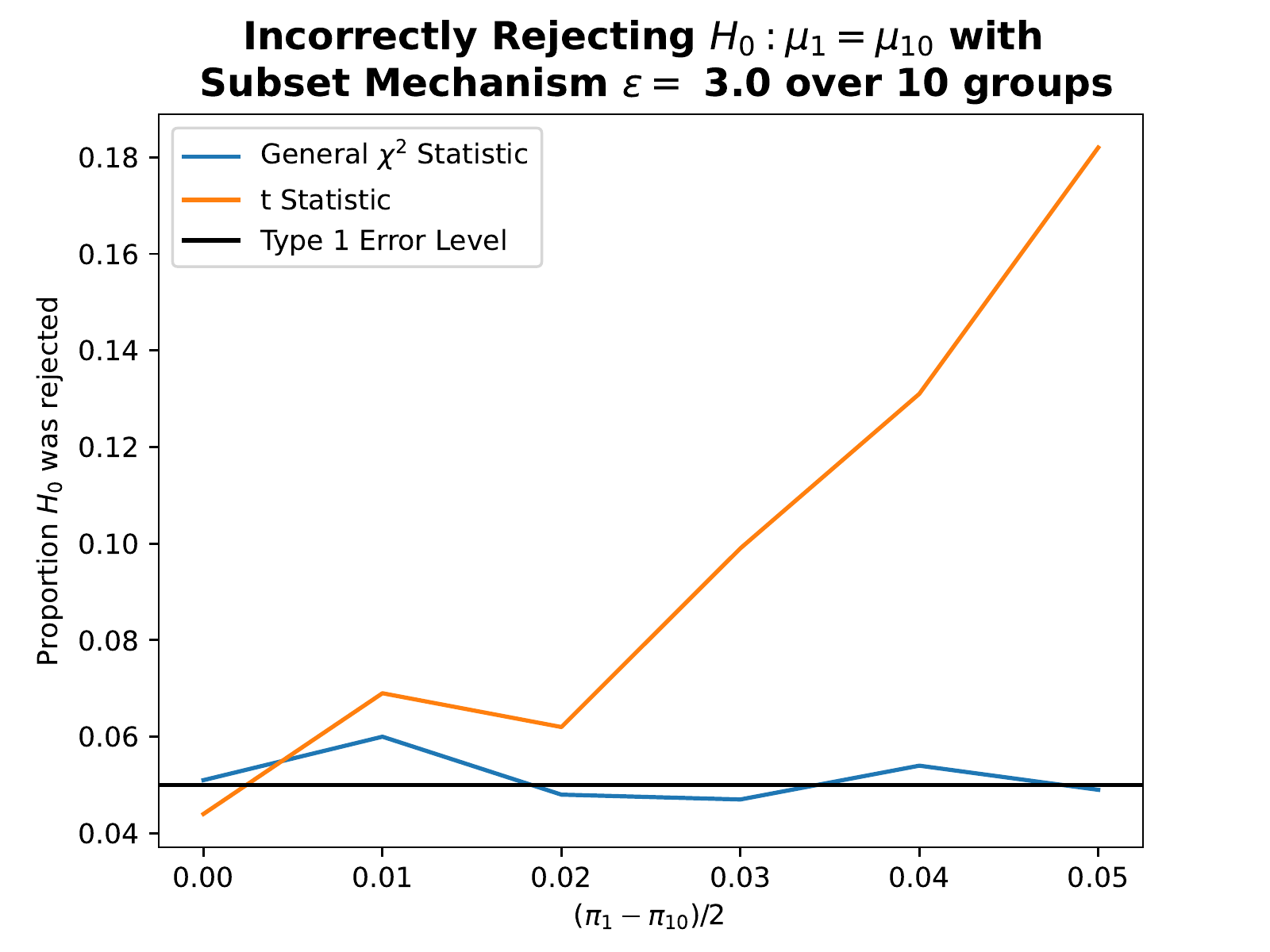}
    \includegraphics[width=0.31\linewidth]{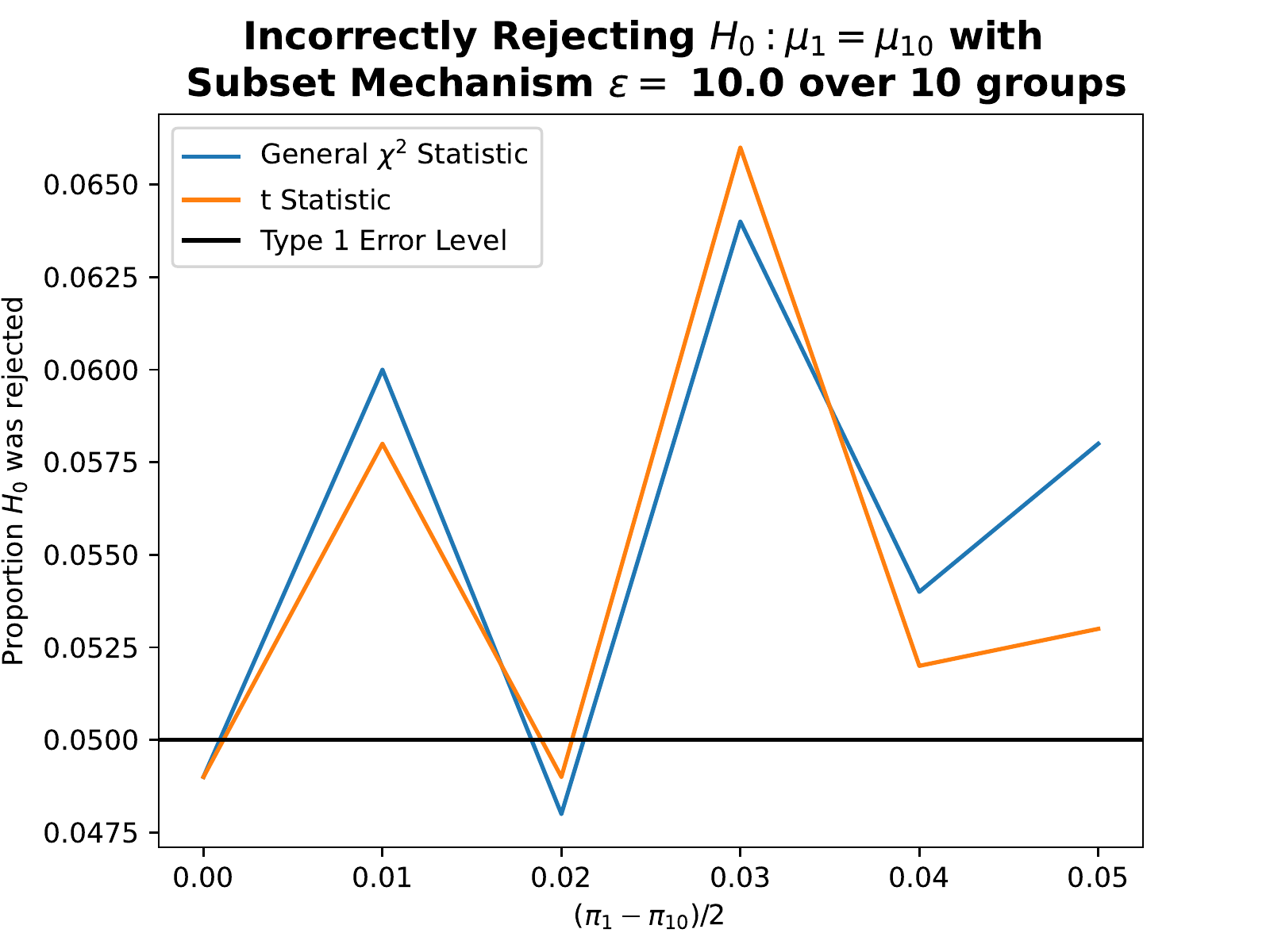}
  \caption{We give the proportion of times in 1000 trials that the classical t-test incorrectly rejects $H_0: \mu_1 = \mu_{10}$ after we privatize the membership of $g= 10$ groups for each sample.  We use the same standard deviation $\sigma = 2$ across all groups and have $\mu_j = 1.0$ for all $j \neq 1,10$, $\mu_{10} = 1.5$ and $n = 10000$. We change the group probability $\pi$ from uniform across all 10 groups and then change the first and last group probabilities.}
  \label{fig:SignificanceSubset}
\end{figure}

\begin{figure}
  \includegraphics[width=0.31\linewidth]{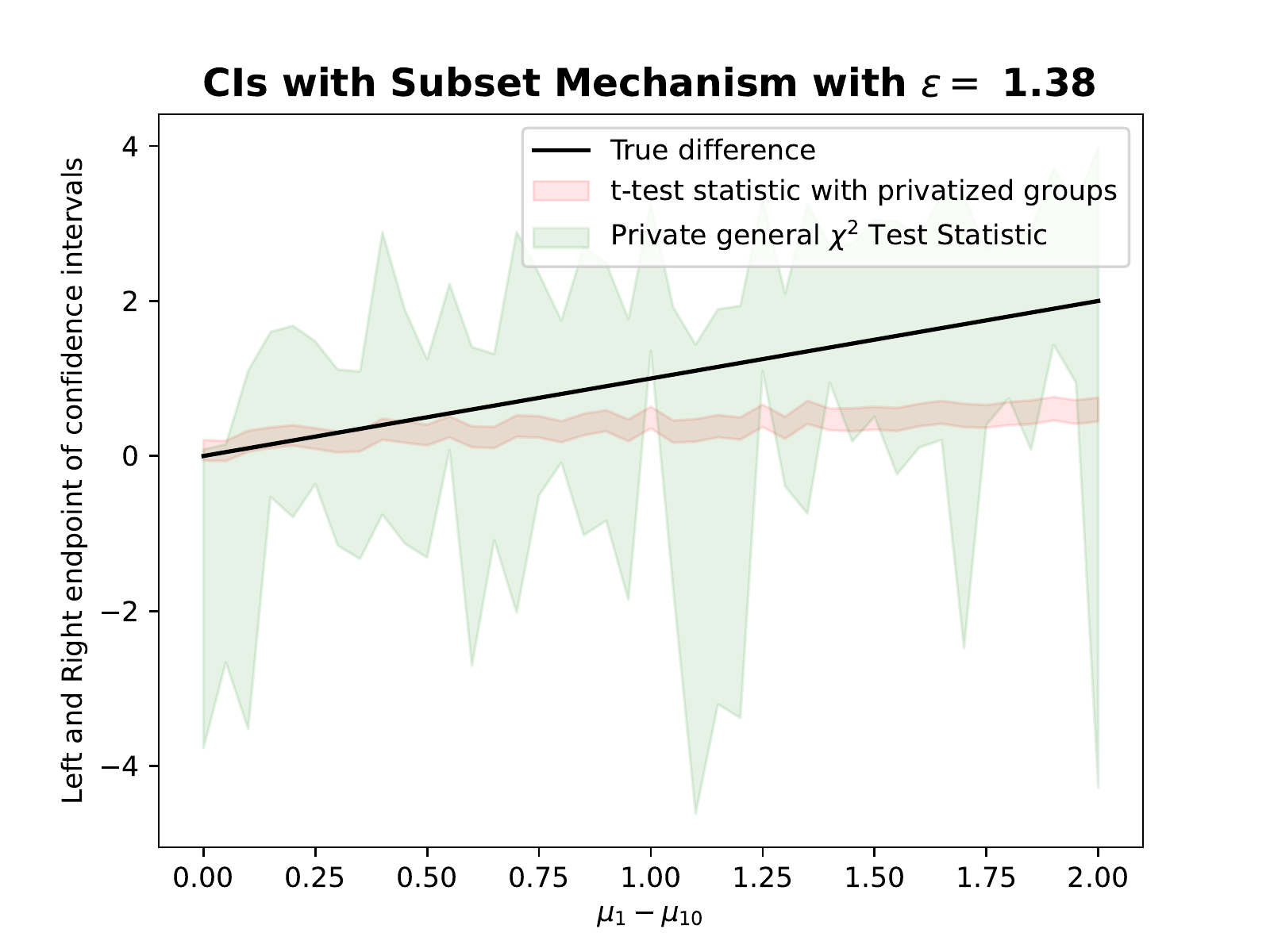}
  \includegraphics[width=0.31\linewidth]{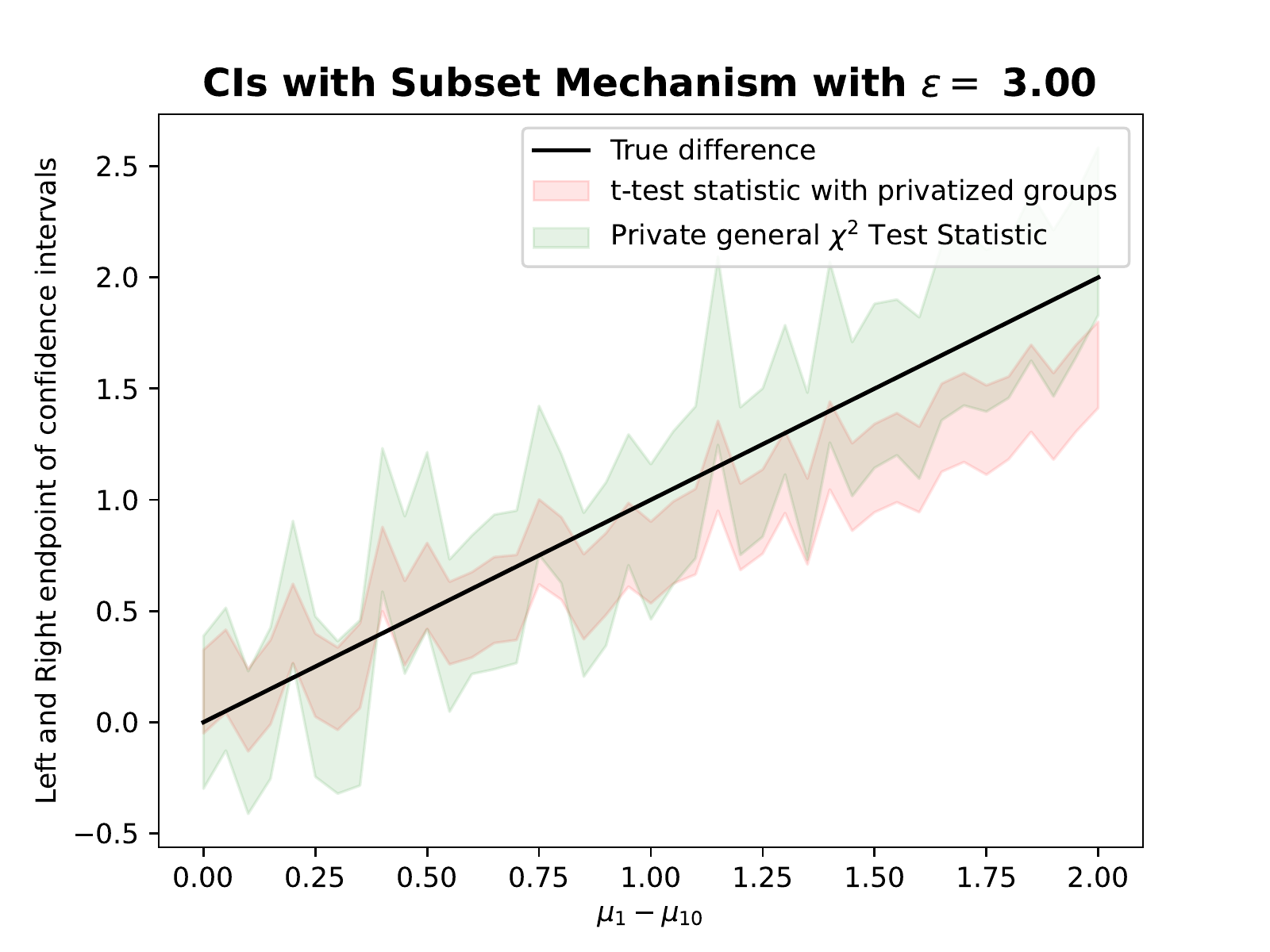}
    \includegraphics[width=0.31\linewidth]{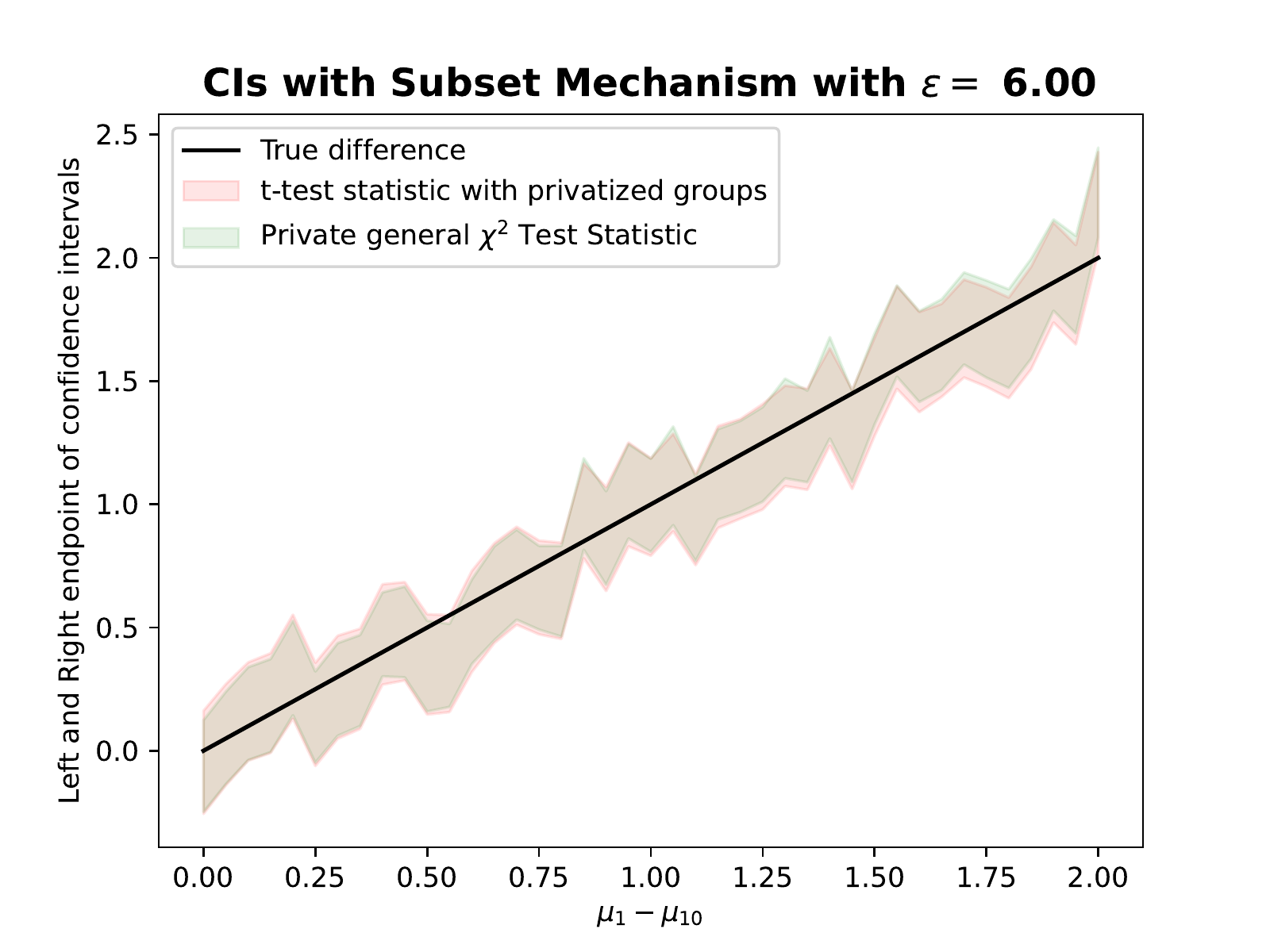}
  \caption{We give the confidence intervals between the first and 10th group mean when computed with the classical t-test and the general $\chi^2$ test after we privatize the membership of $g= 10$ groups for each sample.  We use the same standard deviation $\sigma = 2$ across all groups and have $\mu_j = 1.0$ for all $j \neq 1,10$, $\mu_{10} = 1.5$ and $n = 10000$. The group probability $\pi$ is uniform across all 10 groups except $\pi_1 = 0.15$ and $\pi_{10} = 0.05$.}
  \label{fig:tTestCIs}
\end{figure}

%% file: ABTest.tex
\section{Application to A/B Testing}

To close, we show how our general approach can be applied to the A/B testing setting, specifically to test whether the difference in means between two groups has remained the same or changed in an A/B test. Going back to our motivation, this extension has huge potential as a practical solution for differential privacy in analyzing A/B tests over sensitive groups. With this method, companies can add privacy to data describing demographic information, and then proceed to test hypotheses about different groups and outcomes under many different A/B tests, without incurring privacy loss. As companies move to ask questions about different groups' experiences on their platforms, testing without statistical bias and with high power is crucial for making the best decisions. 

For this application, we assume that samples are randomly assigned to either a treatment or a control variant in an A/B test. We will denote the random variable $T_i \sim \text{Bern}(\lambda)$ to determine whether sample $i$ is in the treatment $T_i =1$ or in the control $T_i = 0$ group.  Note that the parameter $\lambda \in [0,1]$ is known and does not need estimating.  In the treatment set of samples, data in group $j \in \{1, 2\}$ will follow $X_{i}[j,t] \sim \Normal{\mu_{j,t}}{\sigma_{j,t}^2}$, while in the control set of samples, data in group $j \in \{1,2\}$ will follow $X_{i}[j,c] \sim \Normal{\mu_{j,c}}{\sigma_{j,c}^2}$.  Again, we will let $W_i \sim \text{Bern}(\pi)$ determine the group that sample $i$ belongs to, i.e. $W_i = 0$ for group 1 and $W_i = 1$ for group 2.  Our goal here is to test whether the differences in means in the two groups has changed between the treatment and control.  That is, we test $H_0: \mu_{1,t} - \mu_{2,t} = \mu_{1,c} - \mu_{2,c}$.  In order to compute confidence intervals, we will include a $\Delta$ term in the difference, so that we test $H_0: \mu_{1,t} - \mu_{2,t} = \mu_{1,c} - \mu_{2,c}$, but $\Delta = 0$ is the typical hypothesis test.

We are not considering the membership of a sample to the control or treatment to be sensitive, and hence not privatizing it.  Instead, we privatize the group membership $j \in \{1,2\}$ for each sample using randomized response.  Recall that for randomized response, we will use random variables $Z_i^\diffp[j,j] \sim \text{Bern}(\tfrac{e^\diffp}{e^\diffp +1})$ for $j \in \{ 1, 2\}$ with $Z^\diffp_{i}[1,2] = 1- Z^\diffp_i[1,1]$ and $Z^\diffp_{i}[2,1] =1 - Z^\diffp_{i}[2,2]$.  We will consider the random vector $Y^\diffp = \sum_{i=1}^n Y_i^\diffp$ in our general $\chi^2$ test framework, where  
\[
Y_i^\diffp = 
\begin{pmatrix}
T_i \cdot (Z_i^\diffp[1,1] \cdot W_i + Z_i^\diffp[1,2] \cdot(1 - W_i) ) + (1-T_i) \cdot (Z_i^\diffp[1,1] \cdot W_i+ Z_{i}^\diffp[1,2] \cdot (1 - W_i)) \\
T_i  \cdot (Z_i^\diffp[1,1] \cdot W_i \cdot X_{i}[1,t] + Z_i^\diffp[1,2] \cdot (1 - W_i) \cdot X_{i}[2,t]) \\
T_i \cdot (Z_i^\diffp[2,1] \cdot W_i \cdot X_{i}[1,t] + Z_i^\diffp[2,2]\cdot (1 - W_i) \cdot X_{i}[2,t]) \\
(1-T_i) \cdot (Z_i^\diffp[1,1] \cdot W_i \cdot X_{i}[1,c] + Z_i^\diffp[1,2] \cdot(1 - W_i) X_{i}[2,c]) \\
(1-T_i) \cdot (Z_i^\diffp[2,1] \cdot W_i \cdot X_{i}[1,c] + Z_i^\diffp[2,2]\cdot(1 - W_i) X_{i}[2,c]) 
\end{pmatrix}
\]

Since $T_i$ is independent of the other variables, a lot of the calculations we have already done for t-tests in Section~\ref{sect:privateTTest} can be used.  Next we compute its expectation in terms of the population parameters, where we will write $\vec{\mu} = (\mu_{1,t}, \mu_{2,t}, \mu_{1,c}, \mu_{2,c})$,
\[
\vec{\theta}^\diffp(\pi,\vec{\mu}; \diffp, \lambda) = \E[Y_i^\diffp] = 
\begin{pmatrix}  
\tfrac{e^\diffp}{e^\diffp + 1}  \pi  +  \tfrac{1}{e^\diffp + 1} \left(1-\pi\right) \\
\lambda \left( \tfrac{e^\diffp}{e^\diffp + 1}  \pi \mu_{1,t} +  \tfrac{1}{e^\diffp + 1} \left(1-\pi\right) \mu_{2,t} \right)\\
\lambda\left( \tfrac{1}{e^\diffp + 1}  \pi \mu_{1,t} +  \tfrac{e^\diffp}{e^\diffp + 1} \left(1-\pi\right) \mu_{2,t} \right) \\
(1-\lambda) \left( \tfrac{e^\diffp}{e^\diffp + 1}  \pi \mu_{1,c} +  \tfrac{1}{e^\diffp + 1} \left(1-\pi\right) \mu_{2,c} \right)\\
(1-\lambda) \left( \tfrac{1}{e^\diffp + 1}  \pi \mu_{1,c} +  \tfrac{e^\diffp}{e^\diffp + 1} \left(1-\pi\right) \mu_{2,c} \right)
\end{pmatrix}
\]

We next compute the covariance matrix $C(\pi, \vec{\mu}, \vec{\sigma}^2; \diffp, \lambda) =  \E\left[ Y_i^\diffp (Y_i^\diffp)^\intercal \right] - \E\left[ Y_i^\diffp \right] \E\left[ Y_i^\diffp \right]^\intercal $, where $\vec{\sigma}^2 = (\sigma_{1,t}^2, \sigma_{2,t}^2, \sigma_{1,c}^2, \sigma_{2,c}^2)$. We have
\begin{align*}
\E\left[ Y_i^\diffp (Y_i^\diffp)^\intercal \right] [1,1] &= \E\left[ Y_i^\diffp\right][1] \\
\E\left[ Y_i^\diffp (Y_i^\diffp)^\intercal \right] [1,2] & = \lambda \left( \pi \tfrac{e^\diffp}{e^\diffp + 1} \mu_{1,t} + (1-\pi) \tfrac{1}{e^\diffp + 1} \mu_{2,t} \right) \\
\E\left[ Y_i^\diffp (Y_i^\diffp)^\intercal \right] [1,3] & = \E\left[ Y_i^\diffp (Y_i^\diffp)^\intercal \right] [1,5] = 0 \\
\E\left[ Y_i^\diffp (Y_i^\diffp)^\intercal \right] [1,4] & = (1-\lambda) \left( \pi \tfrac{1}{e^\diffp + 1} \mu_{1,c} + (1-\pi) \tfrac{e^\diffp}{e^\diffp + 1} \mu_{2,c} \right) \\
\E\left[ Y_i^\diffp (Y_i^\diffp)^\intercal \right] [2,2] & = \lambda \left( \pi \tfrac{e^\diffp}{e^\diffp + 1} (\mu_{1,t}^2 + \sigma_{1,t}^2) + (1-\pi) \tfrac{1}{e^\diffp + 1} (\mu_{2,t}^2 + \sigma_{2,t}^2) \right) \\
\E\left[ Y_i^\diffp (Y_i^\diffp)^\intercal \right] [3,3] & = \lambda \left( \pi \tfrac{1}{e^\diffp + 1} (\mu_{1,t}^2 + \sigma_{1,t}^2) + (1-\pi) \tfrac{e^\diffp}{e^\diffp + 1} (\mu_{2,t}^2 + \sigma_{2,t}^2) \right) \\
\E\left[ Y_i^\diffp (Y_i^\diffp)^\intercal \right] [4,4] & = (1-\lambda) \left( \pi \tfrac{e^\diffp}{e^\diffp + 1} (\mu_{1,c}^2 + \sigma_{1,c}^2) + (1-\pi) \tfrac{1}{e^\diffp + 1} (\mu_{2,c}^2 + \sigma_{2,c}^2) \right) \\
\E\left[ Y_i^\diffp (Y_i^\diffp)^\intercal \right] [5,5] & = (1-\lambda) \left( \pi \tfrac{1}{e^\diffp + 1} (\mu_{1,c}^2 + \sigma_{1,c}^2) + (1-\pi) \tfrac{e^\diffp}{e^\diffp + 1} (\mu_{2,c}^2 + \sigma_{2,c}^2) \right) \\
\E\left[ Y_i^\diffp (Y_i^\diffp)^\intercal \right] [j,\ell] & = 0, \quad j, \ell\in \{2,3,4,5 \}, j \neq \ell.
\end{align*}
Under the null hypothesis $H_0: \mu_{1,t} - \mu_{2,t} = \mu_{1,c} - \mu_{2,c} + \Delta$ we can solve for one of the means, so we will set $\mu_{1,t} = \mu_{1,c} - \mu_{2,c} + \mu_{2,t} + \Delta$ to reduce the number of parameters (note that we treat $\Delta$ as known).  We now want to use our sample $Y^\diffp$ to estimate the other means and the group probability $\hat{\pi}$.  We start with the group probability, which we have estimated the same way in other sections:
\[
\hat{\pi} = (e^\diffp + 1) \left( \frac{Y^\diffp[0]/n - \tfrac{1}{e^\diffp + 1}}{e^\diffp - 1} \right).
\]

We then solve for estimates of the means $\hat{\mu}_{1,t}, \hat{\mu}_{2,t}, \hat{\mu}_{1,c}, \hat{\mu}_{2,c}$ by setting the empirical averages $(\sum_{i=1}^nY_i[j] /n : j \in \{2,3, 4, 5\})$ equal to the respective coordinates of $\vec{\theta}^\diffp(\hat{\pi}, \vec{\mu}; \diffp, \lambda)$ and solve for the means.  Note that when we substitute in the null hypothesis $\mu_{1,t} = \mu_{2,t} + \mu_{1,c} - \mu_{2, c} + \Delta$, we get 4 equations and 3 unknowns.  In this case, we choose the first three equations to solve for $\mu_{2,t}, \mu_{1,c}, \mu_{2, c}$ and then set $\hat{\mu}_{1,t} = \hat{\mu}_{2,t} + \hat{\mu}_{1,c} - \hat{\mu}_{2, c} + \Delta$.
From these estimates, we can plug them into the covariance matrix.  Note that we will not directly form estimates for $\vec{\sigma}^2$, instead we will compute the sample variance for $\{Y_i^\diffp[j]\}_{i=1}^n$ for $j \in \{2, 3, 4, 5 \}$ to use on the main diagonal of the covariance matrix and check to see that they indeed give valid sample variances (i.e. have positive variance of the true data), as we did for the test in Section~\ref{sect:privateTTest}.  

Our test statistic then becomes the following, where we will write $\hat{C}$ to denote the covariance matrix with the above parameter estimates and given $\diffp, \lambda, \Delta$,
\[
D^\diffp(\lambda, \Delta) = 
 \min_{\substack{\pi \in (0,1), \\ \mu_{2,t}, \mu_{1,c}, \mu_{2,c} \in \R \\ \mu_{1,t} = \mu_{1,c} - \mu_{2,c} + \mu_{2,t} + \Delta}} \left\{ 
\left(  Y^\diffp  - \vec{\theta}^\diffp\left(\pi,  \vec{\mu}; \diffp, \lambda \right) \right)^\intercal 
 \hat{C}^{-1}
\left( 
Y^\diffp  - \vec{\theta}^\diffp\left(\pi,\vec{\mu}; \diffp, \lambda \right)
\right)
 \right\}.
\]
We then compare the test statistic with a $\chi^2$ random variable with $1$ degrees of freedom.  We give power results in Figure~\ref{fig:ABResultsPower} for hypothesis tests comparing the general $\chi^2$ test with the classical t-test on privatized groups, which seems to perform similarly.  We also give confidence intervals for the difference between means across treatment and control in Figure~\ref{fig:ABResultsCI} using the general $\chi^2$ test statistic and the unmodified t-test statistic.  In our experimental setup we generate data with zero means across groups 1 and 2 in both treatment and control and keep the variance across all to be 1.  We then will vary the mean in the control group of group $1$, i.e. $\mu_{1,c}$ to change the difference between groups in the treatment and control.  Observe that the t-test does not cover the true difference in difference in the confidence intervals, as compared to the general $\chi^2$ test, which does not suffer from the same downward bias.

\begin{figure}
  \includegraphics[width=0.48\linewidth]{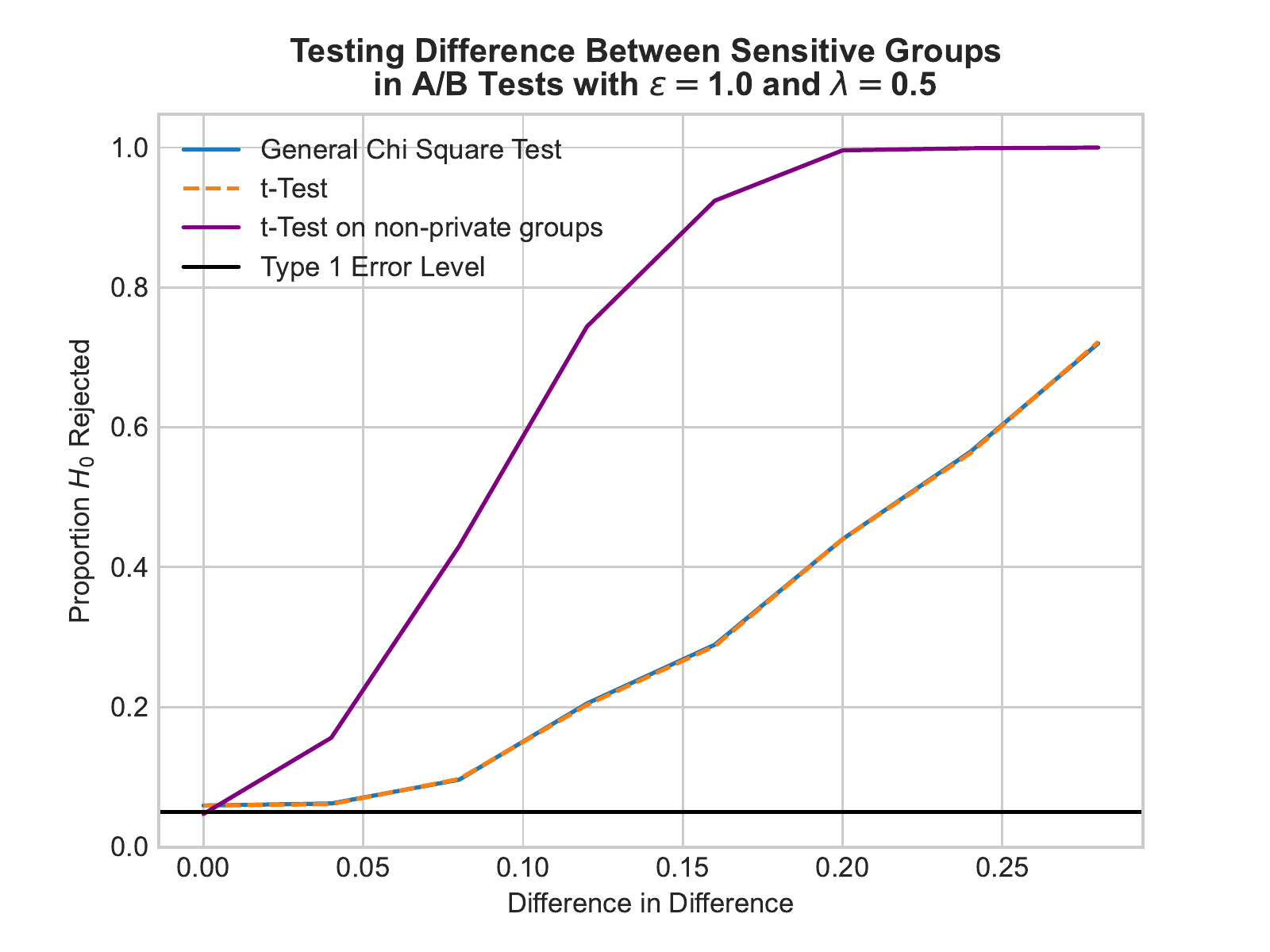}
  \includegraphics[width=0.48\linewidth]{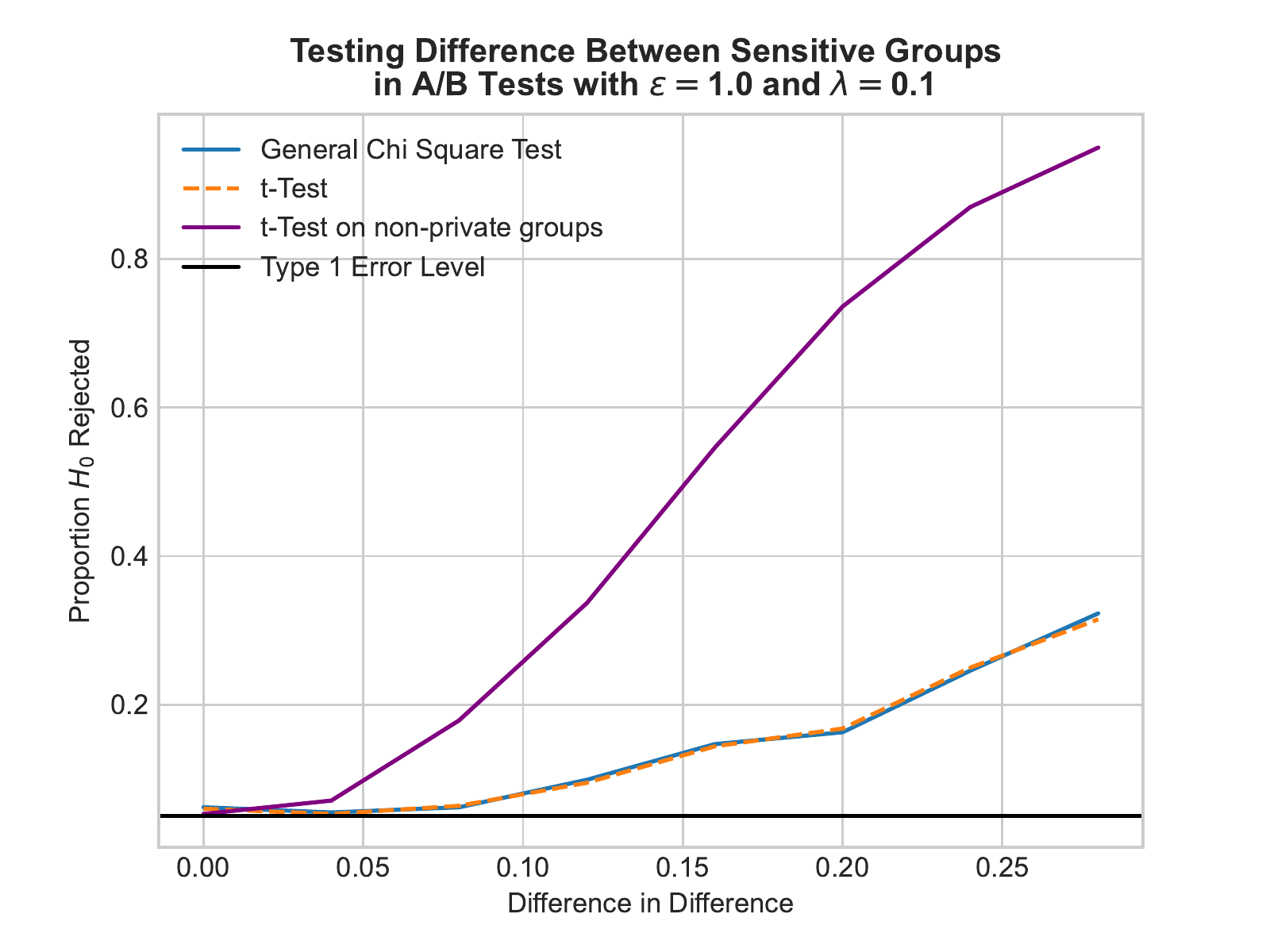}
  \caption{Power results for hypothesis testing for the difference in means across sensitive groups between treatment and control $H_0: \mu_{1,t} - \mu_{2,t} = \mu_{1,c} - \mu_{2,c}$.  We compare the (unmodified) t-test and the general $\chi^2$ test on privatized groups with treatment probability $\lambda \in \{0.5, 0.1 \}$ and $\diffp = 1$.}
  \label{fig:ABResultsPower}
\end{figure}

\begin{figure}
  \includegraphics[width=0.48\linewidth]{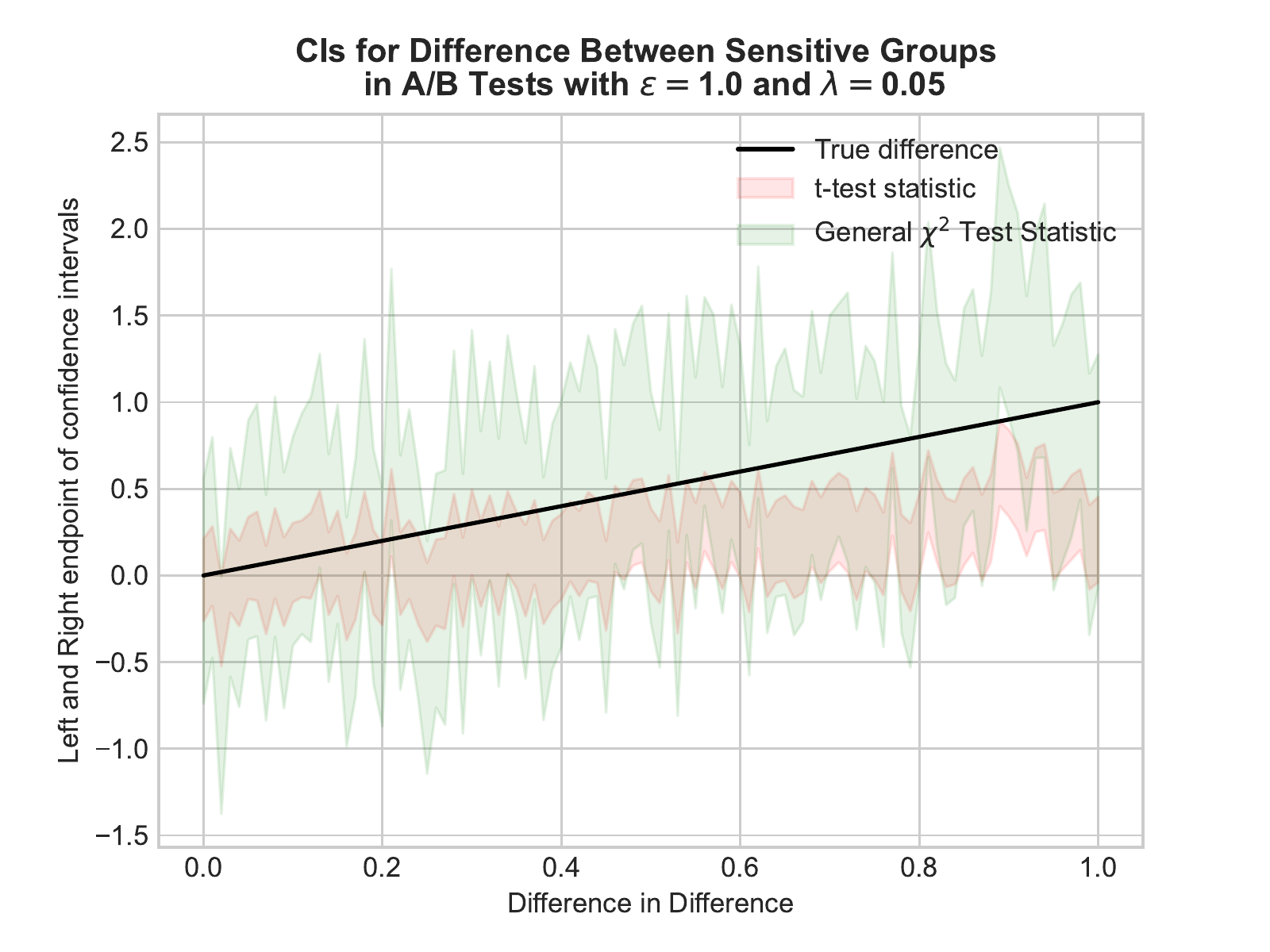}
  \includegraphics[width=0.48\linewidth]{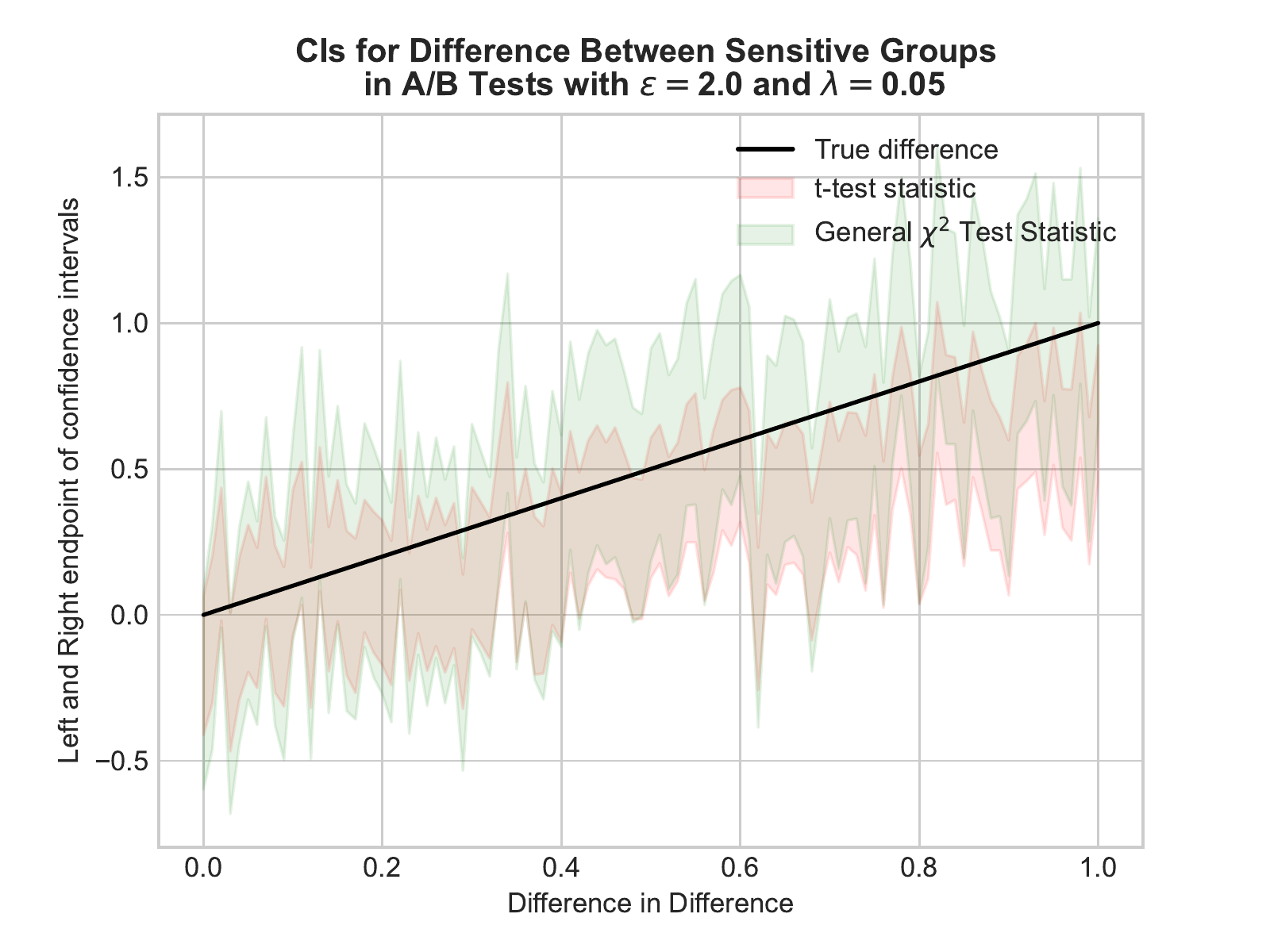}
  \caption{Confidence intervals for the difference in means across sensitive groups between treatment and control $H_0: \mu_{1,t} - \mu_{2,t} = \mu_{1,c} - \mu_{2,c} + \Delta$.  We compare the (unmodified) t-test and the general $\chi^2$ test on privatized groups with treatment probability $\lambda = 0.05$ and $\diffp \in \{ 1, 2\}$.}
  \label{fig:ABResultsCI}
\end{figure}

%% file: conclusion.tex
\section{Conclusion}
We have covered several tests for determining if there are differences in outcomes across sensitive groups.  We introduced the local group DP definition as a less restrictive privacy model than the local DP model, while ensuring the group that each member belongs to remains private no matter the number of tests that are conducted.  We first considered the binary outcome setting, and compared our less restrictive privacy model with local DP $\chi^2$ independence tests for multiple groups and showed a significant improvement in power.  We then extended the general $\chi^2$ framework to test differences in means, where (fully) local DP may cause noise to swamp the signal for even reasonable levels of privacy.  We then provided tests for one-way ANOVA as well as an application to A/B Testing to determine if the difference between groups has changed from the control to the treatment.  Our results show that these general $\chi^2$ tests can also be used to compute valid confidence intervals for the true difference in proportions and means even when group membership is privatized.  This is in contrast to using the traditional $Z$-test or t-test to compute confidence intervals.  Furthermore, the power of the tests over multiple groups can significantly benefit by using the subset mechanism and the corresponding $\chi^2$ test, rather than the traditional tests that do not account for privacy.   

We also mention some possible extensions and limitations.  One extension is to apply the general $\chi^2$ theory to other tests in the global DP model, using the sample moments contingency table from Table~\ref{table:momentContingencyANOVA}.  Such tests were developed in \cite{KiferRo17} but only for categorical data.  With recent significant advances in improving power of DP ANOVA tests using non-parametric rank based tests, we do not expect that adding noise to the contingency tables for global DP and then applying the general $\chi^2$ tests would outperform existing DP tests for ANOVA from \cite{CouchKaShBrGr19}.  However, there are typically \emph{post-hoc} tests conducted after ANOVA has determined there is a significant difference between the groups.  The benefit with privatizing the contingency table is that these post-hoc tests could be conducted without increasing the privacy loss.  Depending on the number of post-hoc tests to be conducted, it may still be beneficial to privatize the full contingency table and run statistical tests directly on it.  We see this as an interesting direction of future work.  

A limitation of this work is that if a new group is to be added to a set of existing groups, increasing the total number of groups from $g$ to $g'$, then it is not clear how to take samples that are privatized over the smaller set of $g$ groups.  The question then becomes, how do we design statistical tests that combines datasets where some samples are privatized over $g$ groups while others are privatized over $g' > g$ groups?  Answering these questions will help researchers to design tests that solve practical problems in hypothesis testing under DP, empowering companies to answer key questions about user experience while protecting user privacy.